\pdfoutput=1
\documentclass[reqno]{amsart}

\usepackage[utf8]{inputenc}
\usepackage[english]{babel}

\usepackage{braket}
\usepackage{amsmath}
\usepackage{amstext}
\usepackage{amssymb}
\usepackage{amsthm}
\usepackage{mathtools}
\usepackage{stmaryrd}
\usepackage{mathrsfs}
\usepackage{extarrows}
\usepackage{afterpage}
\usepackage{tabu}
\usepackage{array}
\usepackage[foot]{amsaddr}

\usepackage{microtype}

\usepackage{graphicx}
\usepackage{amsfonts}

\usepackage{units}
\usepackage[shortlabels]{enumitem}
\usepackage{url}

\usepackage{bm}

\usepackage{tikz}
\usetikzlibrary{matrix,arrows,calc,fit,cd,positioning,intersections,arrows.meta}

\usepackage{hyperref}
\usepackage{cleveref}

\usepackage[margin=1cm]{caption}

\usepackage{colortbl}

\allowdisplaybreaks[4]

\theoremstyle{definition}
\newtheorem{theorem}{Theorem}[section]
\newtheorem{definition}[theorem]{Definition}
\newtheorem{remark}[theorem]{Remark}
\newtheorem{example}[theorem]{Example}
\newtheorem{lemma}[theorem]{Lemma}
\newtheorem{proposition}[theorem]{Proposition}
\newtheorem{corollary}[theorem]{Corollary}
\newtheorem{conjecture}[theorem]{Conjecture}

\newenvironment{provedcorollary}
{\pushQED{\qed}\begin{corollary}}
	{\popQED\end{corollary}}

\newtheorem{innercustomthm}{Theorem}
\newenvironment{customthm}[1]
{\renewcommand\theinnercustomthm{#1}\innercustomthm}
{\endinnercustomthm}

\newtheorem*{mainthm}{Main Theorem}

\newcommand{\N}{\mathbb{N}}
\newcommand{\Z}{\mathbb{Z}}
\newcommand{\R}{\mathbb{R}}
\newcommand{\C}{\mathbb{C}}

\newcommand{\M}{\mathcal{M}}

\newcommand{\F}{\mathcal{F}}
\newcommand{\A}{\mathcal{A}}
\newcommand{\E}{\mathcal{E}}
\newcommand{\CC}{\mathcal{C}}

\DeclareMathOperator{\codim}{codim}

\DeclareMathOperator{\rk}{rk}

\newcommand{\<}{\langle}
\renewcommand{\>}{\rangle}

\DeclareMathOperator{\Mov}{\textsc{Mov}}
\DeclareMathOperator{\Fix}{\textsc{Fix}}
\DeclareMathOperator{\Min}{\textsc{Min}}

\DeclareMathOperator{\Span}{\textsc{Span}}
\DeclareMathOperator{\Dir}{\textsc{Dir}}

\DeclareMathOperator{\inv}{inv}

\DeclareMathOperator{\Isom}{\textsc{Isom}}

\newcommand{\p}{p}

\renewcommand{\a}{a}
\renewcommand{\b}{b}

\newcommand{\h}{{\text{hor}}}
\renewcommand{\v}{{\text{ver}}}

\makeatletter
\newcases{rightcases}{\quad}{%
	$\m@th\displaystyle{##}$\hfil}{\hfil$\m@th\displaystyle{##}$}{\lbrace}{.}
\makeatother

\usepackage{IEEEtrantools}

\title[Proof of the $K(\pi, 1)$ conjecture for affine Artin groups]{Proof of the $K(\pi, 1)$ conjecture\\ for affine Artin groups}

\author[Giovanni Paolini]{Giovanni Paolini}
\address{\textnormal{Giovanni Paolini: Department of Mathematics, University of Fribourg, Switzerland (currently at California Institute of Technology, Pasadena, California, USA).
}}

\author[Mario Salvetti]{Mario Salvetti}
\address{\textnormal{Mario Salvetti: Department of Mathematics, University of Pisa, Italy.
	}}

\email{paolini@caltech.edu\textnormal{,} salvetti@dm.unipi.it}

\thanks{\vskip0.01cm We are grateful to Pierre Deligne for his remarks and suggestions on the first version of this paper.
We are also grateful to Emanuele Delucchi and Alessandro Iraci for the useful discussions, and to the anonymous referee for the helpful comments.
A preliminary version of \Cref{sec:dual-salvetti-complex,sec:classifying-spaces,sec:finite-classifying-spaces} is part of the first author's Ph.D.\ thesis at Scuola Normale Superiore \cite{paolini2019thesis}, written under the supervision of the second author.
This work was also supported by the Swiss National Science
Foundation Professorship grant PP00P2\_179110/1, by Ministero dell'Istruzione, dell'Universit\`a e della Ricerca, Prog.\ PRIN 2017YRA3LK\_005, \emph{Moduli and Lie Theory}
and by the University of Pisa, Prog.\ PRA\_2018\_22, \emph{Geometria e topologia delle variet\`a}.
\vskip0.3cm
The final publication in \emph{Inventiones mathematicae} is available at \url{https://doi.org/10.1007/s00222-020-01016-y}.}

\begin{document}

\begin{abstract}
	We prove the $K(\pi,1)$ conjecture for affine Artin groups: the complexified complement of an affine reflection arrangement is a classifying space.
	This is a long-standing problem, due to Arnol'd, Pham, and Thom.
	Our proof is based on recent advancements in the theory of dual Coxeter and Artin groups, as well as on several new results and constructions.
	In particular: we show that all affine noncrossing partition posets are EL-shellable;
	we use these posets to construct finite classifying spaces for dual affine Artin groups;
	we introduce new CW models for the orbit configuration spaces associated with arbitrary Coxeter groups; we construct finite classifying spaces for the braided crystallographic groups introduced by McCammond and Sulway.
	
\end{abstract}

\maketitle


\section{Introduction}
\label{sec:introduction}

The long-standing $K(\pi,1)$ conjecture for Artin groups states that the orbit configuration space $Y_W$ associated with a Coxeter group $W$ is always a $K(G_W,1)$ space.
Here $G_W$ is the fundamental group of $Y_W$ and is known as the Artin group associated with $W$.
The conjecture was proved for spherical Artin groups (i.e.\ if $W$ is finite) by Deligne \cite{deligne1972immeubles}, after being proved by Fox and Neuwirth in the case $A_n$ \cite{fox1962braid} and by Brieskorn in the cases $C_n$, $D_n$, $G_2$, $F_4$, and $I_2(p)$ \cite{brieskorn1973groupes}.

In this paper, we prove the conjecture for the next important family of Artin groups, namely for all affine Artin groups.
Together with Deligne's result, this covers all the cases where $W$ is a Euclidean reflection group.

\begin{mainthm}[\Cref{thm:conjecture}]
	The $K(\pi,1)$ conjecture holds for all affine Artin groups.
\end{mainthm}

The $K(\pi,1)$ conjecture goes back to the pioneering work of Arnol'd, Brieskorn, Pham, and Thom in the '60s (see \cite{brieskorn1973groupes, van1983homotopy, godelle2012basic, paris2012k}).
After the fundamental contribution of Deligne, the conjecture was proved for the affine Artin groups of type $\tilde{A}_n$, $\tilde{C}_n$ \cite{okonek1979dask}, $\tilde{B}_n$ \cite{callegaro2010k}, and $\tilde G_2$ \cite{charney1995k}.
So our paper completes the list of affine Artin groups with the case $\tilde{D}_n$ and with all the remaining exceptional cases.
Unlike the proofs for previously known cases, our approach is essentially ``case-free,'' although some partial results use the classification of reflection groups.
In particular, it also applies to all previously known affine cases.

Besides Euclidean cases, the conjecture was proved for Artin groups of dimension $\leq 2$ and for those of FC type \cite{hendriks1985hyperplane, charney1995k}.
It was also extended to the configuration spaces of finite complex reflection groups and proved in full generality by Bessis \cite{bessis2015finite}.

Our results were made possible by recent advances in the theory by McCammond and Sulway \cite{mccammond2015dual, mccammond2017artin}, which rely on \emph{dual} Coxeter and Artin groups and on Garside structures (see \Cref{sec:background}).
In \cite{mccammond2017artin}, finite-dimensional classifying spaces for affine Artin groups were constructed, but with an infinite number of cells.
Our proof of the $K(\pi,1)$ conjecture leads to a significant improvement of their construction: we obtain finite classifying spaces for $G_W$ which are homotopy equivalent to the orbit configuration space $Y_W$.
In doing that, we derive many new geometric and combinatorial side results for affine Coxeter and Artin groups, which we think are interesting by themselves.

The following are some consequences of the $K(\pi,1)$ conjecture:
affine Artin groups are torsion-free (this already follows from the construction of McCammond and Sulway \cite{mccammond2017artin});
they have a classifying space with a finite number of cells (see \cite{salvetti1994homotopy});
the well studied homology and cohomology of $Y_W$ coincides with the homology and cohomology of the corresponding affine Artin group $G_W$ (see \cite{callegaro2008cohomology, callegaro2008cohomology2, callegaro2010k, salvetti2013combinatorial, paolini2018weighted, paolini2019local});
the natural map between the classifying space of an affine Artin monoid and the classifying space of the corresponding Artin group is a homotopy equivalence (see \cite{dobrinskaya2002arnold, dobrinskaya2006configuration, ozornova2017discrete, paolini2017classifying}).

\subsection{Outline of the proof and future research directions}

For every Coxeter group $W$ and Coxeter element $w$, there is an associated \emph{dual Artin group} $W_w$.
It is known to be naturally isomorphic to the standard Artin group $G_W$, if $W$ is finite \cite{bessis2003dual} or affine \cite{mccammond2017artin}.
When the noncrossing partition poset $[1,w]$ is a lattice, the dual Artin group is a Garside group, and it admits a standard construction of a classifying space $K_W \simeq K(W_w, 1)$ \cite{charney2004bestvina}.
The poset $[1,w]$ is indeed a lattice if $W$ is finite \cite{brady2001partial, brady2002k, bessis2003dual, brady2008non}, but this is not always the case if $W$ is affine \cite{digne2006presentations, digne2012garside, mccammond2015dual}.

In our proof of the $K(\pi,1)$ conjecture, one of the key points is to show that $K_W$ is a classifying space for $W_w$, for every affine Coxeter group $W$, even when $[1,w]$ is not a lattice (\Cref{thm:KW-classifying}).
This can come as a surprise since the standard argument to show that $K_W$ is a classifying space heavily relies on the lattice property.

Then we show that $K_W$ is homotopy equivalent to the orbit configuration space $Y_W$.
For this, we introduce a new family of CW models $X_W' \simeq Y_W$, which are subcomplexes of $K_W$ (\Cref{def:dual-salvetti}).
Differently from the usual models (such as the Salvetti complex \cite{salvetti1994homotopy}), the structure of $X_W'$ depends on the dual Artin relations in $W_w$ rather than on the standard Artin relations in $G_W$.
Using discrete Morse theory, we prove that $K_W$ deformation retracts onto $X_W'$.
This completes the proof of the $K(\pi,1)$ conjecture, and at the same time, it gives a new proof that the dual Artin group $W_w$ is naturally isomorphic to the Artin group $G_W$ (in the affine case).

The outlined program passes through several intermediate geometric, combinatorial, and topological results.
For example, an important step in the proof of the deformation retraction $K_W \simeq X_W'$ is to construct an EL-labeling of the affine noncrossing partition poset $[1,w]$.
This and other contributions are summarized in \Cref{sec:contributions} below.

Of course, one can hope to use an analog strategy to solve the $K(\pi,1)$ conjecture in the general case.
However, in order to carry out such program, it seems that a general geometric theory of dual Coxeter groups is required, as well as a combinatorial theory of noncrossing partition posets associated with arbitrary Coxeter groups (for example: are these posets EL-shellable? how can they fail in being lattices?), and perhaps also new developments in Garside theory (can the lattice condition be relaxed?).
These are interesting and potentially promising directions for future research.

\subsection{Summary of additional contributions}
\label{sec:contributions}


As mentioned above, in this paper we prove several results in addition to the $K(\pi,1)$ conjecture.
Here we list the ones we consider to be the most important and of independent interest.

In \Cref{sec:affine-coxeter-elements}, we expand the geometric theory of Coxeter elements in affine Coxeter groups, continuing the work started by McCammond and Sulway \cite{mccammond2015dual, mccammond2017artin}.
The following is one of our many results. Its analog for finite Coxeter groups was proved by Bessis \cite{bessis2003dual}.

\begin{customthm}{A}[\Cref{thm:coxeter-elements}]
	\label{thm:A}
	Every element $u$ in an affine noncrossing partition poset $[1,w]$ is a Coxeter element of the Coxeter subgroup generated by the subposet $[1,u]$.
\end{customthm}

The next result sheds some light on the combinatorial structure of affine noncrossing partition posets and is the focus of \Cref{sec:shellability}.
Its analog for finite noncrossing partition lattices was proved by Athanasiadis, Brady, and Watt \cite{athanasiadis2007shellability}.

\begin{customthm}{B}[\Cref{thm:shellability}]
	All affine noncrossing partition posets are EL-shellable.
\end{customthm}

We should emphasize that the affine setting differs substantially from the finite setting.
For example:
an affine noncrossing partition poset $[1,w]$ is infinite and not necessarily a lattice;
not all elements of $[1,w]$ are \emph{parabolic} Coxeter elements;
not all reflections belong to $[1,w]$.
For this reason, the theory requires significant novelties in addition to the well-established results for finite Coxeter groups.

In \Cref{sec:dual-salvetti-complex} we introduce new CW models $X_W'$ for the orbit configuration space $Y_W$, by gluing together classifying spaces $K_{W_T} \simeq K(G_{W_T}, 1)$ of spherical parabolic subgroups.
This is done in full generality, for an arbitrary Coxeter group $W$.

\begin{customthm}{C}[\Cref{thm:dual-salvetti}]
	For every Coxeter group $W$, the subcomplexes $X_W' \subseteq K_W$ are naturally homotopy equivalent to the orbit configuration space $Y_W$.
\end{customthm}

In \Cref{sec:classifying-spaces} we show that $K_W$ is a classifying space, even when $[1,w]$ is not a lattice.
Our proof makes use of the construction of \emph{braided crystallographic groups} by McCammond and Sulway \cite{mccammond2017artin}, a ``Garside completion'' of dual affine Artin groups.

\begin{customthm}{D}[\Cref{thm:KW-classifying}]
	For every affine Coxeter group $W$ and Coxeter element $w \in W$, the complex $K_W$ is a classifying space for the dual Artin group $W_w$.
\end{customthm}

In \Cref{sec:finite-classifying-spaces} we show that $K_W$ deformation retracts onto a subcomplex with a finite number of cells.
Without additional effort, this argument also applies to the classifying space of a braided crystallographic group.
We obtain the following consequence.

\begin{customthm}{E}[\Cref{thm:finite-crystallographic}]
	Every braided crystallographic group has a classifying space with a finite number of cells.
\end{customthm}

It is possible that the techniques of \Cref{sec:conjecture} can be adjusted to braided crystallographic groups, to obtain a smaller classifying space with some interesting geometric interpretation (and perhaps prove a crystallographic version of the $K(\pi,1)$ conjecture).
This might be part of some bigger picture, where every (dual) Artin group has a crystallographic Garside completion, and their classifying spaces are geometrically related.

\subsection{Structure of this paper}

In \Cref{sec:background} we recall the most important background definitions and results that are needed in the rest of the paper.
In \Cref{sec:affine-coxeter-elements} we prove several geometric results about Coxeter elements in affine Coxeter groups, expanding the theory of \cite{mccammond2015dual, mccammond2017artin}.
This section goes hand in hand with Appendix \ref{appendix}, where we carry out explicit computations for the four infinite families of irreducible affine Coxeter groups.
\Cref{sec:shellability,sec:dual-salvetti-complex,sec:classifying-spaces,sec:finite-classifying-spaces} are mostly independent from each other.
They cover separate intermediate steps of our proof of the $K(\pi,1)$ conjecture, as described earlier.
Finally, in \Cref{sec:conjecture}, everything is put together to prove the $K(\pi,1)$ conjecture.


\section{Background}
\label{sec:background}

\subsection{Coxeter groups and Artin groups}
\label{sec:coxeter-artin}

Let $W$ be a Coxeter group, i.e.\ a group with a presentation of the following form:
\begin{equation}
	\label{eq:coxeter-presentation}
 	W = \< S \mid (st)^{m(s,t)} = 1 \;\; \forall\, s,t \in S \text{ such that $m(s,t) \neq \infty$} \>,
\end{equation}
where $S$ is a finite set, $m(s,s) = 1$ for all $s \in S$, and $m(s,t) = m(t,s) \in \{2,3,4,\dotsc\} \cup \{\infty\}$ for all $s \neq t$ in $S$.
This presentation can be encoded into a \emph{Coxeter graph}: the vertices are indexed by $S$, and there is an edge connecting $s$ and $t$ whenever $m(s,t) \geq 3$; this edge is labeled by $m(s,t)$ whenever $m(s,t) \geq 4$.
A Coxeter group is \emph{irreducible} if its Coxeter graph is connected.
Any conjugate of an element of $S$ is called a \emph{reflection}.
Denote by $R \subseteq W$ the set of reflections.
Any conjugate of the set $S$ is called a \emph{set of simple reflections} of $W$, and can be used in place of $S$ to give a presentation of $W$ of the same form as \eqref{eq:coxeter-presentation}, with an isomorphic Coxeter graph.

If $S \subseteq R$ is any set of simple reflections of $W$ (not necessarily the one used to define $W$), the product of the elements of $S$ in any order is called a \emph{Coxeter element} of $W$.
Also, for any subset $T \subseteq S$, the subgroup $W_T$ generated by $T$ is a \emph{parabolic subgroup} of $W$ (it is a Coxeter group, and $T$ is a set of simple reflections of $W_T$).
A Coxeter element of a parabolic subgroup of $W$ is called a \emph{parabolic Coxeter element} of $W$.
When a set of simple reflections $S$ is fixed, the parabolic subgroups $W_T$ with $T \subseteq S$ are called \emph{standard parabolic subgroups}.
The \emph{rank} (or \emph{dimension}) of $W$ is the largest cardinality of a subset $T \subseteq S$ such that the parabolic subgroup $W_T$ is finite.
We refer to \cite{bourbaki1968elements, humphreys1992reflection, bjorner2006combinatorics} for more background information on Coxeter groups.

We are mostly interested in the case where $W$ is a finite or affine Coxeter group, or equivalently, a (finite or affine) real reflection group.
In this case, $W$ acts faithfully by Euclidean isometries on some affine space $E = \R^n$, and the elements of $R$ act as orthogonal reflections with respect to some hyperplanes of $E$.
These hyperplanes form a locally finite hyperplane arrangement in $E$, which we denote by $\A$.
The connected components of the complement of $\A$ in $E$ are called \emph{chambers}.
Given a chamber $C$, its \emph{walls} are the hyperplanes $H \in \A$ such that $H \cap \bar C$ is a $(n-1)$-dimensional polyhedron.
The collection of all the chambers forms the \emph{Coxeter complex} of $W$.
The smallest possible dimension $n$ which is needed to construct such a representation is equal to the rank of $W$.
If $n$ is equal to the rank of $W$, the resulting representation is \emph{essential}.
If $W$ is an irreducible finite Coxeter group, then all its elements must fix a point of $E$ (so $W$ may as well be realized as a group of linear isometries), and all chambers in an essential representation are unbounded simplicial cones.
If $W$ is an irreducible affine Coxeter group, then all chambers in an essential representation are bounded simplices.
We refer to \cite{humphreys1992reflection} for the definition of root systems, positive systems, simple systems, crystallographic systems, and crystallographic Coxeter groups.

Irreducible affine Coxeter groups are classified into four infinite families ($\tilde A_n$, $\tilde B_n$, $\tilde C_n$, $\tilde D_n$) and five exceptional cases ($\tilde E_6$, $\tilde E_7$, $\tilde E_8$, $\tilde F_4$, $\tilde G_2$).
Their Coxeter graphs are shown in \Cref{fig:affine-coxeter-groups}.
These groups can all be constructed from the corresponding irreducible crystallographic root systems, as follows (see \cite[Chapter 4]{humphreys1992reflection}).
For each $\alpha$ in a crystallographic root system $\Phi \subseteq \R^n$, and for each integer $k \in \Z$, consider the affine hyperplane $H_{\alpha, k} = \{ x \in \R^n \mid \< x, \alpha \> = k \}$.
Then take the group generated by the orthogonal reflections with respect to the hyperplanes $H_{\alpha,k}$.
The corresponding reflection arrangement $\A$ is the collection of all these hyperplanes $H_{\alpha,k}$.

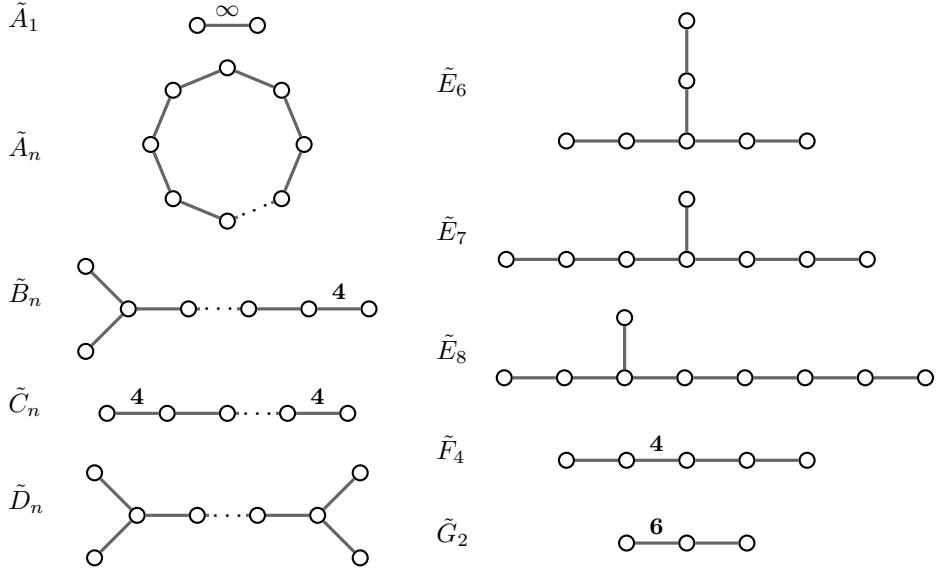
\begin{figure}[t]
	\begin{minipage}{0.45\textwidth}
		\begin{tabular}{>{\centering}l >{\centering\arraybackslash}m{0.73\textwidth}}
			$\tilde A_1$ & \begin{tikzpicture}
\begin{scope}[every node/.style={circle,thick,draw,inner sep=2}, every label/.style={rectangle,draw=none}]
  \node (1) at (0.0,0) {};
  \node (2) at (0.8,0) {};
\end{scope}
\begin{scope}[every edge/.style={draw=black!60,line width=1.2}]
  \path (1) edge node {} node[inner sep=3, above] {\textbf{\small $\infty$}} (2);
\end{scope}
\end{tikzpicture} \\
			$\tilde A_n$ & \begin{tikzpicture}
\begin{scope}[every node/.style={circle,thick,draw,inner sep=2}, every label/.style={rectangle,draw=none}]
  \node (0) at (0.00:1.02) [label={[label distance=1.5]0.00:$ $}] {};
  \node (1) at (45.00:1.02) [label={[label distance=1.5]45.00:$ $}] {};
  \node (2) at (90.00:1.02) [label={[label distance=1.5]90.00:$ $}] {};
  \node (3) at (135.00:1.02) [label={[label distance=1.5]135.00:$ $}] {};
  \node (4) at (180.00:1.02) [label={[label distance=1.5]180.00:$ $}] {};
  \node (5) at (225.00:1.02) [label={[label distance=1.5]225.00:$ $}] {};
  \node (6) at (270.00:1.02) [label={[label distance=1.5]270.00:$ $}] {};
  \node (7) at (315.00:1.02) [label={[label distance=1.5]315.00:$ $}] {};
\end{scope}
\begin{scope}[every edge/.style={draw=black!60,line width=1.2}]
  \path (0) edge node {} (1);
  \path (1) edge node {} (2);
  \path (2) edge node {} (3);
  \path (3) edge node {} (4);
  \path (4) edge node {} (5);
  \path (5) edge node {} (6);
  \path (6) edge node [fill=white, rectangle, inner sep=1.0, rotate=382.50, minimum height = 0.5cm] {$\ldots$} (7);
  \path (7) edge node {} (0);
\end{scope}
\begin{scope}[every edge/.style={draw=black, line width=3}]
\end{scope}
\begin{scope}[every node/.style={draw,inner sep=12.5,yshift=-6}, every label/.style={rectangle,draw=none,inner sep=7.0}, every fit/.append style=text badly centered]
\end{scope}
\end{tikzpicture} \\[-0.43cm]
			$\tilde B_n$ & \begin{tikzpicture}
\begin{scope}[every node/.style={circle,thick,draw,inner sep=2}, every label/.style={rectangle,draw=none}]
  \node (0) at (135.0:0.8) [label={above,minimum height=13}:$ $] {};
  \node (1) at (225.0:0.8) [label={above,minimum height=13}:$ $] {};
  \node (2) at (0.0,0) [label={above,minimum height=13}:$ $] {};
  \node (3) at (0.8,0) [label={above,minimum height=13}:$ $] {};
  \node (4) at (1.6,0) [label={above,minimum height=13}:$ $] {};
  \node (5) at (2.4,0) [label={above,minimum height=13}:$ $] {};
  \node (6) at (3.2,0) [label={above,minimum height=13}:$ $] {};
\end{scope}
\begin{scope}[every edge/.style={draw=black!60,line width=1.2}]
  \path (0) edge node {} (2);
  \path (1) edge node {} (2);
  \path (2) edge node {} (3);
  \path (3) edge node [fill=white, rectangle, inner sep=1.0, minimum height = 0.5cm] {$\ldots$} (4);
  \path (4) edge node {} (5);
  \path (5) edge node {} node[inner sep=3, above] {\textbf{\small 4}} (6);
\end{scope}
\begin{scope}[every edge/.style={draw=black, line width=3}]
\end{scope}
\begin{scope}[every node/.style={draw,inner sep=12.5,yshift=-6}, every label/.style={rectangle,draw=none,inner sep=7.0}, every fit/.append style=text badly centered]
\end{scope}
\end{tikzpicture} \\
			$\tilde C_n$ & \begin{tikzpicture}
\begin{scope}[every node/.style={circle,thick,draw,inner sep=2}, every label/.style={rectangle,draw=none}]
  \node (0) at (0.0,0) [label={above,minimum height=13}:$ $] {};
  \node (1) at (0.8,0) [label={above,minimum height=13}:$ $] {};
  \node (2) at (1.6,0) [label={above,minimum height=13}:$ $] {};
  \node (3) at (2.4,0) [label={above,minimum height=13}:$ $] {};
  \node (4) at (3.2,0) [label={above,minimum height=13}:$ $] {};
\end{scope}
\begin{scope}[every edge/.style={draw=black!60,line width=1.2}]
  \path (0) edge node {} node[inner sep=3, above] {\textbf{\small 4}} (1);
  \path (1) edge node {} (2);
  \path (2) edge node [fill=white, rectangle, inner sep=1.0, minimum height = 0.5cm] {$\ldots$} (3);
  \path (3) edge node {} node[inner sep=3, above] {\textbf{\small 4}} (4);
\end{scope}
\begin{scope}[every edge/.style={draw=black, line width=3}]
\end{scope}
\begin{scope}[every node/.style={draw,inner sep=12.5,yshift=-6}, every label/.style={rectangle,draw=none,inner sep=7.0}, every fit/.append style=text badly centered]
\end{scope}
\end{tikzpicture} \\[-0.17cm]
			$\tilde D_n$ & \begin{tikzpicture}
\begin{scope}[every node/.style={circle,thick,draw,inner sep=2}, every label/.style={rectangle,draw=none}]
  \node (0) at (135.0:0.8) [label={above,minimum height=13}:$ $] {};
  \node (1) at (225.0:0.8) [label={above,minimum height=13}:$ $] {};
  \node (2) at (0.0,0) [label={above,minimum height=13}:$ $] {};
  \node (3) at (0.8,0) [label={above,minimum height=13}:$ $] {};
  \node (4) at (1.6,0) [label={above,minimum height=13}:$ $] {};
  \node (5) at (2.4,0) [label={above,minimum height=13}:$ $] {};
  \node (6) at ($(2.4,0) + (45.0:0.8)$) [label={above,minimum height=13}:$ $] {};
  \node (7) at ($(2.4,0) + (-45.0:0.8)$) [label={above,minimum height=13}:$ $] {};
\end{scope}
\begin{scope}[every edge/.style={draw=black!60,line width=1.2}]
  \path (0) edge node {} (2);
  \path (1) edge node {} (2);
  \path (2) edge node {} (3);
  \path (3) edge node [fill=white, rectangle, inner sep=1.0, minimum height = 0.5cm] {$\ldots$} (4);
  \path (4) edge node {} (5);
  \path (5) edge node {} (6);
  \path (5) edge node {} (7);
\end{scope}
\begin{scope}[every edge/.style={draw=black, line width=3}]
\end{scope}
\begin{scope}[every node/.style={draw,inner sep=12.5,yshift=-6}, every label/.style={rectangle,draw=none,inner sep=7.0}, every fit/.append style=text badly centered]
\end{scope}
\end{tikzpicture} \\
		\end{tabular}
	\end{minipage}%
	\begin{minipage}{0.55\textwidth}
		\begin{tabular}{>{\centering}l >{\centering\arraybackslash}m{0.73\textwidth}}
			$\tilde E_6$ & \begin{tikzpicture}
\def\x{0.8}
\begin{scope}[every node/.style={circle,thick,draw,inner sep=2}, every label/.style={rectangle,draw=none}]
  \node (1) at (0,0) {};
  \node (2) at (1*\x,0) {};
  \node (3) at (2*\x,0) {};
  \node (4) at (3*\x,0) {};
  \node (5) at (4*\x,0) {};
  \node (6) at (2*\x,1*\x) {};
  \node (7) at (2*\x,2*\x) {};
\end{scope}
\begin{scope}[every edge/.style={draw=black!60,line width=1.2}]
  \path (1) edge (2);
  \path (2) edge (3);
  \path (3) edge (4);
  \path (4) edge (5);
  \path (3) edge (6);
  \path (6) edge (7);
\end{scope}
\end{tikzpicture} \\
			\\
			$\tilde E_7$ & \begin{tikzpicture}
\def\x{0.8}
\begin{scope}[every node/.style={circle,thick,draw,inner sep=2}, every label/.style={rectangle,draw=none}]
  \node (0) at (-1*\x,0) {};
  \node (1) at (0,0) {};
  \node (2) at (1*\x,0) {};
  \node (3) at (2*\x,0) {};
  \node (4) at (3*\x,0) {};
  \node (5) at (4*\x,0) {};
  \node (6) at (5*\x,0) {};
  \node (7) at (2*\x,1*\x) {};
\end{scope}
\begin{scope}[every edge/.style={draw=black!60,line width=1.2}]
  \path (0) edge (1);
  \path (1) edge (2);
  \path (2) edge (3);
  \path (3) edge (4);
  \path (4) edge (5);
  \path (5) edge (6);
  \path (3) edge (7);
\end{scope}
\end{tikzpicture} \\
			\\
			$\tilde E_8$ & \begin{tikzpicture}
\def\x{0.8}
\begin{scope}[every node/.style={circle,thick,draw,inner sep=2}, every label/.style={rectangle,draw=none}]
  \node (1) at (0,0) {};
  \node (2) at (1*\x,0) {};
  \node (3) at (2*\x,0) {};
  \node (4) at (3*\x,0) {};
  \node (5) at (4*\x,0) {};
  \node (6) at (5*\x,0) {};
  \node (7) at (6*\x,0) {};
  \node (8) at (2*\x,1*\x) {};
  \node (9) at (7*\x,0) {};
\end{scope}
\begin{scope}[every edge/.style={draw=black!60,line width=1.2}]
  \path (1) edge (2);
  \path (2) edge (3);
  \path (3) edge (4);
  \path (4) edge (5);
  \path (5) edge (6);
  \path (6) edge (7);
  \path (3) edge (8);
  \path (7) edge (9);
\end{scope}
\end{tikzpicture} \\
			\\
			$\tilde F_4$ & \begin{tikzpicture}
\def\x{0.8}
\begin{scope}[every node/.style={circle,thick,draw,inner sep=2}, every label/.style={rectangle,draw=none}]
  \node (1) at (0,0) {};
  \node (2) at (1*\x,0) {};
  \node (3) at (2*\x,0) {};
  \node (4) at (3*\x,0) {};
  \node (5) at (4*\x,0) {};
\end{scope}
\begin{scope}[every edge/.style={draw=black!60,line width=1.2}]
  \path (1) edge (2);
  \path (2) edge node {} node[inner sep=3, above] {\textbf{\small 4}} (3);
  \path (3) edge (4);
  \path (4) edge (5);
\end{scope}
\end{tikzpicture} \\
			\\
			$\tilde G_2$ & \begin{tikzpicture}
\def\x{0.8}
\begin{scope}[every node/.style={circle,thick,draw,inner sep=2}, every label/.style={rectangle,draw=none}]
  \node (1) at (0,0) {};
  \node (2) at (1*\x,0) {};
  \node (3) at (2*\x,0) {};
\end{scope}
\begin{scope}[every edge/.style={draw=black!60,line width=1.2}]
  \path (1) edge node {} node[inner sep=3, above] {\textbf{\small 6}} (2);
  \path (2) edge (3);
\end{scope}
\end{tikzpicture} \\
		\end{tabular}
	\end{minipage}
	
	\caption{Irreducible affine Coxeter graphs.}
	\label{fig:affine-coxeter-groups}
\end{figure}

If $W$ is a finite or affine Coxeter group, acting on $E = \R^n$, define the \emph{configuration space} $Y$ associated with $W$ as
\[ Y = \C^n \setminus \bigcup_{H \in \A} H \otimes_\R\C. \]
In other words, this is the complement of the complexification of the hyperplane arrangement $\A$.
Then $W$ naturally acts on $Y$, and its quotient $Y_W = Y / W$ is the \emph{orbit configuration space} associated with $W$.
Up to homotopy equivalence, $Y$ and $Y_W$ do not depend on the chosen representation of $W$ as a subgroup of the isometry group $\Isom(E)$.
The construction of $Y$ and $Y_W$ can be extended to arbitrary Coxeter groups by considering the Tits cone, see \cite{bourbaki1968elements, vinberg1971discrete, van1983homotopy, salvetti1994homotopy, moroni2012some, godelle2012basic, paris2012k}.

If $S$ is a set of simple reflections of $W$, the \emph{Artin group} associated with the Coxeter group $W$ is
\[ G_W = \< S \mid \!\!\underbrace{stst\dotsm}_{m(s,t) \text{ terms}} \! = \! \underbrace{tsts\dotsm}_{m(s,t) \text{ terms}} \forall\, s,t \in S \text{ such that $m(s,t) \neq \infty$} \>, \]
see \cite{tits1966normalisateurs, tits1969probleme, brieskorn1972artin, deligne1972immeubles, brieskorn1973groupes}.
It is isomorphic to the fundamental group of the orbit configuration space $Y_W$ \cite{van1983homotopy, salvetti1994homotopy}.

\begin{conjecture}[$K(\pi,1)$ conjecture]
	The orbit configuration space $Y_W$ is a classifying space for the Artin group $G_W$.
\end{conjecture}

Artin groups associated with finite (resp.\ affine) Coxeter groups are called \emph{spherical} (resp.\ affine).
Prior to this work, the $K(\pi,1)$ conjecture was proved for spherical Artin groups \cite{fox1962braid, deligne1972immeubles, brieskorn1973groupes}, for affine Artin groups of type $\tilde A_n$, $\tilde C_n$ \cite{okonek1979dask}, and $\tilde B_n$ \cite{callegaro2010k}, for Artin groups of dimension $\leq 2$ (this includes the affine Artin group of type $\tilde G_2$) and of FC type \cite{hendriks1985hyperplane, charney1995k}.

As shown in \cite{salvetti1987topology, salvetti1994homotopy}, the orbit configuration space $Y_W$ has the homotopy type of a CW complex $X_W$, with one $|T|$-cell $c_T$ for every $T$ in
\[ \Delta_W = \{ T \subseteq S \mid \text{the standard parabolic subgroup $W_T$ is finite} \}. \]
In particular, the $1$-cells of $X_W$ are indexed by $S$, and the $2$-cells are indexed by the unordered pairs $\{s,t\} \subseteq S$ with $m(s,t) \neq \infty$.
The $1$-cells are oriented in such a way that $c_{\{s\}}$ corresponds to the generator $s$ of the fundamental group $G_W$, and a $2$-cell $c_{\{s,t\}}$ corresponds to the relation $stst\dotsm = tsts\dotsm$.
In the literature, the CW complex $X_W$ is usually called the \emph{Salvetti complex} of $W$.
For $T \subseteq S$, there is a natural inclusion of complexes $X_{W_T} \subseteq X_W$, induced by the inclusion $\Delta_{W_T} \subseteq \Delta_W$.

\subsection{Posets}
\label{sec:posets}

We now recall some basic terminology about partially ordered sets (\emph{posets}).
See \cite{stanley2012enumerative} for a more detailed exposition.

Let $(P, \leq)$ be a poset.
If $p < q$ in $P$ and there is no element $r\in P$ with $p<r<q$, then we say that $q$ \emph{covers} $p$, and write $p \lessdot q$.
Given an element $q \in P$, define $P_{\leq q}= \{p \in P \mid p \leq q\}$.
We say that $P$ is \emph{bounded} if it contains a unique minimal element and a unique maximal element.
A (finite) \emph{chain} in $P$ is a totally ordered sequence $p_0 < p_1 < \dotsb < p_n$ of elements of $P$.
A chain of $n+1$ elements is conventionally said to have \emph{length} $n$.

If $p \leq q$, the \emph{interval} $[p,q]$ in $P$ is the set of all elements $r \in P$ such that $p \leq r \leq q$.
We say that $P$ is \emph{graded} if, for every $p \leq q$, all the maximal chains in $[p,q]$ have the same (finite) length.
Then there exists a \emph{rank function} $\rk\colon P \to \Z$ such that $\rk(q) - \rk(p)$ is the length of any maximal chain in $[p,q]$.
The rank of $P$ is defined as the maximal length of a chain of $P$.

A poset $P$ is said to be a \emph{lattice} if every pair of elements $p,q \in P$ has a unique maximal common lower bound and a unique minimal common upper bound.

The \emph{Hasse diagram} of a poset $P$ is the graph with vertex set $P$ and having an edge $(p, q)$ for every covering relation $p \lessdot q$.
We indicate by $\E(P) = \{ (p,q) \in P \times P \mid p \lessdot q\}$ the set of edges of the Hasse diagram of $P$.

\subsection{Lexicographic shellability}
\label{sec:lexicographic-shellability}

In this section, we recall the definition of EL-shellability
\cite{bjorner1983lexicographically,bjorner1996shellable}.

Let $P$ be a bounded poset.
An \emph{edge labeling} of $P$ is a map $\lambda\colon \E(P) \to \Lambda$, where $\Lambda$ is some poset.
Given an edge labeling $\lambda$, each maximal chain $c = (x \lessdot z_1 \lessdot \dotsb \lessdot z_t \lessdot y)$ between any two elements $x \leq y$ has an associated word
\[ \lambda(c) = \lambda(x, z_1) \, \lambda(z_1, z_2) \,\dotsm\, \lambda(z_t, y). \]
We say that the chain $c$ is \emph{increasing} if the associated word $\lambda(c)$ is strictly increasing.
Maximal chains in a fixed interval $[x,y]\subseteq P$ can be compared lexicographically (i.e.\ by using the lexicographic ordering on the corresponding words).

\begin{definition}
	Let $P$ be a bounded poset.
	An \emph{edge-lexicographic labeling} (or simply \emph{EL-labeling}) of $P$ is an edge labeling such that in each closed interval $[x,y] \subseteq P$ there is a unique increasing maximal chain, and this chain lexicographically precedes all other maximal chains of $[x,y]$.
	\label{def:el-labeling}
\end{definition}

A bounded poset that admits an EL-labeling is said to be \emph{EL-shellable}.
If $P$ is an EL-shellable poset, then the order complex of $P \setminus \{ \min(P), \max(P) \}$ is homotopy equivalent to a wedge of spheres.

Let $P_1$ and $P_2$ be bounded posets that admit EL-labelings $\lambda_1\colon \E(P_1)\to \Lambda_1$ and $\lambda_2\colon \E(P_2) \to \Lambda_2$, respectively. Assume that $\Lambda_1$ and $\Lambda_2$ are disjoint and totally ordered.
Let $\lambda\colon \E(P_1 \times P_2) \to \Lambda_1 \cup \Lambda_2$ be the edge labeling of $P_1\times P_2$ defined as follows:
\begin{IEEEeqnarray*}{lCl}
	\lambda((x,y), \, (z,y)) &=& \lambda_1(x,z) \\
	\lambda((x,y), \, (x,t)) &=& \lambda_2(y,t).
\end{IEEEeqnarray*}

\begin{theorem}[{\cite[Proposition 10.15]{bjorner1997shellable}}]
	\label{thm:product-el-labelings}
	Fix any shuffle of the total orderings on $\Lambda_1$ and $\Lambda_2$, to get a total ordering on $\Lambda_1 \cup \Lambda_2$.
	Then the product edge labeling $\lambda$ defined above is an EL-labeling of $P_1 \times P_2$.
\end{theorem}

\subsection{Discrete Morse theory}
\label{sec:dmt}

In this section we recall Forman's discrete Morse theory \cite{forman1998morse,forman2002user}.
We follow the point of view of Chari \cite{chari2000discrete}, using acyclic matchings instead of discrete Morse functions, and we make use of the generality of \cite[Section 3]{batzies2002discrete} for the case of infinite CW complexes.

Let $P$ be a graded poset, and denote by $H$ the Hasse diagram of $P$.
Given a subset $\M$ of $\E(P)$, we can orient all edges of $H$ in the following way: an edge $(p,q) \in \E(P)$ (i.e.\ with $p \lessdot q$) is oriented from $p$ to $q$ if it is in $\M$, otherwise in the opposite direction.
Denote this oriented graph by $H_{\M}$.

A \emph{matching} on $P$ is a subset $\M \subseteq \E(P)$ such that every element of $P$ appears in at most one edge of $\M$.
A matching $\M$ is \emph{acyclic} if the graph $H_{\M}$ has no directed cycles.
Given a matching $\M$ on $P$, an \emph{alternating path} is a directed path in $H_{\M}$ such that two consecutive edges of the path do not belong both to $\E(P) \setminus \M$.
In the graph $H_\M$, the edges in $\M$ increase the rank by $1$, and the edges in $\E(P) \setminus \M$ decrease the rank by $1$.
Therefore, a matching $\M$ is acyclic if and only if it has no closed alternating paths (which are called \emph{alternating cycles}).
The elements of $P$ that do not appear in any edge of $\M$ are called \emph{critical} (with respect to the matching $\M$).
An acyclic matching $\M$ is \emph{proper} if, for every $p \in P$, the set of vertices of $H_\M$ reachable from $p$ (with a directed path) is finite.

Let $X$ be a CW complex.
The \emph{face poset} $\F(X)$ of $X$ is the set of its (open) cells together with the partial order defined by $\tau \leq \sigma$ if $\bar\tau \subseteq \bar\sigma$.
For all CW complexes $X$ considered in this paper, the face poset $\F(X)$ is a graded poset with rank function $\rk(\sigma) = \dim(\sigma)$.
Recall that each cell of $X$ has a characteristic map $\Phi\colon D^n\to X$, where $D^n = \{x \in \R^n \mid \|x\| \leq 1 \}$.

Let $\sigma$ and $\tau$ be cells of $X$.
If $\tau \lessdot \sigma$ we say that $\tau$ is a \emph{face} of $\sigma$.
We say that $\tau$ is a \emph{regular face} of $\sigma$ if, in addition, the following two conditions hold (set $n=\dim(\tau)$ and let $\Phi$ be the characteristic map of $\sigma$):
\begin{enumerate}[(i)]
	\itemsep0.1cm
	\item $\Phi|_{\Phi^{-1}(\tau)}\colon \Phi^{-1}(\tau) \to \tau$ is a homeomorphism;
	\item $\smash{\overline{\Phi^{-1}(\tau)}}$ is a closed $n$-ball in $D^{n+1}$.
\end{enumerate}

The following is a particular case of the main theorem of discrete Morse theory and follows from \cite[Theorem 3.2.14 and Remark 3.2.17]{batzies2002discrete}.

\begin{theorem}[\cite{forman1998morse,chari2000discrete,batzies2002discrete}]
	\label{thm:discrete-morse-theory}
	Let $X$ be a CW complex, and let $Y \subseteq X$ be a subcomplex.
	Suppose that there exists a proper acyclic matching $\M$ on the face poset $\F(X)$ such that: $\F(Y)$ is the set of critical cells; for every $(\tau, \sigma) \in \M$, we have that $\tau$ is a regular face of $\sigma$.
	Then $X$ deformation retracts onto $Y$.
	In particular, the inclusion $Y \hookrightarrow X$ is a homotopy equivalence.
\end{theorem}

We conclude by recalling a standard tool to construct acyclic matchings.

\begin{theorem}[Patchwork theorem {\cite[Theorem 11.10]{kozlov2007combinatorial}}]\label{thm:patchwork}
	Let $\eta \colon P \to Q$ be a poset map.
	For all $q \in Q$, assume to have an acyclic matching $\M_q \subseteq \E(P)$ that involves only elements of the fiber $\eta^{-1}(q) \subseteq P$.
	Then the union of these matchings is itself an acyclic matching on $P$.
\end{theorem}

\subsection{Interval groups and Garside structures}
\label{sec:interval-groups}

We now recall the construction of interval groups, and how they give rise to Garside structures.
Our exposition mostly follows \cite[Section 2]{mccammond2017artin} and \cite[Section 1]{mccammond2005introduction}.
See \cite{dehornoy1999gaussian, bessis2003dual, charney2004bestvina, digne2006presentations, digne2012garside, dehornoy2013garside, dehornoy2015foundations, mccammond2015dual, mccammond2017artin} for a complete reference on Garside structures.

Let $G$ be a group, with a (possibly infinite) generating set $R \subseteq G$ such that $R = R^{-1}$.
Suppose that the elements of $R$ are assigned positive weights bounded away from $0$ that form a discrete subset of the positive real numbers.
Assume that $r$ and $r^{-1}$ have the same weight, for every $r \in R$.
For every $x \in G$, denote by $l(x)$ the minimum sum of the weights of some generators $r_1,r_2, \dotsc, r_k \in R$ such that $r_1r_2\dotsm r_k = x$.
In other words, $l(x)$ is the distance between $1$ and $x$ in the weighted right Cayley graph of $G$ (with respect to the weighted generating set $R$).

The group $G$ becomes a poset if we set $x \leq y$ whenever $l(x) + l(x^{-1}y) = l(y)$, i.e.\ if there is a minimal length factorization of $y$ that starts with a minimal length factorization of $x$.
Given an element $g \in G$, denote by $[1,g]^G \subseteq G$ the interval between $1$ and $g$ (with respect to the partial order $\leq$ in $G$).
The Hasse diagram of $[1,g]^G$ embeds into the Cayley graph of $G$.
Every edge of the Hasse diagram is of the form $(x, xr)$ for some $r \in R$, and we label it by $r$.

\begin{definition}[Interval group {\cite[Definition 2.6]{mccammond2017artin}}]
	The \emph{interval group} $G_g$ is the group presented as follows.
	Let $R_0$ be the subset of $R$ consisting of the labels of edges in $[1,g]^G$.
	The group $G_g$ has $R_0$ as its generating set, and relations given by all the closed loops inside the Hasse diagram of $[1,g]^G$.
	\label{def:interval-group}
\end{definition}

The interval $[1,g]^G$ is \emph{balanced} if the following condition is satisfied: for every $x \in G$, we have $l(x) + l(x^{-1}g) = l(g)$ if and only if $l(gx^{-1}) + l(x) = l(g)$.
This condition is automatically satisfied if the generating set $R$ is closed under conjugation and the weight of a generator is equal to the weight of all its conjugates.

\begin{theorem}[{\cite[Theorem 0.5.2]{bessis2003dual}, \cite[Proposition 2.11]{mccammond2017artin}}]
	If the interval $[1,g]^G$ is a balanced lattice, then the group $G_g$ is a Garside group.
	\label{thm:garside-structure}
\end{theorem}

See \cite{dehornoy1999gaussian, digne2006presentations, digne2012garside} for the definition of Garside groups.
As in \cite{mccammond2017artin}, we use the term ``Garside group'' in the sense of Digne \cite{digne2006presentations, digne2012garside} (so that the generating set $R$ need not be finite).

The classifying space of a Garside group can be constructed explicitly, as shown in \cite{dehornoy2003homology, charney2004bestvina}.
Here we generalize this construction to the case of arbitrary interval groups arising from balanced intervals, without the lattice assumption.
In the case of Garside groups, it is equivalent to the definitions given in \cite[Section 3]{charney2004bestvina} (see in particular \cite[Definition 3.5 and Theorem 3.6]{charney2004bestvina}) and in \cite[Definition 1.6]{mccammond2005introduction}.

\begin{definition}[Interval complex]
	Realize the standard $d$-simplex $\Delta^d$ as the set of points $(a_1,a_2,\dotsc,a_d) \in \R^d$ such that $1 \geq a_1 \geq a_2 \geq \dotsb \geq a_d \geq 0$.
	The \emph{interval complex} associated with a balanced interval $[1,g]^G$ is a $\Delta$-complex (in the sense of \cite{hatcher}) having a $d$-simplex $[x_1|x_2|\dotsb|x_d]$ for every $x_1,x_2,\dotsc,x_d \in [1,g]^G$ such that:
	\begin{enumerate}[(i)]
		\item $x_i \neq 1$ for all $i$;
		\item $x_1x_2\dotsm x_d \in [1,g]^G$;
		\item $l(x_1x_2\dotsm x_d) = l(x_1) + l(x_2) + \dotsb + l(x_d)$.
	\end{enumerate}
	The faces of a simplex $[x_1|\dotsb|x_d]$ are as follows.
	\begin{itemize}
		\item The face $\{1=a_1\geq a_2\geq \dotsb \geq a_d \geq 0 \}$ of $[x_1|\dotsb|x_d]$ is glued to the $(d-1)$-simplex $[x_2|\dotsb| x_d]$ by sending $(1,a_2,\dotsc,a_d)$ to $(a_2,\dotsc,a_d) \in \Delta^{d-1}$.

		\item For $1 \leq i \leq d-1$, the face $\{1 \geq a_1 \geq \dotsb \geq a_i = a_{i+1} \geq \dotsb \geq a_d \geq 0\}$ of $[x_1|\dotsb| x_d]$ is glued to the $(d-1)$-simplex $[x_1|\dotsb|x_ix_{i+1}|\dotsb|x_d]$ by sending $(a_1,\dotsc,a_i,a_i,a_{i+2},\dotsc,a_d)$ to $(a_1,\dotsc,a_i,a_{i+2},\dotsc,a_d) \in \Delta^{d-1}$.
		
		\item Finally, the face $\{1\geq a_1 \geq \dotsb \geq a_d = 0\}$ of $[x_1|\dotsb|x_d]$ is glued to the $(d-1)$-simplex $[x_1|\dotsb|x_{d-1}]$ by sending $(a_1,\dotsc,a_{d-1},0)$ to $(a_1,\dotsc,a_{d-1}) \in \Delta^{d-1}$.
	\end{itemize}
	Notice that there is a unique vertex, which is indicated by $[\,]$.
	The $1$-simplices are oriented going from $0$ to $1$ in $\Delta^1 = [0,1]$.
	See \Cref{fig:simplex} for an example.
	The fact that $[1,g]^G$ is balanced ensures that the faces of a simplex also belong to the interval complex.
	\label{def:interval-complex}
\end{definition}

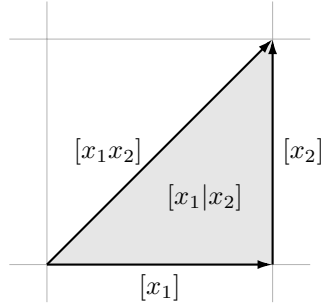
\begin{figure}[tbp]
	\begin{center}
		\begin{tikzpicture}
		\draw[step=3.0,black!30,very thin] (2.5,2.5) grid (6.5,6.5);
		
		\begin{scope}[every path/.style={thick, -{Latex[length=2mm,width=1.2mm]}}]
			\draw (3,3) -- (6,3) node[midway,below] {$[x_1]$};
			\draw (6,3) -- (6,6) node[midway,right] {$[x_2]$};
			\draw (3,3) -- (6,6) node[midway,left] {$[x_1x_2]$\;};
		\end{scope}
		\path[fill=black,opacity=0.1] (3,3) -- (6,3) -- (6,6) -- cycle;
		\node (A) at (5.1,3.9) {$[x_1|x_2]$};
		\end{tikzpicture}
	\end{center}
	\caption{Realization in $\R^2$ of a $2$-simplex $[x_1|x_2]$.}
	\label{fig:simplex}
\end{figure}

The fundamental group of the interval complex associated with $[1,g]^G$ is $G_g$.
This can be easily checked by looking at the $2$-skeleton.

\begin{theorem}[{\cite[Theorem 3.1]{charney2004bestvina}}]
	\label{thm:garside-classifying-space}
	If $[1,g]^G$ is a balanced lattice, then $G_g$ is a Garside group and the interval complex associated with $[1,g]^G$ is a classifying space for $G_g$.
\end{theorem}

\subsection{Intervals in the group of Euclidean isometries}
\label{sec:euclidean-isometries}

In this section, we recall the main result of \cite{brady2015factoring}.
Let $V\cong \R^n$ be a $n$-dimensional Euclidean vector space, and let $E$ be the associated affine space (where the origin has been forgotten).
Given an affine subspace $B \subseteq E$, denote by $\Dir(B) \subseteq V$ the space of directions of $B$.
Given a subset $U \subseteq V$, denote by $\Span(U)$ the linear subspace generated by $U$.

Let $L = \Isom(E)$ be the group of Euclidean isometries of $E$.
For every isometry $u \in L$, define its \emph{move-set} $\Mov(u) = \{ u(\a) - \a \mid \a \in E \} \subseteq V$.
This is an affine subspace of $V$, and it has a unique vector $\mu$ of minimal norm.
Define the \emph{min-set} of $u$ as $\Min(u) = \{ \a \in E \mid u(\a) = \a + \mu \} \subseteq E$. This is an affine subspace of $E$.

An isometry $u \in L$ is called \emph{elliptic} if it fixes at least one point, and \emph{hyperbolic} otherwise.
If $u$ is elliptic, then $\Mov(u)$ is a linear subspace, $\mu = 0$, and $\Min(u)$ coincides with the set of fixed points of $u$, which we denote by $\Fix(u)$.
For every isometry $u \in L$, there is an orthogonal decomposition $V = \Dir(\Mov(u)) \oplus \Dir(\Min(u))$ \cite[Lemma 3.6]{brady2015factoring}.

The group $L$ is generated by the set $R$ of all orthogonal reflections (where every reflection is assigned a weight of 1).
The length $l(u)$ computed using the generating set $R$ is called the \emph{reflection length} of $u$.
If $u$ is elliptic, then $l(u) = \codim\Fix(u)$; if $u$ is hyperbolic, then $l(u) = \dim\Mov(u)+2$ \cite[Theorem 5.7]{brady2015factoring}.

\begin{definition}[{Global poset \cite[Definition 7.1]{brady2015factoring}}]
	Define the \emph{global poset} $(P, \leq)$ as the set containing an element $e^B$ for every affine subspace $B \subseteq E$, and an element $h^M$ for every non-linear affine subspace $M \subseteq V$.
	The order relations in $P$ are as follows: $h^M \leq h^{M'}$ if $M \subseteq M'$; $e^B \leq e^{B'}$ if $B \supseteq B'$; $e^B < h^M$ if $\Span(M)^\perp \subseteq \Dir(B)$.
	Define also an \emph{invariant map} $\inv\colon L \to P$ that sends $u$ to $e^{\Fix(u)}$ if $u$ is elliptic, and to $h^{\Mov(u)}$ if $u$ is hyperbolic.
	\label{def:global-poset}
\end{definition}

\begin{theorem}[{\cite[Theorem 8.7]{brady2015factoring}}]
	For every isometry $u \in L$, the restriction of the invariant map is a poset isomorphism between the interval $[1,u]^L$ and the \emph{model poset} $P(u) = P_{\leq \,\inv(u)}$.
	\label{thm:model-poset}
\end{theorem}

\subsection{Dual Artin groups}
\label{sec:dual-artin}

In this section we recall the definition and some properties of dual Artin groups associated with a Coxeter group $W$, focusing on the cases where $W$ is finite or affine.
See \cite{brady2002k, bessis2003dual, brady2008non} for the finite case, and \cite{mccammond2015dual, mccammond2017artin} for the affine case.

Let $W$ be a Coxeter group, $R$ its set of reflections, and $S \subseteq R$ a set of simple reflections.
Assign a weight of $1$ to every reflection $r \in R$.
Let $w$ be a Coxeter element, obtained as a product of the elements of $S$.
The \emph{dual Artin group with respect to $w$} is the interval group $W_w$ constructed using $R$ as the generating set of $W$.

The properties of a dual Artin group are strictly related to the combinatorics of its defining interval $[1,w]^W$, which in turn depends on the geometry of the Coxeter element $w$.
The interval $[1,w]^W$ is a graded poset of rank $|S|$.
As explained in \Cref{sec:interval-groups}, the edges of the Hasse diagram of $[1,w]^W$ are naturally labeled by a subset $R_0 \subseteq R$.
Maximal chains in $[1,w]^W$ correspond to minimal length factorizations of $w$ as a product of reflections.

Since the set $R$ of reflections is closed under conjugation, it is possible to rewrite factorizations as follows (this is a consequence of the so called \emph{Hurwitz action}).
\begin{lemma}[{\cite[Lemma 3.7]{mccammond2015dual}}]
	\label{lemma:hurwitz-action}
	Let $u = r_1r_2\dotsm r_m$ be a reflection factorization in a Coxeter group $W$.
	For any selection $1 \leq i_1 < i_2 < \dotsb < i_j \leq m$ of positions, there is a length $m$ reflection factorization of $u$ whose first $j$ reflections are $r_{i_1}r_{i_2}\dotsm r_{i_j}$, and another length $m$ reflection factorization of $u$ where these are the last $j$ reflections.
\end{lemma}

If the Coxeter graph of $W$ is a tree, then all its Coxeter elements are \emph{geometrically equivalent},
and give rise to isomorphic intervals $[1,w]^W$ \cite[Proposition 7.5]{mccammond2015dual}.
This is the case for all irreducible finite and affine Coxeter groups except $\tilde A_n$.
In the case $\tilde A_n$, there are $\lfloor \frac{n+1}{2} \rfloor$ equivalence classes of Coxeter elements: a choice of representatives is given by \emph{$(p,q)$-bigon Coxeter elements}, where $(p,q)$ is a pair or positive integers such that $p \geq q$ and $p+q = n+1$ \cite[Definition 7.7]{mccammond2015dual}.

The generating set $R_0 \subseteq R$ of a dual Artin group contains $S$ (for a general Coxeter group, this follows from \Cref{lemma:coxeter-reflection-length}).
Then there is a natural group homomorphism from the usual Artin group $G_W$ to the dual Artin group $W_w$.

\begin{theorem}[\cite{brady2002k, bessis2003dual, mccammond2017artin}]
	\label{thm:dual-artin-group}
	If $W$ is a finite or affine Coxeter group, the natural homomorphism $G_W \to W_w$ is an isomorphism.
\end{theorem}

It is not known in general whether a dual Artin group is isomorphic to the corresponding Artin group, or whether the isomorphism type of a dual Artin group depends on the chosen Coxeter element $w$.

One important motivation to introduce dual Artin groups is that sometimes they are Garside groups.
For example, this happens if $W$ is finite.

\begin{theorem}[\cite{bessis2003dual, brady2008non}]
	If $W$ is a finite Coxeter group, the interval $[1,w]^W$ is a lattice for every Coxeter element $w$.
	Therefore the dual Artin group $W_w$ is a Garside group.
	\label{thm:finite-lattice}
\end{theorem}

The intervals $[1,w]^W$ that arise from finite Coxeter groups $W$ are well-studied, and are called \emph{(generalized) noncrossing partition lattices} (see \cite{armstrong2009generalized}).
They are known to be EL-shellable \cite{athanasiadis2007shellability}.
By analogy with the finite case, for any Coxeter group $W$ and Coxeter element $w$, we call the interval $[1,w]^W$ a \emph{noncrossing partition poset} associated with $W$.

Suppose now that $W$ is an irreducible affine Coxeter group, acting as a reflection group on $E = \R^n$, where $n$ is the rank of $W$.
The Coxeter element $w$ is a hyperbolic isometry of reflection length $n+1$, and its min-set is a line $\ell$ called the \emph{Coxeter axis} \cite[Proposition 7.2]{mccammond2015dual}.
We gain some insight on the structure of the interval $[1,w]^W$ by comparing it with the interval $[1,w]^L$ in the group of all Euclidean isometries of $E$ (see \Cref{sec:euclidean-isometries}).

\begin{lemma}[cf.\ \cite{mccammond2015dual}]
	\label{lemma:model-poset}
	Let $W$ be an irreducible affine Coxeter group, and $w$ one of its Coxeter elements.
	Then the inclusion $[1,w]^W \hookrightarrow [1,w]^L$ is order-preserving and rank-preserving.
	In particular, $u\leq v$ in $[1,w]^W$ implies $\inv(u) \leq \inv(v)$ in the model poset $P(w)$.
\end{lemma}

\begin{proof}
	The length functions of $W$ and $L$ agree on the Coxeter element $w$, so they agree on the entire interval $[1,w]^W$.
	As a consequence, $[1,w]^W \subseteq [1,w]^L$.
	In addition, if we have $u \leq v$ in $[1,w]^W$, then $u^{-1}v \in [1,w]^W$ and $l(u^{-1}v) = l(v)-l(u)$.
	Since the length functions agree, we have $u \leq v$ also in $[1,w]^L$.
	The last part of the statement follows from \Cref{thm:model-poset}.
\end{proof}

Notice that $[1,w]^W$ is not a subposet of $[1,w]^L$ in general: it is possible to have $u, v \in [1,w]^W$ with $u \not\leq v$ in $[1,w]^W$, but $u \leq v$ in $[1,w]^L$ (see \Cref{ex:horizontal-A2}).
However, if $W'$ is a finite Coxeter group (acting as a reflection group on the vector space $V = \R^n$) and $w'$ is one of its Coxeter elements, then $[1,w']^{W'}$ is known to be a subposet of $[1,w']^{\Isom(V)}$ \cite[Sections 2 and 3]{brady2008non}.
Then, in the affine case, we can show that the condition $\inv(u) \leq \inv(v)$ implies $u \leq v$ in $[1,w]^W$ whenever $u$ and $v$ are elliptic.
In this case, the condition $\inv(u) \leq \inv(v)$ means $\Fix(u) \supseteq \Fix(v)$.

\begin{lemma}
	\label{lemma:elliptic-subposet}
	Let $W$ be an irreducible affine Coxeter group, and $w$ one of its Coxeter elements.
	If $u,v \in [1,w]^W$ are elliptic elements with $\Fix(u) \supseteq \Fix(v)$, then $u \leq v$ in $[1,w]^W$.
\end{lemma}

\begin{proof}
	We proceed by induction on $l(u)$, the case $l(u)=0$ being trivial.
	Suppose from now on that $l(u) \geq 1$.
	Let $W'$ be the subgroup of $W$ generated by the reflections in $[1,u]^W \cup [1,v]^W$.
	By \cite[Theorem 8.2]{humphreys1992reflection} this is a Coxeter group, and its set of reflections contains $R \cap ([1,u]^W \cup [1,v]^W)$.
	Every reflection in a minimal length factorization of $u$ (resp.\ $v$) in $W$ belongs to $[1,u]^W$ (resp.\ $[1,v]^W$), and so it belongs to $W'$.
	This means that the minimal length factorizations of $u$ and $v$ are the same in $W$ and $W'$.
	
	For every reflection $r$ in $[1,u]^W \cup [1,v]^W$, we have $\Fix(r) \supseteq \Fix(v)$ by \Cref{lemma:model-poset}.
	Therefore every element of $W'$ fixes $\Fix(v)$, and so $W'$ is finite.
	By \cite[Lemma 1.2.1]{bessis2003dual} (see also \cite{brady2002k, athanasiadis2007shellability, brady2008non}), every reflection $r \in W'$ is part of a minimal length factorization of $v$ in $W'$, and so also in $W$, thus $r \in [1,v]^W$.
	This proves that $R \cap [1,u]^W \subseteq R \cap [1,v]^W$.
	
	Let $r$ be any reflection in $[1,u]^W$ (there is at least one reflection because $l(u) \geq 1$).
	We have $\Fix(r) \supseteq \Fix(u) \supseteq \Fix(v)$ by \Cref{lemma:model-poset}, and therefore $r \leq u \leq v$ in $[1,w]^L$ by \Cref{thm:model-poset}.
	If we write $u = ru'$ and $v = rv'$, we have $u' \leq v'$ in $[1,w]^L$, and so $\Fix(u') \supseteq \Fix(v')$.
	Since $r \in [1,u]^W$, we have $u' \in [1,u]^W \subseteq [1,w]^W$ and $l(u') = l(u)-1$.
	In addition, since $r \in [1,v]^W$, we have $v' \in [1,v]^W \subseteq [1,w]^W$.
	By induction, $u' \leq v'$ in $[1,w]^W$.
	This implies that $u \leq v$ in $[1,w]^W$.
\end{proof}

The direction of the Coxeter axis $\ell$ is declared to be \emph{vertical}, and the orthogonal directions are \emph{horizontal}.
An elliptic isometry is \emph{horizontal} if it moves every point in a horizontal direction, and it is \emph{vertical} otherwise \cite[Definition 5.3]{mccammond2017artin}.
Given $u \in [1,w]^W$, the \emph{right complement} of $u$ is the unique $v\in [1,w]^W$ such that $uv=w$. Define the \emph{left complement} similarly.

\Cref{lemma:model-poset} and some geometric considerations in \cite{brady2015factoring} allow to coarsely describe the combinatorial structure of the interval $[1,w]^W$.

\begin{proposition}[cf.\ {\cite[Definitions 5.4 and 5.5]{mccammond2017artin}}]
    The elements $u \in [1,w]^W$ are split into 3 rows according to the following cases (where $v$ is the right complement of $u$):
    \begin{itemize}
    	\item (bottom row) $u$ is horizontal elliptic and $v$ is hyperbolic;
    	\item (middle row) both $u$ and $v$ are vertical elliptic;
    	\item (top row) $u$ is hyperbolic and $v$ is horizontal elliptic.
    \end{itemize}
    The bottom and the top rows contain a finite number of elements, whereas the middle row contains infinitely many elements.
\end{proposition}

This coarse structure has the following implications, given elements $u \leq v$ in $[1,w]^W$: if $v$ is elliptic, then $u$ is elliptic; if $v$ is horizontal elliptic, then $u$ is horizontal elliptic; if $u$ is vertical, then $v$ is vertical; if $u$ is hyperbolic, then $v$ is hyperbolic.


The roots corresponding to horizontal reflections form a root system $\Phi_\h \subseteq \Phi$, called the \emph{horizontal root system} associated with the Coxeter element $w \in W$.
It decomposes as a disjoint union of orthogonal irreducible root systems of type $A$, as described in \Cref{table:horizontal-root-systems}.
The number $k$ of irreducible components varies from $1$ to $3$.
See \cite[Section 11]{mccammond2015dual} and \Cref{appendix}.

\begin{table}[t]
	\captionsetup{position=below}
	{\def\arraystretch{1.3}
	\begin{tabular}{|c|l|}
		\hline
		Type & Horizontal root system \\
		\hline
		$\tilde A_n$ & $\Phi_{A_{p-1}} \sqcup \Phi_{A_{q-1}}$ \\
		$\tilde C_n$ & $\Phi_{A_{n-1}}$ \\
		$\tilde B_n$ & $\Phi_{A_1} \sqcup \Phi_{A_{n-2}}$ \\
		$\tilde D_n$ & $\Phi_{A_1} \sqcup \Phi_{A_1} \sqcup \Phi_{A_{n-3}}$ \\
		\hline
		$\tilde G_2$ & $\Phi_{A_{1}}$ \\
		$\tilde F_4$ & $\Phi_{A_{1}} \sqcup \Phi_{A_2}$ \\
		$\tilde E_6$ & $\Phi_{A_{1}} \sqcup \Phi_{A_2} \sqcup \Phi_{A_2}$ \\
		$\tilde E_7$ & $\Phi_{A_{1}} \sqcup \Phi_{A_2} \sqcup \Phi_{A_3}$ \\
		$\tilde E_8$ & $\Phi_{A_{1}} \sqcup \Phi_{A_2} \sqcup \Phi_{A_4}$ \\
		\hline
	\end{tabular}}

	\caption{Horizontal root systems \cite[Table 1]{mccammond2017artin}.
	In the case $\tilde A_n$, we show the horizontal root system associated with a $(p,q)$-bigon Coxeter element.}
	\label{table:horizontal-root-systems}
\end{table}

\begin{theorem}[\cite{digne2006presentations, digne2012garside, mccammond2015dual}]
	Let $W$ be an irreducible affine Coxeter group, and $w$ one of its Coxeter elements.
	The interval $[1,w]^W$ is a lattice (and thus $W_w$ is a Garside group) if and only if the horizontal root system associated with $w$ is irreducible.
	This happens in the cases $\tilde C_n$, $\tilde G_2$, and $\tilde A_n$ if $w$ is a $(n,1)$-bigon Coxeter element.
\end{theorem}

Since the interval $[1,w]^W$ is not a lattice in general, in \cite{mccammond2017artin} a new group of isometries $C \supseteq W$ is constructed, with the property that $[1,w]^C$ is a balanced lattice and $[1,w]^W \subseteq [1,w]^C$.
The corresponding interval group $C_w$ (called \emph{braided crystallographic group}) is a Garside group, and there is a natural inclusion $W_w \subseteq C_w$.
By \Cref{thm:garside-classifying-space}, the interval complex $K_C$ associated with $[1,w]^C$ is a (finite-dimensional) classifying space for $C_w$.
The cover of $K_C$ corresponding to the subgroup $W_w$ is a classifying space for the (dual) affine Artin group $W_w$.
Therefore affine Artin groups admit a finite-dimensional classifying space.
We come back to braided crystallographic groups in \Cref{sec:classifying-spaces}, where we show that the interval complex $K_W$ associated with $[1,w]^W$ is a classifying space for $W_w$ (this complex is much simpler than the aforementioned cover of $K_C$).

In the subsequent sections, we sometimes suppress the superscript $W$ when writing intervals $[1,u]^W$ in a Coxeter group $W$, and simply write $[1,u]$.

\section{Affine Coxeter elements}
\label{sec:affine-coxeter-elements}

This section is devoted to proving some results on the geometry of Coxeter elements of affine Coxeter groups, expanding the theory of \cite{mccammond2015dual, mccammond2017artin}.
We start by recalling, in \Cref{sec:bipartite-coxeter}, the results of \cite[Sections 8 and 9]{mccammond2015dual} on bipartite Coxeter elements.
In \Cref{sec:A-coxeter} we develop a parallel theory for the case $\tilde A_n$.
In \Cref{sec:isometries-below} we prove a few structural results for the elements of the interval $[1,w]$.
Finally, in \Cref{sec:horizontal-components} we make a digression on the geometry of the irreducible horizontal components.

This section goes hand in hand with the appendix, where we carry out explicit computations for the four infinite families of irreducible affine Coxeter groups.
A few results of \Cref{sec:A-coxeter} and one technical lemma of \Cref{sec:isometries-below} are checked by hand in the appendix.
The appendix can be also used as a source of additional examples.

Let $W$ be an irreducible affine Coxeter group, acting faithfully by Euclidean isometries on $E = \R^n$ where $n$ is the rank of $W$, as described in \Cref{sec:coxeter-artin}.
Let $R \subseteq W$ be the set of reflections, $w$ a Coxeter element of $W$, and $\ell = \Min(w) \subseteq E$ the Coxeter axis of $w$.
Denote by $\A$ the reflection arrangement associated with the action of $W$ on $E$.

The shortest vector $\mu$ in $\Mov(w)$ gives an orientation to the Coxeter axis $\ell$: we say that a point $\a \in \ell$ is \emph{above} a point $\b \in \ell$ (or, equivalently, $\b$ is \emph{below} $\a$) if $\a-\b$ is a positive multiple of $\mu$.
We also say that $\mu$ points towards the \emph{positive} direction of $\ell$, whereas $-\mu$ points towards the \emph{negative} direction.

\subsection{Bipartite Coxeter elements}
\label{sec:bipartite-coxeter}

\begin{definition}
	A \emph{bipartite Coxeter element} is a Coxeter element $w \in W$ for which there exist a set of simple reflections $S \subseteq R$ and a bipartition $S = S_0 \sqcup S_1$ of the Coxeter graph (i.e.\ the reflections in $S_0$ pairwise commute, and so do the reflections in $S_1$) such that $w = w_1w_0$, where $w_i$ is the product of the elements of $S_i$ in some order.
	\label{def:bipartite-coxeter-element}
\end{definition}

If the Coxeter graph of $W$ is a tree, then every Coxeter element $w \in W$ is a bipartite Coxeter element \cite[Corollary 7.6]{mccammond2015dual}.
In particular, this happens for all irreducible affine Coxeter groups except $\tilde A_n$.

Let $w$ be a bipartite Coxeter element, as in \Cref{def:bipartite-coxeter-element}.
Let $C_0$ be the (open) chamber of the Coxeter complex corresponding to the set $S$ of simple reflections so that the elements of $S = S_0 \sqcup S_1$ are the reflections with respect to the walls of $C_0$.
Let $F_i$ be the face of $C_0$ determined by the intersection of the hyperplanes of the reflections in $S_i$, and let $B_i$ be the affine span of $F_i$.
There is a unique pair of points $\p_i \in B_i$ that realize the minimum distance between $B_0$ and $B_1$.
Each $\p_i$ lies in the relative interior of the corresponding face $F_i$ \cite[Lemma 8.5]{mccammond2015dual}.
The line determined by $\p_0$ and $\p_1$ is exactly the Coxeter axis $\ell$ \cite[Proposition 8.8]{mccammond2015dual}, and it intersects $C_0$ \cite[Lemma 8.5]{mccammond2015dual}.

Each $w_i$ is an involution, and it restricts to a reflection on the Coxeter axis $\ell$ that fixes only $\p_i$ \cite[Lemma 8.7]{mccammond2015dual}.
Then $w_0$ and $w_1$ generate an infinite dihedral group that acts on $\ell$.
Using this action, we can extend the definitions of $F_i$, $B_i$, and $\p_i$ to arbitrary subscripts $i \in \Z$: let $F_{-i}$ (resp.\ $B_{-i}$, or $\p_{-i}$) be the image of $F_i$ (resp.\ $B_i$, or $\p_i$) under $w_0$, and let $F_{2-i}$ (resp.\ $B_{2-i}$, or $\p_{2-i}$) be the image of $F_i$ (resp.\ $B_i$, or $\p_i$) under $w_1$.
We obtain a sequence of equally spaced points $\p_i$ along the line $\ell$, with $p_{i+1} - p_i = \frac12 \mu$ (where $\mu$ is the shortest vector in $\Mov(w)$).
The \emph{axial chambers} are given by all possible images of the chamber $C_0$ under this dihedral group action.
Their vertices are called \emph{axial vertices}.

\begin{remark}
	\label{rmk:axial-orbits-bipartite}
	If $b$ is an axial vertex in the face $F_i$ (for some $i \in \Z$), then $w^j(b)$ is an axial vertex in $F_{i+2j}$.
	Therefore every axial chamber has exactly one vertex in the orbit $\{ w^j(b) \mid j \in \Z\}$.
\end{remark}

\begin{theorem}[{\cite[Theorem 8.10]{mccammond2015dual}}]
	Let $W$ be an irreducible affine Coxeter group not of type $\tilde A_n$, and $w$ one of its Coxeter elements.
	For every axial chamber $C$ there is a bipartite factorization $w = w_+w_-$, where $w_+$ (resp.\ $w_-$) is the product of the reflections with respect to the walls of $C$ that intersect $\ell$ above (resp.\ below) $C$.
	\label{thm:bipartite-coxeter}
\end{theorem}

If a hyperplane $H$ of a reflection of $W$ crosses the Coxeter axis $\ell$, then there is an index $i$ such that $H$ contains all of $F_i$, all but one vertex of $F_{i-1}$ and all but one vertex of $F_{i+1}$ \cite[Corollary 8.11]{mccammond2015dual}.

\begin{lemma}[{\cite[Lemma 9.3]{mccammond2015dual}}]
	Let $W$ be an irreducible affine Coxeter group not of type $\tilde A_n$, and $w$ one of its Coxeter elements.
	Let $H$ be the hyperplane of a vertical reflection $r$ in $W$ that intersects the Coxeter axis $\ell$ at the point $\p_i$.
	If $\b$ and $\b'$ are the unique vertices of $F_{i-1}$ and $F_{i+1}$ not contained in $H$, then $w$ sends $\b$ to $\b'$, $r$ swaps $\b$ and $\b'$,
	$rw$ fixes $\b$, and $wr$ fixes $\b'$.
	Moreover, the elliptic isometry $rw$ (resp.\ $wr$) is a Coxeter element for the finite parabolic subgroup of $W$ that fixes $\b$ (resp.\ $\b'$).
	\label{lemma:elliptic-coxeter-bipartite}
\end{lemma}

\begin{theorem}[{\cite[Propositions 9.4 and 9.5, Theorem 9.6]{mccammond2015dual}}]
	\label{thm:reflections-bipartite}
	Let $W$ be an irreducible affine Coxeter group not of type $\tilde A_n$, and $w$ one of its Coxeter elements.
	Every vertical reflection $r \in W$ is in $[1,w]$, and fixes many axial vertices.
	A horizontal reflection $r \in W$ is in $[1,w]$ if and only if it fixes at least one axial vertex.
\end{theorem}

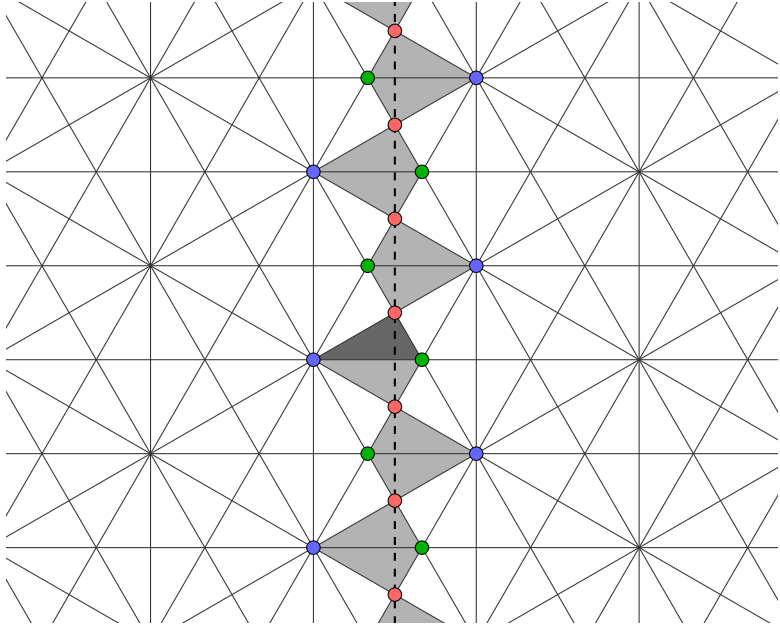
\begin{figure}[t]
	\newcommand*\rows{10}
	\begin{tikzpicture}[scale=1.25,
	extended line/.style={shorten >=-#1,shorten <=-#1}, extended line/.default=35cm,
	every node/.style={inner sep=1.8pt, circle, draw}]
	
	\clip (0.2, 1.2) rectangle (8.4, 7.8);
	
	\begin{scope}[every path/.style={black!30}]
	\foreach \i in {-3, ..., 3} {
		\fill ($({3*sqrt(3)}, 3+2*\i)$) -- ($({5/2*sqrt(3)}, 3.5+2*\i)$) -- ($({7/3*sqrt(3)}, 3+2*\i)$) -- ($({5/2*sqrt(3)}, 2.5+2*\i)$) -- cycle;
		
		\fill ($({2*sqrt(3)}, 4+2*\i)$) -- ($({5/2*sqrt(3)}, 4.5+2*\i)$) -- ($({8/3*sqrt(3)}, 4+2*\i)$) -- ($({5/2*sqrt(3)}, 3.5+2*\i)$) -- cycle;
	}
	\end{scope}
	\fill[black!60] ($({2*sqrt(3)}, 4)$) -- ($({5/2*sqrt(3)}, 4.5)$) -- ($({8/3*sqrt(3)}, 4)$) -- cycle;
	
	\begin{scope}[black!80]
	\foreach \row in {-\rows, ...,\rows} {
		\draw [extended line] ($\row*(0, 1)$) -- ($(1, 0)+\row*(0, 1)$);
		\draw [extended line] ($\row*({sqrt(3)},0)$) -- ($(0, 1)+\row*({sqrt(3)},0)$);
		\draw [extended line] ($\row*(0, 2)$) -- ($({sqrt(3)}, 3)+\row*(0, 2)$);
		\draw [extended line] ($\row*(0, 2)$) -- ($({sqrt(3)}, 1)+\row*(0, 2)$);
		\draw [extended line] ($\row*(0, 2)$) -- ($({sqrt(3)}, -3)+\row*(0, 2)$);
		\draw [extended line] ($\row*(0, 2)$) -- ($({sqrt(3)}, -1)+\row*(0, 2)$);
	}
	\end{scope}
	
	\draw [extended line, dashed, thick] ($({5/2*sqrt(3)}, 0)$) -- ($({5/2*sqrt(3)}, 1)$);
	
	\foreach \i in {-6, ..., 6} {
		\node[fill=blue!60] at ($({3*sqrt(3)}, 3+2*\i)$) {};
		\node[fill=blue!60] at ($({2*sqrt(3)}, 4+2*\i)$) {};
		\node[fill=black!30!green] at ($({7/3*sqrt(3)}, 3+2*\i)$) {};
		\node[fill=black!30!green] at ($({8/3*sqrt(3)}, 4+2*\i)$) {};
		\node[fill=red!60] at ($({5/2*sqrt(3)}, 3.5+\i)$) {};
	}
	\end{tikzpicture}
	\caption{Coxeter complex of type $\tilde G_2$ \cite[Figure 11]{mccammond2017artin}.}
	\label{fig:G2}
\end{figure}

\begin{example}[Case $\tilde G_2$]
	\Cref{fig:G2} shows the Coxeter complex of a Coxeter group of type $\tilde G_2$.
	Every Coxeter element $w$ is a glide reflection, whose glide axis is the Coxeter axis $\ell = \Min(w)$ (the dashed line).
	The axial chambers are shaded, and the axial vertices are colored, with one color for each $w$-orbit (see \Cref{rmk:axial-orbits-bipartite}).
	We use the notation of \cite[Definition 5.10]{mccammond2017artin} to indicate the reflections in $R_0 = R \cap [1,w]$: let $C_0$ be the darkly shaded chamber in \Cref{fig:G2}; let $p_0 \in \ell$ be below $C_0$, and $p_1 \in \ell$ above $C_0$; denote by $a_i$ (for $i \equiv 1$ mod 4) the reflection with respect to the line of slope $-\sqrt3$ passing through the point $p_i \in \ell$;
	similarly, denote by $b_j, c_k, d_i, e_j$ the vertical reflections with slopes $-\frac{1}{\sqrt3}, 0, \frac{1}{\sqrt3}, \sqrt3$, respectively (they are defined for $i \equiv 1$ mod 4, $j \equiv 3$ mod 4, and $k \equiv 0$ mod 2);
	finally, let $f$ and $f'$ be the two horizontal reflections of $[1,w]$.
	The walls of $C_0$ are the fixed lines of $a_1, d_1, c_0$.
	A bipartite factorization of $w$ is $w=a_1d_1c_0$.
	\label{ex:G2}
\end{example}

\subsection{Coxeter elements of type \texorpdfstring{$\tilde A_n$}{tilde An}}
\label{sec:A-coxeter}

If $W$ is a Coxeter group of type $\tilde A_n$, most of its Coxeter elements are not bipartite, and thus the theory of \Cref{sec:bipartite-coxeter} does not apply.
In this section, we derive a parallel theory and highlight the most important differences with the bipartite case.

As shown in \cite[Section 7]{mccammond2015dual}, every Coxeter element is geometrically equivalent to a \emph{$(p,q)$-bigon Coxeter element}, for a unique pair $(p,q)$ of positive integers such that $p \geq q$ and $p+q = n+1$.
Therefore there are exactly $\lfloor \frac{n+1}{2} \rfloor$ distinct equivalence classes of Coxeter elements.
For the explicit construction of $(p,q)$-bigon Coxeter elements, see Section \ref{sec:appendix-A} in the appendix.
The first four results in this section (\Cref{lemma:axis-A}, \Cref{thm:coxeter-A}, \Cref{prop:axial-point-A,prop:axial-orbits-A}) are verified in the appendix by explicit computation.

\begin{lemma}
	\label{lemma:axis-A}
	Let $W$ be a Coxeter group of type $\tilde A_n$, and $w$ one of its $(p,q)$-bigon Coxeter elements.
	The Coxeter axis $\ell$ is not contained in any reflection hyperplane of $W$, and it intersects the vertical hyperplanes in an infinite sequence of equally spaced points $\{p_i\}_{i \in \Z}$.
	More precisely we have $p_{i+1} - p_i = \frac{\gcd(p,q)}{p+q} \mu$, where $\mu$ is the shortest vector of $\Mov(w)$.
	In particular, $w(p_i) = p_j$ with $j = i + \frac{p+q}{\gcd(p,q)}$.
\end{lemma}

As in the bipartite case, the chambers that intersect the Coxeter axis $\ell$ are called \emph{axial chambers}, and the vertices of the axial chambers are called \emph{axial vertices}.
The following theorem is the analog of \Cref{thm:bipartite-coxeter}, and describes how axial chambers yield a factorization of $w$.

\begin{theorem}
	\label{thm:coxeter-A}
    Let $W$ be a Coxeter group of type $\tilde A_n$, and $w$ one of its $(p,q)$-bigon Coxeter elements with $p \geq q$.
    Fix an axial chamber $C$, and let $S_C \subseteq R$ be the set of the $n+1$ reflections with respect to the walls of $C$.
    Write $S_C = S^+ \sqcup S^- \sqcup S^\h$, where $S^+$ (resp.\ $S^-$) consists of the reflections that intersect the Coxeter axis $\ell$ above (resp.\ below) $C \cap \ell$, and $S^\h$ consists of the horizontal reflections.
    Then:
    \begin{enumerate}[(i)]
        \item $|S^+| = |S^-| = q$, and $|S^\h| = p-q$;
        \item the reflections in $S^+$ (resp.\ $S^-$) pairwise commute;
        \item $w$ can be written as a product of the reflections in $S_C$, where the reflections in $S^+$ come first, and the reflections in $S^-$ come last.
    \end{enumerate}
\end{theorem}

The next result describes how the vertical hyperplanes of the arrangement $\A$ intersect the Coxeter axis $\ell$.
Unlike the bipartite case, the vertical walls of an axial chamber $C$ can intersect $\ell$ outside of its closure $\bar C$.

\begin{proposition}
	\label{prop:axial-point-A}
	Let $W$ be a Coxeter group of type $\tilde A_n$, and $w$ one of its $(p,q)$-bigon Coxeter elements with $p \geq q$.
	\begin{enumerate}[(i)]
		\item Let $\a \in \ell$ be a point which is fixed by at least one vertical reflection of $W$.
		There are exactly $\gcd(p,q)$ vertical reflections of $W$ that fix $\a$, and they pairwise commute.
		
		\item Let $\{p_i\}_{i \in \Z}$ be the sequence of points of \Cref{lemma:axis-A}, and let $C$ be an axial chamber that intersects $\ell$ between $p_i$ and $p_{i+1}$.
		A vertical hyperplane of $\A$ is a wall of $C$ if and only if it intersects $\ell$ in one of the $2m$ consecutive points $p_{i-m+1}, p_{i-m+2}, \dotsc, p_{i+m}$, where $m = \frac{q}{\gcd(p,q)}$.
	\end{enumerate}
	
\end{proposition}

The following result gives some insight into the geometry of axial vertices.
The first part is the $\tilde A_n$ analog of \Cref{rmk:axial-orbits-bipartite}.

\begin{proposition}
	\label{prop:axial-orbits-A}
	Let $W$ be a Coxeter group of type $\tilde A_n$, and $w$ one of its $(p,q)$-bigon Coxeter elements.
	Let $b$ be an axial vertex.
	\begin{enumerate}[(i)]
		\item Every axial chamber has exactly one vertex in the set $\{ w^j(b) \mid j \in \Z \}$.
		\item There are exactly $\frac{p+q}{\gcd(p,q)}$ axial chambers having $b$ as one of their vertices, and they are consecutive (i.e.\ the union of their closures intersects the Coxeter axis $\ell$ in a connected set).
	\end{enumerate}
\end{proposition}

\begin{remark}[Bipartite case]
	If $n$ is odd and $p=q=\frac{n+1}{2}$, then $w$ is a bipartite Coxeter element, and we recover the results of \Cref{sec:bipartite-coxeter}.
	In particular, part (iii) of \Cref{thm:coxeter-A} reduces to the bipartite factorization $w = w_+w_-$ of \Cref{thm:bipartite-coxeter}.
\end{remark}

\begin{figure}
	\begin{tikzpicture}[scale=1.4,
	extended line/.style={shorten >=-#1,shorten <=-#1}, extended line/.default=35cm,
	every node/.style={inner sep=1.8pt, circle, draw}]
	
	\clip (-3.23, -3.6) rectangle (4.09, 2.3);
	
	\begin{scope}[every path/.style={black!30}]
	\fill (0, -10) -- (0, 10) -- ($({0.5*sqrt(3)}, 10)$) -- ($({0.5*sqrt(3)}, -10)$) -- cycle;
	\end{scope}
	
	\fill[black!60] ($(0, 0)$) -- ($(0, -1)$) -- ($({0.5*sqrt(3)}, -0.5)$) -- cycle;
	
	\begin{scope}[black!80]
	\newcommand{\rows}{5}
	\foreach \row in {-\rows, ...,\rows} {
		\draw [extended line] ($\row*({0.5*sqrt(3)}, 0.5)$) -- ($(0,\rows)+\row*({0.5*sqrt(3)}, -0.5)$);
		\draw [extended line] ($\row*(0, 1)$) -- ($({\rows/2*sqrt(3)}, \rows/2)+\row*({-0.5*sqrt(3)}, 0.5)$);
		\draw [extended line] ($\row*(0, 1)$) -- ($({-\rows/2*sqrt(3)}, \rows/2)+\row*({0.5*sqrt(3)}, 0.5)$);
	}
	\end{scope}
	
	\draw [extended line, dashed, thick] ($({0.25*sqrt(3)}, 0)$) -- ($({0.25*sqrt(3)}, 1)$);
	
	\foreach \i in {-6, ..., 6} {
		\node[fill=blue!60] at ($(0, 3*\i)$) {};
		\node[fill=blue!60] at ($({0.5*sqrt(3)}, 1.5+3*\i)$) {};
		\node[fill=black!30!green] at ($(0, 1+3*\i)$) {};
		\node[fill=black!30!green] at ($({0.5*sqrt(3)}, -0.5+3*\i)$) {};
		\node[fill=red!60] at ($(0, -1+3*\i)$) {};
		\node[fill=red!60] at ($({0.5*sqrt(3)}, 0.5+3*\i)$) {};
	}
	\end{tikzpicture}
	\caption{Coxeter complex of type $\tilde A_2$.}
	\label{fig:A2}
\end{figure}
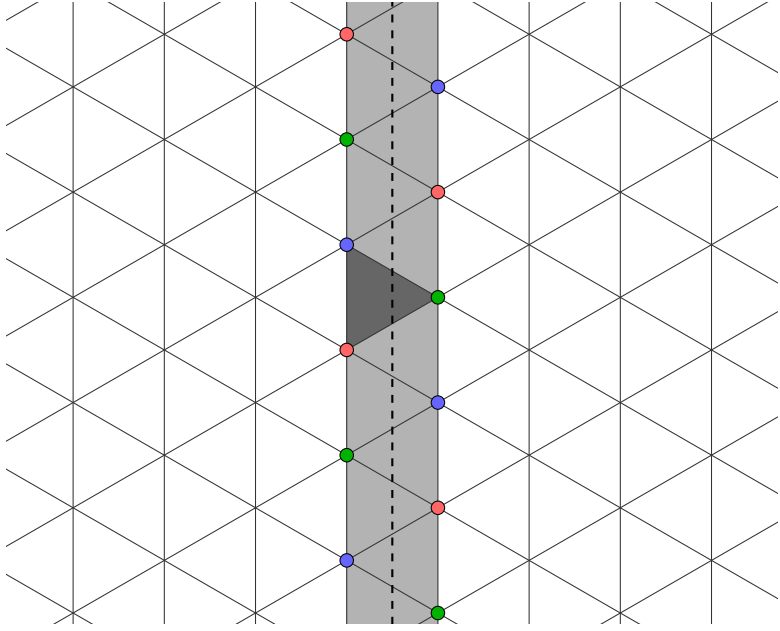

\begin{example}[Case $\tilde A_2$]
	\Cref{fig:A2} shows the Coxeter complex of a Coxeter group of type $\tilde A_2$.
	Every $(2,1)$-bigon Coxeter element $w$ is a glide reflection, whose glide axis is the Coxeter axis $\ell = \Min(w)$ (the dashed line).
	As in \Cref{fig:G2}, the axial chambers are shaded, and the axial vertices are colored, with one color for each $w$-orbit (see part (i) of \Cref{prop:axial-orbits-A}).
	We use a notation similar to \Cref{ex:G2}: let $C_0$ be the darkly shaded chamber; let $p_0 \in \ell$ be the point immediately below $C_0$, and $p_1 \in \ell$ the point immediately above $C_0$;
	denote by $a_i$ the vertical reflection that fixes $p_i$ for $i \equiv 1$ mod 2; denote by $c_j$ the vertical reflection that fixes $p_j$ for $j \equiv 0$ mod 2; finally, let $b$ and $b'$ be the two horizontal reflections adjacent to the Coxeter axis $\ell$ (by \Cref{thm:reflections-A} below, these are precisely the horizontal reflections of $[1,w]$).
	The walls of $C_0$ are the fixed lines of $a_1, b, c_0$, and $w = a_1bc_0$.
	\label{ex:A2}
\end{example}

If $H$ is the fixed hyperplane of some vertical reflection of $W$, there is a well-defined axial chamber which is \emph{immediately above} $H$, and one which is \emph{immediately below} $H$.
These are the only two chambers that intersect a small neighborhood of $H \cap \ell$ in $\ell$.
Denote them by $C_H^+$ and $C_H^-$, respectively.
By part (i) of \Cref{prop:axial-point-A}, all the vertical reflections fixing $H \cap \ell$ pairwise commute.
Therefore they are all walls of $C_H^+$ and $C_H^-$.
This proves the following analog of \cite[Corollary 8.11]{mccammond2015dual}.

\begin{provedcorollary}
	Let $W$ be a Coxeter group of type $\tilde A_n$, and $w$ a Coxeter element of $W$.
	If the fixed hyperplane $H$ of a reflection in $W$ crosses the axis $\ell$, then $H$ is the affine span of a facet of an axial chamber.
\end{provedcorollary}


The following is the $\tilde A_n$ analog of \Cref{lemma:elliptic-coxeter-bipartite}.

\begin{lemma}
	\label{lemma:elliptic-coxeter-A}
	Let $W$ be a Coxeter group of type $\tilde A_n$, and $w$ one of its Coxeter elements.
	Let $r$ be a vertical reflection in $W$, $H = \Fix(r)$, and $\b^-$ (resp.\ $\b^+$) the unique vertex of $C_H^-$ (resp.\ $C_H^+$) which is not in $H$.
	Then the elliptic isometry $rw$ (resp.\ $wr$) is a Coxeter element for the finite parabolic subgroup of $W$ that fixes $\b^-$ (resp.\ $\b^+$).
\end{lemma}

\begin{proof}
	Apply \Cref{thm:coxeter-A} to the axial chamber $C = C_H^-$.
	Since $r \in S^+$, the Coxeter element $w$ can be written as a product $rr_2\dotsm r_{n+1}$, where $r_2, \dotsc, r_{n+1}$ are the reflections with respect to the other walls $H_2, \dotsc, H_{n+1}$ of $C_H^-$.
	The walls $H_2, \dotsc, H_{n+1}$ bound a chamber of the finite parabolic subgroup of $W$ that fixes $\b^-$, so $rw = r_2\dotsm r_{n+1}$ is a Coxeter element for this subgroup.
	Similarly, $wr$ is a Coxeter element for the finite parabolic subgroup that fixes $\b^+$.
\end{proof}

\begin{remark}
	\label{rmk:elliptic-coxeter-A-any-chamber}
	In the proof of \Cref{lemma:elliptic-coxeter-A}, the chamber $C_H^-$ (resp.\ $C_H^+$) can be replaced with any axial chamber $C$ that intersects $\ell$ below (resp.\ above) $H \cap \ell$ and such that $H$ is a wall of $C$.
	However, the statement of \Cref{lemma:elliptic-coxeter-A} makes it clear that at least one such chamber $C$ exists, so $rw$ and $wr$ are indeed parabolic Coxeter elements.
\end{remark}

\begin{lemma}
	\label{lemma:vertical-opposite-wall}
	Let $W$ be a Coxeter group of type $\tilde A_n$, and $w$ one of its Coxeter elements.
	For every axial vertex $\b$, there exists an axial chamber $C$ such that $\b$ is a vertex of $C$ and the wall of $C$ opposite to $\b$ is vertical.
\end{lemma}

\begin{proof}
	Let $C$ be the lowest axial chamber such that $\b$ is a vertex of $C$.
	Suppose by contradiction that the wall of $C$ opposite to $\b$ is horizontal.
	Then $\b$ is fixed by all the reflections with respect to the vertical walls of $C$.
	Let $\a$ be the lowest point of $\bar C \cap \ell$, and let $S_\a$ be the set of the (vertical) reflections of $W$ that fix $\a$.
	By part (i) of \Cref{prop:axial-point-A}, the reflections in $S_\a$ pairwise commute.
	Then their product fixes $\b$ and sends $C$ to an axial chamber $C'$ which is below $C$.
	This is a contradiction.
\end{proof}

We are now ready to prove the $\tilde A_n$ analog of \Cref{thm:reflections-bipartite}.

\begin{theorem}
	\label{thm:reflections-A}
	Let $W$ be a Coxeter group of type $\tilde A_n$, and $w$ one of its Coxeter elements.
	Every vertical reflection $r \in W$ is in $[1,w]$, and fixes many axial vertices.
	A horizontal reflection $r \in W$ is in $[1,w]$ if and only if it fixes at least one axial vertex.
\end{theorem}

\begin{proof}
	For the first part, it is enough to apply \Cref{thm:coxeter-A} to $C_H^+$ or $C_H^-$, where $H = \Fix(r)$.
	
	For the second part, suppose that $r$ is a horizontal reflection in $[1,w]$, so there is a factorization $w = r_1r_2\dotsm r_n r$ that ends with $r$.
	Since $w$ is vertical, at least one of the $r_i$ is vertical.
	By \Cref{lemma:hurwitz-action} we can move this reflection to the beginning, and thus assume that $r_1$ is vertical.
	Then $r \leq w' = r_1w$, and $w'$ is an elliptic isometry that fixes an axial vertex $\b^-$ by \Cref{lemma:elliptic-coxeter-A}.
	Since $\Fix(w') \subseteq \Fix(r)$ (\Cref{lemma:model-poset}), we have that $r$ fixes $\b^-$.
	
	On the other hand, suppose that $r \in W$ is a (horizontal) reflection that fixes some axial vertex $\b$.
	By \Cref{lemma:vertical-opposite-wall}, we have that $\b$ is a vertex of some axial chamber $C$ such that the wall $H'$ of $C$ opposite to $\b$ is vertical.
	Let $r'$ be the reflection with respect to $H'$.
	By \Cref{lemma:elliptic-coxeter-A} and \Cref{rmk:elliptic-coxeter-A-any-chamber}, one of $r'w$ and $wr'$ is a Coxeter element for the finite parabolic subgroup of $W$ that fixes $\b$.
	Recall that every reflection in a finite Coxeter group occurs in some minimal length factorization of any of its Coxeter elements \cite[Lemma 1.3.3]{bessis2003dual}.
	Therefore $r \leq r'w$ or $r \leq wr'$, and so $r \leq w$.
\end{proof}

\subsection{Isometries below an affine Coxeter element}
\label{sec:isometries-below}

Now that the case $\tilde A_n$ is well understood, we turn to the general case of an irreducible affine Coxeter group and prove a few more results about the elements of the interval $[1,w]$.

\begin{lemma}
	\label{lemma:elliptic-coxeter-element}
	Let $W$ be an irreducible affine Coxeter group, and $w$ one of its Coxeter elements.
	For every axial vertex $\b$, there exists a unique element $w_\b \in [1,w]$ which is a Coxeter element for the finite parabolic subgroup of $W$ that fixes $\b$.
\end{lemma}

\begin{proof}
	Let $C$ be an axial chamber such that $\b$ is a vertex of $C$.
	Denote by $r, r_1,\dotsc,r_n$ the reflections with respect to the walls of $C$, where $r$ is the reflection that does not fix $\b$.
	By \Cref{thm:coxeter-A} (for the case $\tilde A_n$) and \Cref{thm:bipartite-coxeter} (for the other cases), the Coxeter element $w$ can be written as a product of the reflections $r, r_1,\dotsc,r_n$ in some order.
	Remove $r$ from this factorization, and let $w_\b$ be the product of the remaining reflections in the same relative order.
    Then $w_\b \in [1, w]$ by \Cref{lemma:hurwitz-action}.
	By construction, $w_\b$ is a Coxeter element for the finite parabolic subgroup of $W$ that fixes $\b$.
	Every such element $w_\b' \in [1,w]$ has a fix-set equal to $\{\b\}$, so uniqueness follows from \Cref{lemma:model-poset}.
\end{proof}

\begin{lemma}
	Let $W$ be an irreducible affine Coxeter group, and $w$ one of its Coxeter elements.
	For every elliptic element $u \in [1,w]$, there exists an axial vertex $\b$ such that $u \leq w_\b$, where $w_\b$ is the unique element of $[1,w]$ which is a Coxeter element for the finite parabolic subgroup that fixes $\b$ (see \Cref{lemma:elliptic-coxeter-element}).
	In particular, $u$ fixes at least one axial vertex $\b$.
	\label{lemma:fixed-axial-vertex}
\end{lemma}

\begin{proof}
	Let $v$ be the right complement of $u$, so that $uv = w$.
	Let $v = r_1\dotsm r_m$ be a minimal length factorization of $v$ as a product of reflections.
	Since $u$ is elliptic, $v$ is vertical and therefore at least one $r_i$ is a vertical reflection.
	By \Cref{lemma:hurwitz-action} we can move this reflection to the end, and thus assume that $r_m$ is vertical.
	By \Cref{lemma:elliptic-coxeter-A} (for the case $\tilde A_n$) and \Cref{lemma:elliptic-coxeter-bipartite} (for the other cases), $wr_m = ur_1\dotsm r_{m-1}$ is a  Coxeter element for the finite parabolic subgroup of $W$ that fixes an axial vertex $\b$.
	By construction, $u \leq w r_m$.
\end{proof}

Given an element $u \in [1,w]$, denote by $W_u$ the subgroup of $W$ generated by the reflections in $[1,u]$ (or, equivalently, by all the elements of $[1,u]$).
Then $W_u$ is a Coxeter group having $R \cap W_u$ as its set of reflections, by \cite[Theorem 8.2]{humphreys1992reflection} (see also \cite{deodhar1989note, dyer1990reflection}).
Denote by $\A_u = \{ \Fix(r) \mid r \in R \cap W_u \} \subseteq \A$ the reflection arrangement associated with the Coxeter group $W_u$.

\begin{lemma}[Hyperbolic-horizontal decomposition]
	\label{lemma:hyperbolic-decomposition}
	Let $W$ be an irreducible affine Coxeter group, $w$ one of its Coxeter elements, and $u \in [1,w]$ a hyperbolic element.
	There exists a unique decomposition $u=u'h$ such that:
	\begin{enumerate}
		\item $u', h \in [1,u]$, $u'$ is hyperbolic, $h$ is horizontal elliptic, and $l(u) = l(u') + l(h)$;
		\item $W_u$ is the direct product of the Coxeter subgroups $W_{u'}$ and $W_{h}$;
		\item $W_{u'}$ is an irreducible affine Coxeter subgroup;
		\item $W_h$ is a finite horizontal Coxeter subgroup;
		\item $[1,u] = [1,u'] \times [1,h]$.
	\end{enumerate}
\end{lemma}

\begin{proof}
	Decompose the Coxeter group $W_u$ as a direct product of irreducible subgroups: $W_u = W_1 \times \dotsb \times W_t$.
	Then $u \in W_u$ can be written uniquely as $u = u_1\dotsm u_t$ with $u_i \in W_i$.
	
	Since $u$ is hyperbolic, its right complement $v$ is horizontal elliptic.
	Let $u = r_1\dotsm r_m$ be a minimal length factorization of $u$ as a product of reflections.
	Since $u$ is vertical, at least one $r_i$ is a vertical reflection.
	By \Cref{lemma:hurwitz-action} we can move this reflection to the beginning, and thus assume that $r_1$ is vertical.
	Therefore its right complement $r_2\dotsm r_m v$ is vertical elliptic, and in particular $r_{2}\dotsm r_m$ is elliptic.
	Each reflection $r_i$ belongs to one of the irreducible components $W_j$.
	Without loss of generality, assume that $r_1, \dotsc, r_k \in W_1$ and $r_{k+1}, \dotsc, r_m \in W_2 \times \dotsb \times W_t$, for some $k \in \{1, \dotsc, m\}$.
	By uniqueness of the decomposition $u=u_1\dotsm u_t$, we have that $u_1 = r_1\dotsm r_k$ and $u_2\dotsm u_t = r_{k+1}\dotsm r_m$.
	In particular, both $u_1$ and $u_2\dotsm u_t$ belong to the interval $[1,u]$, and $l(u) = l(u_1) + l(u_2\dotsm u_t) = l(u_1) + l(u_2) + \dotsb + l(u_t)$.
	
	Denote by $\alpha_1, \dotsc, \alpha_m$ the roots corresponding to $r_1, \dotsc, r_m$.
	Since $u_2\dotsm u_t = r_{k+1} \dotsm r_m$ is elliptic, the roots $\alpha_{k+1}, \dotsc, \alpha_m$ are linearly independent by \cite[Lemma 6.4]{brady2015factoring}.
	Suppose by contradiction that $u_1 = r_1\dotsm r_k$ is also elliptic. Then for the same reason the roots $\alpha_1,\dotsc, \alpha_k$ are linearly independent.
	Therefore the roots $\alpha_1,\dotsc, \alpha_m$ are linearly independent, so $u$ is elliptic by \cite[Lemma 6.4]{brady2015factoring}, and this is a contradiction.
	We deduce that $u_1$ is hyperbolic.
	This implies that its right complement $r_{k+1}\dotsm r_m v$ is horizontal elliptic, so the reflections $r_{k+1}, \dotsc, r_m$ are horizontal.
	Then $u_2\dotsm u_t = r_{k+1}\dotsm r_m$ is horizontal elliptic.
	
	Recall that $W_u$ is generated by the reflections in $[1,u]$.
	Each reflection $r \in [1,u]$ belongs to exactly one irreducible factor $W_i$, and so is part of a minimal length factorization of $u_i$, which implies that $r \in [1,u_i]$.
	Therefore $W_i$ is equal to the subgroup $W_{u_i} \subseteq W$ generated by the reflections in $[1,u_i]$.
	For all $i \geq 2$ we have that $u_i$ is horizontal elliptic, so all the elements of $W_{u_i}$ are horizontal.
	Then the irreducible factor $W_1 = W_{u_1}$ is uniquely determined as the only factor which contains at least one vertical reflection.
	
	Define $u' = u_1$ and $h = u_2\dotsm u_t$.
	Notice that $W_h$ is the group generated by the reflections in $[1,h]$, i.e.\ the reflections in $[1,u_2] \cup \dotsb \cup [1,u_t]$, so $W_h = W_{u_2} \times \dotsb \times W_{u_t} = W_2 \times \dotsb \times W_t$.
	Therefore $W_u = W_{u'} \times W_h$.
	This, together with the fact that $W_h$ is horizontal and $W_{u'}$ is irreducible, is enough to ensure that the decomposition $u = u'h$ is unique.
	Since $W_u = W_{u'} \times W_h$, we have $[1,u] = [1,u'] \times [1,h]$.
	
	Recall that $u' = r_1\dotsm r_k$ is hyperbolic, so $\Fix(r_1) \cap \dotsb \cap \Fix(r_k) = \emptyset$.
	Therefore the Coxeter subgroup $W_{u'}$ is infinite.
	By construction, $W_{u'}$ is also irreducible, so it must be an irreducible affine Coxeter group.
\end{proof}

We will refer to the decomposition $u = u'h$ of \Cref{lemma:hyperbolic-decomposition} as the \emph{hyperbolic-horizontal decomposition} of $u$.

The following technical lemma is proved in the appendix for the four infinite families and was checked by computer for the exceptional cases (see \cite{program}).

\begin{lemma}
	\label{lemma:hyperbolic-coxeter}
	Let $W$ be an irreducible affine Coxeter group, $w$ one of its Coxeter elements, and $u \in [1,w]$ a hyperbolic element such that $W_u$ is irreducible.
	Let $\a$ be a point of $\ell$ that does not lie on any hyperplane of $\A_u$, and let $C$ be the chamber of $\A_u$ containing $\a$.
	Then $C$ has exactly $l(u)$ walls, and $u$ can be written as the product of the reflections with respect to the walls of $C$ in the following order:
	\begin{itemize}
		\item first there are the vertical reflections that fix a point of $\ell$ above $\a$, and $r$ comes before $r'$ if $\Fix(r) \cap \ell$ is below $\Fix(r') \cap \ell$;
		\item then there are the horizontal reflections, in some order;
		\item finally there are the vertical reflections that fix a point of $\ell$ below $\a$, and again $r$ comes before $r'$ if $\Fix(r) \cap \ell$ is below $\Fix(r') \cap \ell$.
	\end{itemize}
\end{lemma}

The conclusion of \Cref{lemma:hyperbolic-coxeter} seems to hold for all hyperbolic elements $u \in [1,w]$, without the irreducibility hypothesis.
In addition, for the case $\tilde G_2$ and for the four infinite families ($\tilde A_n$, $\tilde B_n$, $\tilde C_n$, and $\tilde D_n$), the vertical walls of $C$ that fix a point of $\ell$ above $\a$ (resp.\ below $\a$) pairwise commute.
However, a computer check shows that this is not true for the exceptional cases $\tilde F_4$, $\tilde E_6$, $\tilde E_7$, and $\tilde E_8$ (see \cite{program}).
Notice also that in general $\Min(w) \nsubseteq \Min(u)$, so a point $\a \in \ell = \Min(w)$ is usually not on the Coxeter axis of $u$ (otherwise \Cref{lemma:hyperbolic-coxeter} would follow easily from \Cref{thm:bipartite-coxeter,thm:coxeter-A} applied to $W_u$).

We can now prove the affine analog of \cite[Lemma 1.4.3]{bessis2003dual}.

\begin{theorem}
	\label{thm:coxeter-elements}
	Let $W$ be an irreducible affine Coxeter group, and $w$ one of its Coxeter elements.
	Every element $u \in [1,w]$ is a Coxeter element of $W_u$.
	In addition:
	\begin{enumerate}[(i)]
		\item if $u$ is elliptic, then $u$ is a parabolic Coxeter element;
		\item if $u$ is hyperbolic and $u=u'h$ is its hyperbolic-horizontal decomposition, then $u'$ is a Coxeter element of $W_{u'}$ and $h$ is a Coxeter element of the parabolic subgroup $W_h$ of $W_u$.
	\end{enumerate}
\end{theorem}

\begin{proof}
	Suppose that $u$ is elliptic.
	By \Cref{lemma:fixed-axial-vertex} there is an axial vertex $\b$ such that $u \leq w_\b$, where $w_\b$ is the parabolic Coxeter element of \Cref{lemma:elliptic-coxeter-element}.
	By \cite[Lemma 1.4.3]{bessis2003dual} applied to the finite parabolic subgroup that fixes $\b$, we have that $u$ is a Coxeter element of $W_u$, and $W_u$ is a parabolic subgroup of $W$.

	Suppose now that $u$ is hyperbolic, and let $u=u'h$ be its hyperbolic-horizontal decomposition.
	Since $W_{u'}$ is irreducible, by \Cref{lemma:hyperbolic-coxeter} we immediately have that $u'$ is a Coxeter element of $W_{u'}$.
	Also, since $h$ is horizontal, we have already proved that $h$ is a Coxeter element of the parabolic subgroup $W_h$ (notice that $W_h$ is a parabolic subgroup of both $W$ and $W_u$).
	Then $u'h$ is a Coxeter element of $W_u = W_{u'} \times W_h$.
\end{proof}

In the case $\tilde C_n$, the first part of \Cref{thm:coxeter-elements} was already noted by Digne \cite[Remark 7.2]{digne2012garside}.

\begin{remark}
	If $W$ is an irreducible affine Coxeter group, then all its proper parabolic subgroups are finite.
	Therefore, if $u \neq w$ is a hyperbolic element, then $W_u$ is not a parabolic subgroup (and $u$ is not a parabolic Coxeter element).
\end{remark}

\begin{example}[Hyperbolic elements in the case $\tilde C_3$]
	If $W$ is a Coxeter group of type $\tilde C_3$, the interval $[1,w]$ has $3$ hyperbolic elements of length $2$ (the translations), $6$ of length $3$, and $1$ of length $4$ ($w$ itself).
	They are the complements of the horizontal elements, which are explicitly described in \Cref{ex:horizontal-A2} below.
	Among the $6$ hyperbolic elements $u \in [1,w]$ of length $3$, in $3$ cases $W_u \cong W_{\tilde A_1} \times W_{A_1}$ (so the hyperbolic-horizontal decomposition of $u$ has a non-trivial horizontal factor), whereas in the remaining three cases $W_u \cong W_{\tilde C_2}$.
	See Section \ref{sec:appendix-C} for an explicit computation of the hyperbolic elements in the case $\tilde C_n$.
	\label{ex:C3}
\end{example}

\subsection{Horizontal components}
\label{sec:horizontal-components}

We now describe the geometry of the irreducible horizontal components of an affine Coxeter group.
The ideas for this section are mostly already present in \cite{mccammond2017artin}, but we find it convenient to write them down more explicitly.

As in \Cref{sec:isometries-below}, let $W$ be an irreducible affine Coxeter group acting on $E = \R^n$ where $n$ is the rank of $W$, and fix a Coxeter element $w$.
Let $\Phi$ be the root system of $W$, and let $\Phi_\h \subseteq \Phi$ be the horizontal root system.

As shown in \cite[Definition 6.1]{mccammond2017artin}, there exists at least one \emph{horizontal factorization} $w = th$ where $t \in [1,w]$ is a translation, $h \in [1,w]$ is a horizontal isometry of reflection length $n-1$, and every horizontal reflection of $W$ is parallel to some reflection of the parabolic Coxeter subgroup $W_h$.
In other words, the root system of $W_h$ coincides with the horizontal root system $\Phi_\h$ of $W$.

The horizontal root system $\Phi_\h$ decomposes as a disjoint union of orthogonal irreducible root systems of type $A$ \cite[Section 6]{mccammond2017artin}: $\Phi_\h = \Phi_{1} \sqcup \Phi_{2} \sqcup \dotsb \sqcup \Phi_{k}$, where $\Phi_i$ is a root system of type $A_{n_i}$, and $n_1+n_2+\dotsb+n_k = n-1$ (see \Cref{table:horizontal-root-systems}).
Accordingly, the horizontal isometry $h$ decomposes as a product $h = h_1h_2\dotsm h_k$ where $h_i$ belongs to the $i$-th irreducible horizontal component and has reflection length $n_i$.
Then $\Phi_i$ is the root system of the irreducible parabolic Coxeter subgroup $W_{h_i}$, and $h_i$ is a Coxeter element of $W_{h_i}$ by \Cref{thm:coxeter-elements}.

We now focus on a single irreducible horizontal component $\Phi_i$ and describe the geometry of the associated horizontal reflections.
Let $m = n_i$ be the rank of $\Phi_i$.
Since $\Phi_i$ is a root system of type $A_{m}$, the reflections of $W$ in the directions of $\Phi_i$ generate a Coxeter subgroup $W_i \subseteq W$ of type $\tilde A_m$.
Denote by $\A_i$ the hyperplane arrangement associated with $W_i$ (this is a subarrangement of the hyperplane arrangement $\A$ associated with $W$).

The Coxeter axis $\ell$ does not intersect any horizontal hyperplane, and so it is contained in one chamber $C_i$ of the arrangement $\A_i$.
We call $C_i$ the $i$-th \emph{horizontal prism}.
Like every chamber of $\A_i$, it has $m+1$ faces of minimal dimension $n-m$.
We call them the \emph{minimal faces} of $C_i$.

Since the Coxeter axis $\ell$ is contained in $C_i$, every axial chamber of $\A$ is also contained in $C_i$.
As a consequence, all axial vertices are contained in $\bar C_i$.
The horizontal isometry $h_i$ has reflection length $m$, so $\dim \Fix(h_i) = n-m$, and $\Fix(h_i)$ is the intersection of $m$ hyperplanes of $\A_i$.
Since $h_i \in [1,w]$, by \Cref{lemma:fixed-axial-vertex} we have that $\Fix(h_i)$ contains at least one axial vertex.
Therefore $\Fix(h_i)$ is one of the minimal faces of $C_i$.

\begin{lemma}
	\label{lemma:cycle}
	Let $W$ be an irreducible affine Coxeter group, and $w$ one of its Coxeter elements.
	Every horizontal prism $C_i$ is fixed by $w$ (as a set).
	In addition, $w$ acts transitively on the set of walls of $C_i$ and on the set of minimal faces of $C_i$.
\end{lemma}

\begin{proof}
	From the factorization $w=th_1h_2 \dotsm h_k$ we see that $w$ sends hyperplanes of $\A_i$ to hyperplanes of $\A_i$.
	Also, it fixes the Coxeter axis $\ell$ (as a set).
	Therefore it fixes $C_i$ (as a set), it permutes the walls of $C_i$, and it permutes the minimal faces of $C_i$.
	The action of $w$ on the walls of $C_i$ is the same as the action of $th_i$, because $h_j$ fixes every hyperplane of $\A_i$ for all $j \neq i$.
	
	Let $H_1, \dotsc, H_{m+1}$ be the walls of $C_i$, and for every $j \in \{1, \dotsc, m+1\}$ let $\alpha_j$ be the root of $H_j$ that points from $H_j$ towards the half-space containing $C_i$.
	The linear part of $th_i$ (which coincides with the linear part of $h_i$) permutes $\alpha_1, \dotsc, \alpha_{m+1}$.
	Recall that $h_i$ is a Coxeter element of $W_{h_i}$, which is a Coxeter group of type $A_m$ with root system $\Phi_i$.
	In a finite Coxeter group, the orbits of roots under the action of a Coxeter element are known to have cardinality equal to the Coxeter number, which is $m+1$ for a Coxeter group of type $A_m$.
	Therefore $h_i$ transitively permutes the roots $\alpha_1, \dotsc, \alpha_{m+1}$.
	Then $w$ transitively permutes the walls of $C_i$.
	Every minimal face of $C_i$ is opposite to exactly one wall of $C_i$, so the same conclusion holds for the minimal faces.
\end{proof}

Let $\varphi\colon [1,w] \to [1,w]$ be the conjugation by the Coxeter element $w$: $\varphi(u) = w^{-1}uw$.

\begin{proposition}
	\label{prop:horizontal-axial}
	Let $W$ be an irreducible affine Coxeter group, and $w$ one of its Coxeter elements.
	Every minimal face of the $i$-th horizontal prism $C_i$ is the fixed set of $\varphi^p(h_i)$ for some $p \in \{0, \dotsc, m\}$, and it contains at least one axial vertex.
	In addition, the elements $\varphi^p(h_i)$ for $p \in \{0, \dotsc, m\}$ are the maximal elements of the subposet $[1,w] \cap W_i \subseteq [1,w]$, and they all have the same linear part.
\end{proposition}

\begin{proof}
	We have that $\Fix(h_i)$ is a minimal face of $C_i$.
	Then the first part follows from \Cref{lemma:cycle}.
	Since $\varphi^p(h_i) \in [1,w]$, its fixed set contains at least one axial vertex by \Cref{lemma:fixed-axial-vertex}.
	
	Every element of $[1,w] \cap W_i$ of reflection length $m$ must have a fixed set equal to a minimal face of $C_i$.
	By \Cref{lemma:model-poset}, there can be at most one such element for every minimal face of $C_i$.
	Therefore, the elements $\varphi^p(h_i)$ for $p \in \{0, \dotsc, m\}$ are the only elements of reflection length $m$ in $[1,w] \cap W_i$.
	Every other element $u \in [1,w] \cap W_i$ has a reflection length strictly smaller than $m$.

	We want to show that, for every $u \in [1,w] \cap W_i$, we have $u \leq \varphi^p(h_i)$ for some $p$.
	Let $v$ be the left complement of $u$, so that $vu = w$.
	Then $v$ is hyperbolic.
	Let $v = v'h'$ be the hyperbolic-horizontal decomposition of $v$ (see \Cref{lemma:hyperbolic-decomposition}).
	Recall that the Coxeter subgroup $W_{v'}$ generated by $[1,v']$ is an irreducible affine Coxeter group, and $v'$ is one of its Coxeter elements by \Cref{thm:coxeter-elements}.
	Then the interval $[1,v']$ contains at least one translation $t'$ (this follows for example from the existence of a horizontal factorization of $v'$ in $W_{v'}$).
	Therefore we can write $v' = t'v''$, with $l(v') = l(t') + l(v'')$.
	Putting everything together, we get the factorization $w = t'v''h'u$, with $l(t') + l(v'') + l(h') + l(u) = l(w) = n+1$.
	Since $t'$ is a translation, and so has length $l(t') = 2$, its right complement $\bar h = v''h'u$ is a horizontal element of length $l(\bar h) = n-1$.
	In addition, $u \leq \bar h$ because $u$ is part of a minimal length factorization of $\bar h$.
	Write $\bar h = \bar h_1 \bar h_2 \dotsm \bar h_k$, with $\bar h_j \in  [1,w] \cap W_j$.
	We have $l(\bar h_j) \leq n_j$ for all $j$, and $n_1 + \dotsb + n_k = n-1$, so $l(\bar h_j) = n_j$.
	In particular $l(\bar h_i) = n_i = m$, so $\bar h_i = \varphi^p(h_i)$ for some $p \in \{0, \dotsc, m\}$.
	Since $u \leq \bar h$ and $u \in [1,w] \cap W_i$, we have $u \leq \bar h_i = \varphi^p(h_i)$.
	
	Finally, we want to show that the linear part of $\varphi^p(h_i)$ is equal to the linear part of $h_i$.
	This follows from the factorization $w = th_1h_2\dotsm h_k$, together with the fact that $h_j$ commutes with $h_i$ for all $j \neq i$, and that the linear part of the translation $t$ is the identity.
\end{proof}

\begin{corollary}
	\label{cor:two-horizontal-reflections}
	Let $W$ be an irreducible affine Coxeter group, and $w$ one of its Coxeter elements.
	For every horizontal root $\alpha \in \Phi_\h$, the interval $[1,w]$ contains exactly two reflections in the direction of $\alpha$, namely those determined by adjacent hyperplanes which contain the Coxeter axis $\ell$ between them.
\end{corollary}

\begin{proof}
	Let $\Phi_i$ be the irreducible component of $\Phi_\h$ containing $\alpha$.
	A hyperplane of $\A_i$ yields a reflection in $[1,w]$ if and only if it contains an axial vertex, which happens if and only if it contains a minimal face of $C_i$ (by \Cref{prop:horizontal-axial}).
	In the arrangement $\A_i$, which is of type $\tilde A_m$, there are exactly two hyperplanes in the direction of $\alpha$ that contain at least one minimal face of $C_i$.
	These two hyperplanes are adjacent, and they contain $C_i$ (and therefore also the Coxeter axis $\ell$) between them.
\end{proof}

\begin{lemma}
	\label{lemma:translation}
	Let $W$ be an irreducible affine Coxeter group, and $w$ one of its Coxeter elements.
	There exists an integer $p>0$ such that $w^p$ is a translation in the positive direction of the Coxeter axis $\ell$.
\end{lemma}

\begin{proof}
	Let $n_1, \dotsc, n_k$ be the ranks of the irreducible components of the horizontal root system $\Phi_\h$.
	If $p$ is a multiple of $n_i+1$ for every $i$, then $w^p$ acts trivially on all horizontal directions, and so it must be a translation in the direction of the Coxeter axis $\ell$.
	If $\mu$ is the shortest vector in $\Mov(w)$, then $p\mu \in \Mov(w^p)$. Therefore $w^p$ is a translation of $p\mu$, which is in the positive direction of $\ell$.
\end{proof}

\begin{lemma}
	Let $W$ be an irreducible affine Coxeter group, and $w$ one of its Coxeter elements.
	For every irreducible component $\Phi_i$ of the horizontal root system $\Phi_\h$, there exists a hyperbolic element $u \in [1,w]$ such that $W_u$ is an irreducible affine Coxeter group with horizontal root system $\Phi_i$ (with respect to the Coxeter element $u$).
	In particular, $[1,u] \cap W_i = [1,w] \cap W_i$.
	\label{lemma:common-horizontal}
\end{lemma}

\begin{proof}
	Let $\Phi_\h = \Phi_1 \sqcup \dotsb \sqcup \Phi_k$.
	If $k=1$, then we can simply take $u = w$.
	Suppose from now on that $k \geq 2$, and assume without loss of generality that $i=1$.
	Let $w=th_1\dotsm h_k$ be a horizontal factorization, with $h_j \in W_j$.
	Notice that $t$ does not commute with any $h_j$, because $t^{-1}h_j t = \varphi(h_j) \neq h_j$ (by \Cref{prop:horizontal-axial}).
	Let $u = th_1 \leq w$.
	
	Since $h_2\dotsm h_k$ is horizontal, its left complement $u$ is hyperbolic.
	Let $u = u'h'$ be the hyperbolic-horizontal decomposition of $u$ (see \Cref{lemma:hyperbolic-decomposition}).
	The irreducible root system $\Phi_1$ is entirely contained in the root system of $W_{u'}$ or of $W_{h'}$, because $W_u = W_{u'} \times W_{h'}$.
	Since $t$ is a (vertical) translation and $[1,u] = [1,u'] \times [1,h']$, we have $t \leq u'$.
	Thus $\Phi_1$ is contained in the root system of $W_{u'}$, because otherwise $h_1$ would commute with $t$.
	
	By \Cref{lemma:model-poset} and \cite[Lemma 3.6]{brady2015factoring} we have $\Dir(\Min(w)) \subseteq \Dir(\Min(u'))$, so every root which is horizontal with respect to $w$ is horizontal also with respect to $u' \leq w$.
	Therefore the horizontal root system $\Phi_\h'$ of $W_{u'}$ (associated with the Coxeter element $u'$) contains $\Phi_1$.
	On the other hand, $l(u') \leq l(u) = l(t) + l(h_1) = n_1 + 2$, so the rank of $\Phi_\h'$ is at most $n_1$.
	We conclude that $\Phi_\h' = \Phi_1$ and $u' = u$.
	
	Finally we have $[1,u] \cap W_1 \subseteq [1,w] \cap W_1$, and every reflection in $[1,w] \cap W_1$ is also contained in $[1,u]$ by \Cref{cor:two-horizontal-reflections}, so actually $[1,u] \cap W_1 = [1,w] \cap W_1$.
\end{proof}

\begin{remark}
	By \Cref{lemma:common-horizontal}, in order to study the geometry of horizontal components it is enough to look at affine Coxeter groups with a unique horizontal component ($\tilde A_n$ with a $(n,1)$-bigon Coxeter element, $\tilde C_n$, and $\tilde G_2$).
	The group $\tilde G_2$ has a horizontal component of rank $1$.
	The groups $\tilde A_n$ and $\tilde C_n$ (see Sections \ref{sec:appendix-A} and \ref{sec:appendix-C} in the appendix) have a horizontal component of rank $m = n-1$.
	In all cases, choosing suitable coordinates we have that: the arrangement $\A_i$ consists of the hyperplanes $\{ x_j - x_{j'} = q \}$ for $1 \leq j < j' \leq m$ and $q \in \Z$; the horizontal prism $C_i$ is described by the inequalities $x_1 < x_2 < \dotsb < x_m < x_1 + 1$; the linear part of the Coxeter element $w$ sends $(x_1, \dotsc, x_m)$ to $(x_m, x_1, \dotsc, x_{m-1})$.
	In particular, the isomorphism type of $[1,w] \cap W_i$ (as a labeled poset) depends only on the rank $m$, and not on the ambient group $W$.
	Using ideas from the proofs of \cite[Propositions 4.7 and 7.6]{mccammond2017artin}, one can show that $[1,w] \cap W_i$ is isomorphic to the subposet of the noncrossing partition lattice of type $B_{m+1}$ consisting of the partitions without a zero block (this terminology is defined for example in \cite[Section 4.5]{armstrong2009generalized}).
	\label{rmk:horizontal-coordinates}
\end{remark}

\begin{figure}
	\begin{tikzpicture}[scale=1.8,
	extended line/.style={shorten >=-#1,shorten <=-#1}, extended line/.default=35cm,
	]
	
	\clip (-2.23, -2.6) rectangle (3.09, 1.35);
	
	\fill[black!30] ($(0, 0)$) -- ($(0, -1)$) -- ($({0.5*sqrt(3)}, -0.5)$) -- cycle;
	
	\begin{scope}[black!30]
	\newcommand{\rows}{5}
	\foreach \row in {-\rows, ...,\rows} {
		\draw [extended line] ($\row*({0.5*sqrt(3)}, 0.5)$) -- ($(0,\rows)+\row*({0.5*sqrt(3)}, -0.5)$);
		\draw [extended line] ($\row*(0, 1)$) -- ($({\rows/2*sqrt(3)}, \rows/2)+\row*({-0.5*sqrt(3)}, 0.5)$);
		\draw [extended line] ($\row*(0, 1)$) -- ($({-\rows/2*sqrt(3)}, \rows/2)+\row*({0.5*sqrt(3)}, 0.5)$);
	}
	\end{scope}
	
	\begin{scope}[very thick]
	\draw [extended line] ($(0, 0.5)$) -- ($(0,1)$);
	\draw [extended line] ($({0.5*sqrt(3)}, 0.5)$) -- ($({0.5*sqrt(3)},1)$);
	
	\draw [extended line] ($(0,0)$) -- ($({0.5*sqrt(3)}, 0.5)$);
	\draw [extended line] ($(0,-1)$) -- ($({0.5*sqrt(3)}, -0.5)$);
	
	\draw [extended line] ($(0, 0)$) -- ($({-0.5*sqrt(3)}, 0.5)$);
	\draw [extended line] ($(0, -1)$) -- ($({-0.5*sqrt(3)}, -0.5)$);
	\end{scope}
	
	\node at (0.1,-2.3) {$b$};
	\node at (1,-2.3) {$b'$};
	\node at (2.2,-1.15) {$a$};
	\node at (2.2,-2.1) {$a'$};
	\node at (2.2,1.15) {$c$};
	\node at (2.2,0.15) {$c'$};
	
	\begin{scope}[every node/.style={inner sep=1.8pt, circle, draw, fill=white}]
	\node at ($(0, 0)$) {};
	\node at ($(0, -1)$) {};
	\node at ($({0.5*sqrt(3)}, -0.5)$) {};
	\end{scope}
	
	\node[inner sep=1pt, circle, draw, fill=black] at ($({0.5/sqrt(3)}, -0.5)$) {};
	\node at ($({0.5/sqrt(3)+0.13}, -0.5)$) {$\ell$};
	\end{tikzpicture}
	\caption{A horizontal component of rank 2.
		For example, this is a section of the Coxeter complex of type $\tilde C_3$ with a plane orthogonal to the Coxeter axis $\ell$.}
	\label{fig:horizontal-A2}
\end{figure}
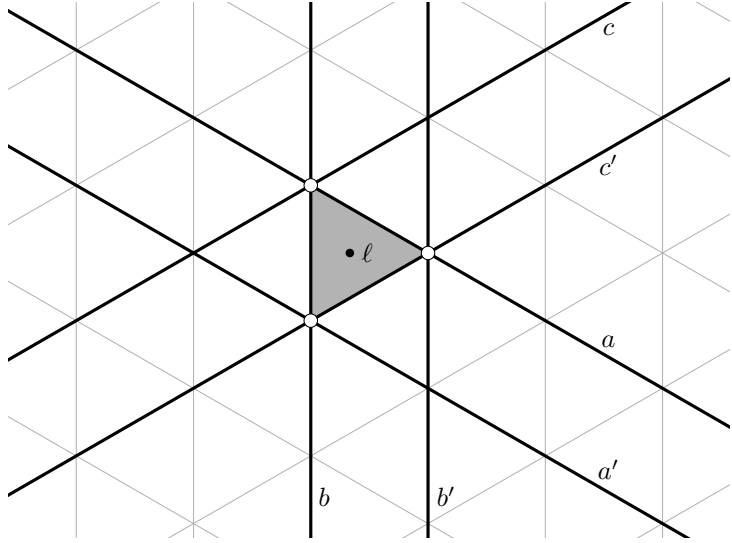


\begin{example}[Horizontal component of rank 2]
	\Cref{fig:horizontal-A2} shows the arrangement $\A_i$ of a horizontal component of rank 2.
	The horizontal prism $C_i$ is shaded, and its minimal faces are the three white vertices.
	The reflections in $[1,w] \cap W_i$ are denoted by $a,a',b,b',c,c'$, and they correspond to the $6$ thick lines.
	The Coxeter axis $\ell$ is contained in $C_i$ and is equidistant from the minimal faces.
	The Coxeter element $w$ acts on $\A_i$ as a $2\pi/3$ rotation around $\ell$, say counterclockwise.
	Then the $3$ maximal elements of $[1,w] \cap W_i$ are $ab=bc=ca$, $a'b = bc' = c'a'$, and $ab' = b'c' = c'a$.
	They are $2\pi/3$ rotations around the minimal faces of $C_i$, in counterclockwise direction.
	Since $a'b \in [1,w]$ and $a'b' \not\in [1,w]$, the right complement $u$ of $a'$ is such that $b \leq u$ and $b' \not\leq u$ in $[1,w]$.
	This is an example where the converse of \Cref{lemma:model-poset} does not hold: we have both $b \leq u$ and $b' \leq u$ in $[1,w]^L$, but $b' \not\leq u$ in $[1,w] = [1,w]^W$.
	\label{ex:horizontal-A2}
\end{example}

\section{Shellability of affine noncrossing partition posets}
\label{sec:shellability}

In this section, we construct an EL-labeling for the noncrossing partition poset $[1,w]$, where $w$ is any Coxeter element of an affine Coxeter group $W$.
Therefore the poset $[1,w]$ is EL-shellable.
This extends the analog result of Athanasiadis, Brady, and Watt for noncrossing partition lattices associated with finite Coxeter groups \cite{athanasiadis2007shellability}.

The EL-labelings of finite noncrossing partition lattices play a fundamental role in our construction and are recalled in \Cref{sec:shellability-finite}.
However, the need for a global labeling and the presence of hyperbolic intervals make the affine case substantially different from the finite case.

The EL-labeling of $[1,w]$ is going to be used in \Cref{sec:conjecture}, to complete the proof of the $K(\pi,1)$ conjecture.

\subsection{Reflection orderings and shellability of finite noncrossing partition lattices}
\label{sec:shellability-finite}

We start by describing the EL-la\-be\-lings of \cite{athanasiadis2007shellability} for the noncrossing partition lattices associated with finite crystallographic Coxeter groups.

Let $W$ be a finite Coxeter group acting on $V = \R^n$ by linear isometries, and let $R$ be its set of reflections.
Denote by $\Phi \subseteq V$ the root system of $W$, and let $\Phi^+ \subseteq \Phi$ be a positive root system.
For $\alpha \in \Phi^+$, denote by $r_\alpha \in R$ the orthogonal reflection with respect to $\alpha$.

\begin{definition}[{\cite{bourbaki1968elements, dyer1993hecke, bjorner2006combinatorics}}]
	A total ordering $\prec$ of $R$ is called a \emph{reflection ordering} for $W$ if, whenever $\alpha, \alpha_1, \alpha_2 \in \Phi^+$ are distinct positive roots and $\alpha$ is a positive linear combination of $\alpha_1$ and $\alpha_2$, we have either
	\[ r_{\alpha_1} \prec r_\alpha \prec r_{\alpha_2} \quad \text{or} \quad r_{\alpha_2} \prec r_\alpha \prec r_{\alpha_1}. \]
\end{definition}

\begin{definition}[{\cite[Definition 3.1]{athanasiadis2007shellability}}]
	Let $w$ be a Coxeter element of $W$.
	A reflection ordering $\prec$ of $R$ is \emph{compatible with $w$} if for every irreducible rank $2$ induced subsystem $\Phi' \subseteq \Phi$ the following holds: if $\alpha$ and $\beta$ are the simple roots of $\Phi'$ with respect to the positive system $\Phi' \cap \Phi^+$ and $r_\alpha r_\beta \in [1,w]$, then $r_\alpha \prec r_\beta$.
\end{definition}

Let $w$ be a Coxeter element of $W$.
Recall that the edges of the Hasse diagram of $[1,w]$ are naturally labeled by reflections: $\lambda(u, ur) = r$.
We call $\lambda\colon \E([1,w]) \to R$ the \emph{natural edge labeling} of $[1,w]$.

\begin{theorem}[{\cite[Theorem 3.5]{athanasiadis2007shellability}}]
	\label{thm:finite-EL}
	Let $W$ be a finite crystallographic Coxeter group, $R$ its set of reflections, and $w$ one of its Coxeter elements.
	If $R$ is totally ordered by a reflection ordering which is compatible with $w$, then the natural edge labeling of $[1,w]$ is an EL-labeling.
\end{theorem}

Notice that every finite Coxeter subgroup of an irreducible affine Coxeter group is crystallographic.
We are going to use \Cref{thm:finite-EL} through the following geometric construction of reflection orderings, which is similar to a construction already considered in \cite[Section 2]{dyer1993hecke}.

Let $\A$ be the reflection arrangement associated with $W$, and let $C_0$ be the chamber of the Coxeter complex of $W$ corresponding to the choice of the positive system $\Phi^+$.
Fix a point $\a \in C_0$ and a non-zero vector $\mu \in V$. Consider the affine line $\ell' = \{ \a + \theta\mu \mid \theta \in \R \} \subseteq V$, with basepoint $a$.
Assume that $\ell'$ is \emph{generic} with respect to $\A$: it intersects every hyperplane of $\A$ in exactly one point (equivalently, it is not parallel to any hyperplane of $\A$), and $H \cap \ell' \neq H' \cap \ell'$ for all hyperplanes $H \neq H'$ in $\A$.
The vector $\mu$ gives an orientation to $\ell'$. In accordance with the notation of \Cref{sec:affine-coxeter-elements}, we say that a point $\b \in \ell'$ is \emph{above} a point $\b' \in \ell'$ (or, equivalently, $\b'$ is \emph{below} $\b$) if $\b-\b'$ is a positive multiple of $\mu$.
Define a total ordering $\prec_{\ell'}$ on $R$ as follows:
\begin{itemize}
	\item first, there are the reflections that fix a point of $\ell'$ above $\a$, and $r$ comes before $r'$ if $\Fix(r) \cap \ell'$ is below $\Fix(r') \cap \ell'$;
	\item then there are the reflections that fix a point of $\ell'$ below $\a$, and again $r$ comes before $r'$ if $\Fix(r) \cap \ell'$ is below $\Fix(r') \cap \ell'$.
\end{itemize}
Notice the similarity with \Cref{lemma:hyperbolic-coxeter}.

\begin{proposition}
	\label{prop:line-ordering}
	Let $W$ be a finite Coxeter group.
	For every generic line $\ell'$ as above, the total ordering $\prec_{\ell'}$ of $R$ is a reflection ordering for $W$.
\end{proposition}

\begin{proof}
	Denote by $\<\cdot,\cdot\>$ the scalar product of $V = \R^n$.
	By definition of the chamber $C_0$, we have that $\< \a, \alpha\> > 0$ for every positive root $\alpha \in \Phi^+$.
	Until the end of this proof, renormalize all positive roots $\alpha \in \Phi^+$ so that $\<\a, \alpha\> = 1$.
	
	Let $\alpha \in \Phi^+$.
	The intersection point $\a + \theta_\alpha \mu$ of $\Fix(r_\alpha)$ with $\ell'$ is determined by the relation $\< \a + \theta_\alpha \mu, \alpha \> = 0$.
	Since $\<\a, \alpha\> = 1$, we get $\theta_\alpha = - \< \mu, \alpha \>^{-1}$.
	By definition of $\prec_{\ell'}$, we have that $r_\alpha \prec_{\ell'} r_\beta$ if and only if $\theta_\alpha^{-1} > \theta_\beta^{-1}$, which happens if and only if $\<\mu, \alpha\> < \< \mu, \beta \>$.
	
	Now suppose that $\alpha\in \Phi^+$ is a positive linear combination of $\alpha_1, \alpha_2 \in \Phi^+$.
	Since $\<\a, \alpha\> =  \<\a, \alpha_1\> = \<\a, \alpha_2 \> = 1$, we have $\alpha = c\alpha_1 + (1-c) \alpha_2$ with $0 < c < 1$.
	Then $\< \mu, \alpha \> = c \< \mu, \alpha_1 \> + (1-c) \< \mu, \alpha_2 \>$, which is between $\< \mu, \alpha_1 \>$ and $\< \mu, \alpha_2 \>$.
	Therefore $r_\alpha$ is between $r_{\alpha_1}$ and $r_{\alpha_2}$ in the total ordering $\prec_{\ell'}$.
\end{proof}

\subsection{Orderings of horizontal reflections}
\label{sec:horizontal-intervals}

Let $W$ be an irreducible affine Coxeter group, acting on an affine space $E = \R^n$ by Euclidean isometries, where $n$ is the rank of $W$.
As usual, denote by $R$ its set of reflections, and let $w$ be one of its Coxeter elements.
Let $\Phi$ be the root system associated with $W$, and $\Phi_\h \subseteq \Phi$ the horizontal root system.

Recall from \Cref{sec:affine-coxeter-elements} that the reflections occurring as labels of the interval $[1,w]$ are those that fix at least one axial vertex.
This set $R_0 \subseteq R$ includes the set $R_\v$ of all vertical reflections, and a finite set $R_\h$ of horizontal reflections, with two consecutive horizontal reflections for each pair of opposite roots.
We are going to construct a total order $\prec$ of $R_0$ which makes the natural edge labeling $\lambda\colon \E([1,w]) \to R_0$ an EL-labeling of $[1,w]$.
To do this, we start by defining a total ordering $\prec_\h$ on the subset $R_\h \subseteq R_0$ of the horizontal reflections.

We use the notation of \Cref{sec:horizontal-components}.
Let $\Phi_\h = \Phi_1 \sqcup \dotsb \sqcup \Phi_k$, where $\Phi_i$ is an irreducible root system of type $A_{n_i}$.
Fix a horizontal factorization $w = th_1\dotsm h_k$, so that $W_{h_i} \subseteq W$ is a finite parabolic subgroup with root system $\Phi_i$, and $h_i$ is a Coxeter element of $W_{h_i}$.
For now, we focus on a single horizontal component $i$.
Let $m = n_i$ be the rank of $\Phi_i$, and let $W_i \subseteq W$ be the Coxeter subgroup of type $\tilde A_m$ generated by the reflections with respect to roots in $\Phi_i$.
Denote by $\A_i$ the corresponding hyperplane arrangement, and let $C_i$ be the $i$-th horizontal prism.

\begin{lemma}
	\label{lemma:horizontal-line-ordering}
	Let $u = \varphi^p(h_i)$ be any maximal element of $[1,w] \cap W_i$, as in \Cref{prop:horizontal-axial}.
	Fix any point $a$ of the Coxeter axis $\ell$.
	There exists a line $\ell'$, with basepoint $a$ and direction in $\Span(\Phi_i)$, such that the reflection ordering $\prec_{\ell'}$ for $W_{u}$ (defined in \Cref{sec:shellability-finite}) is compatible with $u$.
\end{lemma}

\begin{proof}
	Choosing coordinates as in \Cref{rmk:horizontal-coordinates}, we can assume that: the arrangement $\A_i$ consists of the hyperplanes $\{ x_j - x_{j'} = q \}$ for $1 \leq j < j' \leq m$ and $q \in \Z$; the horizontal prism $C_i$ is described by the inequalities $x_1 < x_2 < \dotsb < x_m < x_1 + 1$; the element $u$ sends $(x_1, \dotsc, x_m)$ to $(x_m, x_1, \dotsc, x_{m-1})$, and the minimal face fixed by $u$ is given by $\{x_1 = x_2 = \dotsb = x_m\}$.
	For a small enough $\epsilon > 0$, the line $\ell' = \{ a + \theta (1, \epsilon, \epsilon^2, \dotsc, \epsilon^{m-1}) \mid \theta \in \R \}$ intersects the hyperplanes $\{ x_j - x_{j'} = 0\}$ (with $j < j'$) in the lexicographic order of the pairs $(j,j')$, and always for $\theta > 0$.
	This is the reflection ordering described in \cite[Example 3.3]{athanasiadis2007shellability}, and it is compatible with the Coxeter element $u$.
	In order to get a line with direction in $\Span(\Phi_i)$, project $\ell'$ to the affine subspace parallel to $\Span(\Phi_i)$ and containing $a$.
\end{proof}

By the previous lemma, there exists a reflection ordering $\prec_i'$ for $W_{h_i}$ which is compatible with the Coxeter element $h_i$.
Extend $\prec_i'$ to a total ordering $\prec_i$ of the set $R_\h \cap W_i$, in the following way: whenever $r_1 \prec_i' r_2$ in $R_\h \cap W_{h_i}$, then every parallel translate of $r_1$ comes before every parallel translate of $r_2$.
Since there are no minimal factorizations of $w$ that use two parallel horizontal reflections, the relative order of parallel horizontal reflections is not important and can be chosen arbitrarily.

\begin{remark}
	\label{rmk:horizontal-loop}
	The total ordering $\prec_i$ does not restrict to an ordering of \Cref{lemma:horizontal-line-ordering} for $u \neq h_i$.
	Rather, it is obtained by ``translating'' a chosen ordering for $W_{h_i}$ (given by \Cref{lemma:horizontal-line-ordering}) to the other subgroups $W_{u}$.
	There exists no total ordering of $R_\h \cap W_i$ which restricts to an ordering of \Cref{lemma:horizontal-line-ordering} for every maximal element $u$ since the reflections with respect to the walls of $C_i$ would necessarily form a loop.
\end{remark}

\begin{lemma}
	\label{lemma:EL-horizontal-factor}
	For every horizontal element $u \in [1,w] \cap W_i$, the total ordering $\prec_i$ makes the natural edge-labeling $\lambda \colon \E([1,u]) \to R_\h \cap W_i$ an EL-labeling of $[1,u]$.
\end{lemma}

\begin{proof}
	It is enough to prove this for the maximal elements of $[1,w] \cap W_i$.
	By \Cref{prop:horizontal-axial}, these are of the form $u = \varphi^p(h_i)$ for $p \in \{0, \dotsc, m\}$.
	
	Both $\Fix(h_i)$ and $\Fix(u)$ are minimal faces of the $i$-th horizontal prism $C_i$.
	Let $t'$ be a translation that sends $\Fix(h_i)$ to $\Fix(u)$ (it does not need to be an element of $W$).
	The linear part of $h_i$ is equal to the linear part of $u$ by \Cref{prop:horizontal-axial}, so $t' h_i t'^{-1} = u$.
	Also, $t'$ sends any hyperplane containing $\Fix(h_i)$ to a parallel hyperplane containing $\Fix(u)$.
	Therefore the conjugation by $t'$ is an isomorphism $W_{h_i} \to W_{u}$ that sends the Coxeter element $h_i$ to the Coxeter element $u$, and it sends any reflection in $W_{h_i}$ to its unique parallel translate in $W_u$.
	In particular, this isomorphism preserves the total ordering $\prec_i$.
	
	By construction, $\prec_i$ restricts to a reflection ordering for $W_{h_i}$ compatible with $h_i$, so the natural edge-labeling $\lambda \colon \E([1,h_i]) \to R_\h \cap W_i$ is an EL-labeling of $[1,h_i]$ by \Cref{thm:finite-EL}.
	Using the above isomorphism $W_{h_i} \to W_u$, we obtain the same conclusion for $u$.
\end{proof}

\begin{example}[Ordering of horizontal reflections in a component of rank $2$]
	Consider a horizontal component of rank $2$, with the notation of \Cref{ex:horizontal-A2}.
	If we choose $h_i = ab$ as the preferred maximal element of $[1,w] \cap W_i$, one of the possible total orderings $\prec_i$ of $R_\h \cap W_i$ is the following: $a \prec_i a' \prec_i c \prec_i c' \prec_i b \prec_i b'$.
	The factorizations of the $3$ maximal elements as a $\prec_i$-increasing product of reflections are: $ab$, $a'b$, and $ab'$.
\end{example}

Let $\prec_\h$ be any total ordering of $R_\h$ obtained as a shuffle of the total orderings $\prec_i$ for $i \in \{1,\dotsc, k\}$.

\begin{lemma}[EL-shellability of horizontal intervals]
	\label{lemma:horizontal-intervals}
	Let $W$ be an irreducible Coxeter group, and $w$ one of its Coxeter elements.
	For every horizontal element $u \in [1,w]$, the total ordering $\prec_\h$ makes the natural edge labeling $\lambda \colon \E([1,u]) \to R_\h$ an EL-labeling of $[1,u]$.
\end{lemma}

\begin{proof}
	Write $u = u_1u_2\dotsm u_k$, with $u_i \in [1,w] \cap W_i$.
	Then $[1,u] = [1,u_1] \times [1,u_2] \times \dotsb \times [1,u_k]$.
	We have an EL-labeling on each factor $[1,u_i]$ by \Cref{lemma:EL-horizontal-factor}, and we conclude using \Cref{thm:product-el-labelings}.
\end{proof}

\subsection{Axial orderings and shellability of affine noncrossing partition posets}

We are now ready to construct an EL-labeling of the interval $[1,w]$ in an affine Coxeter group $W$, for any fixed Coxeter element $w$.
As in \Cref{sec:horizontal-intervals}, we assume that $W$ is irreducible.
Once the irreducible case is settled, one can easily get an EL-labeling for reducible affine Coxeter groups by applying \Cref{thm:product-el-labelings}.

Let $\ell$ be the Coxeter axis, and fix an axial chamber $C_0$ of the Coxeter complex.

\begin{definition}[Axial ordering]
	\label{def:axial-ordering}
	An \emph{axial ordering} of the set of reflections $R_0 = R \cap [1,w]$ is a total ordering of the following form:
	\begin{itemize}
		\item first, there are the vertical reflections that fix a point of $\ell$ above $C_0$, and $r$ comes before $r'$ if $\Fix(r) \cap \ell$ is below $\Fix(r') \cap \ell$ (we call these the \emph{positive vertical reflections});
		\item then, there are the horizontal reflections in $R_\h$, following any total ordering $\prec_\h$ constructed in \Cref{sec:horizontal-intervals};
		\item finally, there are the vertical reflections that fix a point of $\ell$ below $C_0$, and again $r$ comes before $r'$ if $\Fix(r) \cap \ell$ is below $\Fix(r') \cap \ell$ (we call these the \emph{negative vertical reflections}).
	\end{itemize}
	The relative order of vertical reflections that fix the same point of $\ell$ can be chosen arbitrarily.
\end{definition}

\begin{remark}
	\label{rmk:commuting-reflections}
	If two vertical reflections fix the same point of $\ell$, they commute.
	This is proved in \Cref{prop:axial-point-A} for the case $\tilde A_n$, and follows from \Cref{sec:bipartite-coxeter} for the other cases.
\end{remark}

\begin{example}[Axial orderings for $\tilde A_2$ and $\tilde G_2$]
	In the case $\tilde A_2$, with the notation of \Cref{ex:A2}, one of the two axial orderings of $R_0$ is the following:
	\[ a_1 \prec c_2 \prec a_3 \prec \dotsb \prec b \prec b' \prec \dotsb \prec c_{-2} \prec a_{-1} \prec c_0. \]
	The other axial ordering is obtained by exchanging the two horizontal reflections $b$ and $b'$.
	Notice that there are infinitely many reflections before $b$, and infinitely many reflections after $b'$.
	Indeed, $b$ has no immediate predecessor, and $b'$ has no immediate successor.
	The following is a portion of one of the infinitely many axial orderings in the case $\tilde G_2$, using the notation of \Cref{ex:G2}:
	\[ a_1 \prec d_1 \prec c_2 \prec e_3 \prec b_3 \prec c_4 \prec \dotsb \prec f \prec f' \prec \dotsb \prec c_{-2} \prec e_{-1} \prec b_{-1} \prec c_0. \]
	\label{ex:axial-orderings}
\end{example}

Our aim for the rest of this section is to prove that an axial ordering makes the natural edge labeling $\lambda \colon \E([1,w]) \to R_0$ an EL-labeling of $[1,w]$.

\begin{remark}[{cf.\ \cite[Lemma 3.7]{athanasiadis2007shellability}}]
	\label{rmk:isomorphic-interval}
	Let $[u,v]$ be an interval in $[1,w]$.
	The map $f \colon [1, u^{-1}v] \to [u,v]$ defined by $f(x) = ux$ is a label-preserving poset isomorphism.
\end{remark}

Notice that an axial ordering of $R_0$ is not a well-ordering.
For example, the set of negative vertical reflections does not have a smallest element.
However, the well-ordering property holds for the subsets of the form $R_0 \cap [1,u]$, as we show in the second of the next three preparatory lemmas.

\begin{lemma}
	\label{lemma:positive-negative-reflection}
	For every hyperbolic element $u \in [1,w]$, the interval $[1,u]$ contains at least one positive vertical reflection and at least one negative vertical reflection.
\end{lemma}

\begin{proof}
	Since $u$ is vertical, the interval $[1,u]$ contains at least one vertical reflection $r$.
	Let $w^p$ be a power of $w$ that acts as a translation in the positive direction of the Coxeter axis $\ell$, where $p$ is a positive integer (see \Cref{lemma:translation}).
	We have that $w^p$ commutes with the right complement $v$ of $u$, because $v$ is horizontal.
	Then $w^p$ commutes also with $u = wv^{-1}$.
	In particular, the conjugation by $w^p$ fixes $u$ and is an automorphism of $[1,u]$.
	If we conjugate the vertical reflection $r \in [1,u]$ by $w^{mp}$, for a sufficiently large $m \in \Z$, we get a vertical reflection $r' = w^{-mp}rw^{mp} \in [1,u]$ that fixes a point of $\ell$ below $C_0$, i.e.\ $r'$ is negative.
	Similarly, if we conjugate $r$ by $w^{mp}$ for a sufficiently small $m \in \Z$, we get a positive vertical reflection $r'' \in [1,u]$.
\end{proof}

\begin{lemma}
	\label{lemma:smallest-reflection}
	Let $\prec$ be an axial ordering of $R_0$.
	For every element $u\in [1,w]$ with $u \neq 1$, the set $R_0 \cap [1,u]$ has a $\prec$-smallest reflection and a $\prec$-largest reflection.
\end{lemma}

\begin{proof}
	Suppose by contradiction that there is no smallest reflection in $R_0 \cap [1,u]$.
	Since $R_0 \cap [1,u]$ is non-empty, we must have an infinite decreasing sequence of reflections $r_1 \succ r_2 \succ \dotsb$ with $r_i \in [1,u]$.
	In particular, the interval $[1,u]$ is infinite, so $u$ is hyperbolic.
	By \Cref{lemma:positive-negative-reflection}, there is at least one positive vertical reflection $r \in [1,u]$.
	There is only a finite number of reflections $r' \preceq r$ in $R_0$, so there exists a $\prec$-smallest reflection in $R_0 \cap [1,u]$.
	Similarly, there is also a $\prec$-largest reflection.
%
\end{proof}

In the definition of an EL-labeling, maximal chains are compared lexicographically.
In our case it is also useful to compare them \emph{colexicographically} (or \emph{antilexicographically}): a tuple $(r_1, r_2, \dotsc, r_m)$ is colexicographically smaller than $(r_1', r_2', \dotsc, r_m')$ if the reflected tuple $(r_m, r_{m-1}, \dotsc, r_1)$ is lexicographically smaller than the reflected tuple $(r_m', r_{m-1}', \dotsc, r_1')$.

\begin{lemma}
	\label{lemma:unique-lexicographically-smallest}
	Fix an axial ordering $\prec$ of $R_0$.
	Every interval $[u,v]$ in $[1,w]$ has a unique lexicographically smallest maximal chain, and this chain is increasing.
	Similarly, every interval $[u,v]$ in $[1,w]$ has a unique colexicographically largest maximal chain, and this chain is increasing.
\end{lemma}

\begin{proof}
	We proceed by induction on the length of the interval $[u,v]$, the case $u=v$ being trivial.
	Suppose from now on that $u < v$.
	The labels of the covering relations $u \lessdot u'$ with $u' \in [u,v]$ are all distinct, and they are given by the reflections in $R_0 \cap [1, u^{-1}v]$.
	By \Cref{lemma:smallest-reflection}, there is a $\prec$-smallest reflection $r$ in $R_0 \cap [1, u^{-1}v]$.
	Let $u' = ur$.
	In view of the induction hypothesis applied to $[u',v]$, it is enough to prove that all covering relations in $[u',v]$ have labels greater than $r$.
	If $r'$ is a label of some covering relation in $[u',v]$, by \Cref{lemma:hurwitz-action} there is a minimal length factorization of $u^{-1}v$ that starts with $rr'$.
	Then $r' \neq r$.
	In addition we have $r' \in [1, u^{-1}v]$, so $r' \succ r$ by $\prec$-minimality of $r$.
	
	The proof for the colexicographic order is similar.
\end{proof}

After these preparatory lemmas, we are ready to prove the key results that lead to shellability.

\begin{lemma}[Increasing chains in elliptic intervals]
	\label{lemma:elliptic-intervals}
	Fix an axial ordering $\prec$ of $R_0$, and let $u \in [1,w]$ be an elliptic element.
	The interval $[1,u]$ has at most one increasing maximal chain.
\end{lemma}

\begin{proof}
	Consider an irreducible horizontal component $W_i \subseteq W$, with associated hyperplane arrangement $\A_i$ and horizontal prism $C_i$ (with the notation of \Cref{sec:horizontal-components}).
	Let $L$ be the intersection of all the hyperplanes of $\A_i$ containing $\Fix(u)$.
	Since $\Fix(u)$ contains at least one axial vertex by \Cref{lemma:fixed-axial-vertex}, only hyperplanes containing an axial vertex can occur, and so $L$ is a flat in the subarrangement of $\A_i$ consisting of the hyperplanes that contain a minimal face of $C_i$.
	Then $L$ itself contains at least one minimal face $F$ of $C_i$.
	By construction, the fixed set of any element of $[1,u] \cap W_i$ contains $F$ (because it contains $\Fix(u)$ and is a flat of $\A_i$).
	By \Cref{prop:horizontal-axial}, this minimal face $F$ is the fixed set of a maximal element $u_i$ of $[1,w] \cap W_i$.
	By \Cref{lemma:elliptic-subposet}, $[1,u] \cap W_i \subseteq [1,u_i]$.
	
	The construction of the previous paragraph yields elements $u_1, \dotsc, u_k$, one for each irreducible horizontal component.
	Fix a point $\a \in C_0 \cap \ell$.
	For $i \in \{1,\dotsc, k\}$, apply \Cref{lemma:horizontal-line-ordering} to the maximal element $u_i$ and to the point $\a \in \ell$, and get an oriented line $\ell_i'$ with unit direction $\mu_i \in \Span(\Phi_i)$.
	As usual, denote by $\mu$ the shortest vector in $\Mov(w)$, which indicates the positive direction of the Coxeter axis $\ell$.
	For $\epsilon > 0$, consider the oriented line $\ell' = \{ \a + \theta \mu' \mid \theta \in \R \}$ with basepoint $a$ and direction $\mu' = \mu + \epsilon \mu_1 + \dotsb + \epsilon \mu_k$.
	By construction, its projection on the affine subspace through $a$ parallel to $\Span(\Phi_i)$ is precisely the line $\ell_i'$. Therefore $\ell'$ and $\ell_i'$ intersect the hyperplanes of the reflections in $W_{u_i}$ in the same order.
	
	Perturb the line $\ell'$ slightly, so that it becomes generic with respect to the hyperplanes of the reflections in $W_u$, and the basepoint $\a$ remains in $C_0$.
	If $\epsilon > 0$ is small enough and the perturbation is small enough, the total ordering $\prec_{\ell'}$ of the finite set of reflections $R_0 \cap W_u$ has the following form:
	\begin{itemize}
		\item first, there are the positive vertical reflections of $W_u$ (i.e.\ those that fix a point of $\ell$ above $C_0$), and $r$ comes before $r'$ if $\Fix(r) \cap \ell$ is below $\Fix(r') \cap \ell$;
		\item then there are the horizontal reflections of $W_u$, and in each irreducible component $W_i$ they are ordered as in $\prec_{\ell_i'}$;
		\item finally there are the negative vertical reflections of $W_u$, and again $r$ comes before $r'$ if $\Fix(r) \cap \ell$ is below $\Fix(r') \cap \ell$.
	\end{itemize}
	Notice that the relative order of the vertical reflections that fix the same point of $\ell$ can be different in $\prec_{\ell'}$ and in the axial ordering $\prec$.
	
	Let $\Phi_u \subseteq \Phi$ be the root system of $W_u$.
	By \Cref{prop:line-ordering}, the total ordering $\prec_{\ell'}$ is a reflection ordering for $W_u$, with respect to the positive system $\Phi_u^+ \subseteq \Phi_u$ consisting of the roots that point towards the halfspaces containing the chamber $C_0$.
	We want to show that $\prec_{\ell'}$ is compatible with the Coxeter element $u$.
	For this, let $\Phi' \subseteq \Phi_u$ be an irreducible rank 2 induced subsystem.
	Let $\alpha$ and $\beta$ be the simple roots of $\Phi'$ with respect to $\Phi' \cap \Phi_u^+$, with corresponding reflections $r_\alpha, r_\beta \in W_u$, and assume that $r_\alpha r_\beta \in [1,u]$.
	We need to prove that $r_\alpha \prec_{\ell'} r_\beta$.
	\begin{itemize}
		\item Case 1: $r_\alpha$ is vertical.
		By \Cref{lemma:elliptic-coxeter-A} (for the case $\tilde A_n$) and \Cref{lemma:elliptic-coxeter-bipartite} (for the other cases), the right complement $v$ of $r_\alpha$ is a Coxeter element for the finite parabolic subgroup $W_v$ of $W$ that fixes $\b$, where $\b$ is the unique vertex not fixed by $r_\alpha$ among the vertices of the axial chamber $C$ immediately below $\Fix(r_\alpha) \cap \ell$.
		Since $r_\alpha r_\beta \in [1,u] \subseteq [1,w]$, we have that $r_\beta \leq v$, so $r_\beta$ fixes $\b$ by \Cref{lemma:model-poset}.
		
		\noindent Let $\A'$ be the reflection arrangement associated with the dihedral group $\< r_\alpha, r_\beta \> \subseteq W_u$ generated by $r_\alpha$ and $r_\beta$ (so its root system is $\Phi'$).
		Denote by $C'$ the chamber of $\A'$ containing $C_0$, i.e.\ the chamber corresponding to the positive system $\Phi' \cap \Phi_u^+$.
		Denote by $C''$ the chamber of $\A'$ containing $C$.
		Both $C'$ and $C''$ have $H_\alpha = \Fix(r_\alpha)$ and $H_\beta= \Fix(r_\beta)$ as their walls: this is true for $C'$ by definition of $\alpha$ and $\beta$; $H_\alpha$ is a wall of $C''$ because it is a wall of $C$; $H_\beta$ is a wall of $C''$ because it contains $b$, which is a vertex of $C$ not contained in $H_\alpha$.
		So either $C'$ and $C''$ are the same chamber, or they are opposite chambers in $\A'$ (because $H_\alpha$ and $H_\beta$ are not orthogonal, so the arrangement $\A'$ contains at least another hyperplane).
		Since $C_0$ and $C$ are axial chambers, the Coxeter axis $\ell$ intersects both $C'$ and $C''$.
		The first hyperplane of $\A'$ that intersects $\ell$ above $C$ is $H_\alpha$, so the first hyperplane of $\A'$ that intersects $\ell'$ above $C''$ is $H_\alpha$.
		
		\noindent If $C' = C''$, then the first hyperplane of $\A'$ that intersects $\ell'$ above $\a$ is $H_\alpha$, and therefore $r_\alpha \prec_{\ell'} r_\beta$.
		Suppose now that $C'$ and $C''$ are opposite chambers in $\A'$.
		Since $\ell$ intersects both $C'$ and $C''$, and $H_\beta$ separates $C'$ and $C''$, we have that $r_\beta$ is also vertical.
		In addition, $C'$ and $C''$ are the two chambers of $\A'$ that intersect $\ell$ (resp.\ $\ell'$) in an unbounded subset of $\ell$ (resp.\ $\ell'$).
		Therefore, moving the basepoint $\a$ of $\ell'$ from $\ell' \cap C'$ to $\ell' \cap C''$ does not alter the total ordering $\prec_{\ell'}$ of the reflections in $\< r_\alpha, r_\beta\>$.
		As in the case $C' = C''$, we conclude that $r_\alpha \prec_{\ell'} r_\beta$.
		
		\item Case 2: $r_\beta$ is vertical.
		The argument is the same as for case 1, with the roles of $r_\alpha$ and $r_\beta$ exchanged.
		In this case $v$ is the left complement of $r_\beta$, and $C$ is the axial chamber immediately above $\Fix(r_\beta) \cap \ell$.
		
		\item Case 3: both $r_\alpha$ and $r_\beta$ are horizontal.
		Since $\alpha$ and $\beta$ are not orthogonal, they must belong to the same irreducible horizontal component $\Phi_i$.
		Then $r_\alpha r_\beta \in [1,u] \cap W_i \subseteq [1,u_i]$.
		By \Cref{lemma:horizontal-line-ordering}, the reflection ordering for $W_{u_i}$ induced by $\ell_i'$ is compatible with $u_i$, so we have $r_\alpha \prec_{\ell_i'} r_\beta$, and therefore $r_\alpha \prec_{\ell'} r_\beta$.
	\end{itemize}

	We proved that $\prec_{\ell'}$ is a reflection ordering for $W_u$ which is compatible with $u$.
	By \Cref{thm:finite-EL}, it makes the natural labeling $\lambda$ of $[1,u]$ an EL-labeling.
	
	Suppose to have a $\prec$-increasing maximal chain, which corresponds to a minimal length factorization $u = r_1r_2\dotsm r_m$ with $r_1 \prec r_2 \prec \dotsb \prec r_m$.
	Our aim is to show that the reflections $r_1, r_2, \dotsc, r_m$ are uniquely determined.
	
	By definition of $\prec$, the sequence $r_1 \prec r_2 \prec \dotsb \prec r_m$ consists of an initial segment $r_1 \prec \dotsb \prec r_{j}$ of positive vertical reflections, a middle segment $r_{j+1} \prec \dotsb \prec r_{j'}$ of horizontal reflections, and a final segment $r_{{j'}+1} \prec \dotsb \prec r_m$ of negative vertical reflections.
	
	Reorder the reflections of the initial segment so that they are $\prec_{\ell'}$-increasing: $r_{\sigma(1)} \prec_{\ell'} r_{\sigma(2)} \prec_{\ell'} \dotsb \prec_{\ell'} r_{\sigma(j)}$, for some permutation $\sigma$ of $\{1, \dotsc, j\}$.
	The relative order of vertical reflections that fix different points of $\ell$ is the same in $\prec$ and $\prec_{\ell'}$ (by construction of $\ell'$).
	By \Cref{rmk:commuting-reflections}, this means that the relative order of non-commuting vertical reflections is the same in $\prec$ and $\prec_{\ell'}$.
	Therefore $r_{\sigma(1)} r_{\sigma(2)} \dotsm r_{\sigma(j)} = r_1 r_2 \dotsm r_j$.
	Similarly, if $\tau$ is the permutation of $\{j'+1, \dotsc, m\}$ such that $r_{\tau(j'+1)} \prec_{\ell'} r_{\tau(j'+2)} \prec_{\ell'} \dotsb \prec_{\ell'} r_{\tau(m)}$, we have $r_{\tau(j'+1)} r_{\tau(j'+2)} \dotsm r_{\tau(m)} = r_{j'+1} r_{j'+2} \dotsm r_m$.
	
	Let $h = r_{j+1} r_{j+2}\dotsm r_{j'} \in [1,u]$ be the product of the reflections in the middle segment.
	Since the total ordering $\prec_{\ell'}$ makes $\lambda$ an EL-labeling of $[1,u]$, the interval $[1,h]$ has a unique $\prec_{\ell'}$-increasing maximal chain.
	It corresponds to a minimal length factorization $h = r_{j+1}' r_{j+2}' \dotsm r_{j'}'$ with $r_{j+1}' \prec_{\ell'} r_{j+2}' \prec_{\ell'} \dotsb \prec_{\ell'} r_{j'}'$.
	Since $h$ is horizontal, the reflections $r_{j+1}', \dotsc, r_{j'}'$ are horizontal.
	Therefore they come after the positive vertical reflections and before the negative vertical reflections, in the total ordering $\prec_{\ell'}$ of $R_0 \cap W_u$.
	Putting everything together, we get
	\begin{IEEEeqnarray*}{L}
		r_{\sigma(1)} \prec_{\ell'} r_{\sigma(2)} \prec_{\ell'} \dotsb \prec_{\ell'} r_{\sigma(j)} \\
		\prec_{\ell'} r_{j+1}' \prec_{\ell'} r_{j+2}' \prec_{\ell'} \dotsb \prec_{\ell'} r_{j'}' \\
		\prec_{\ell'} r_{\tau(j'+1)} \prec_{\ell'} r_{\tau(j'+2)} \prec_{\ell'} \dotsb \prec_{\ell'} r_{\tau(m)}.
	\end{IEEEeqnarray*}
	The product of these reflections is equal to $u$.
	Since $\prec_{\ell'}$ makes $\lambda$ an EL-labeling of $u$, this factorization of $u$ is uniquely determined.
	In particular, the sets $\{ r_{\sigma(1)}, \dotsc, r_{\sigma(j)}\} \allowbreak = \{ r_1, \dotsc, r_j \}$ and $\{ r_{\tau(j'+1)}, \dotsc, r_{\tau(m)} \} \allowbreak = \{ r_{j'+1}, \dotsc, r_m \}$ are uniquely determined, and also the horizontal element $h = r_{j+1}' \dotsm r_{j'}'$ is uniquely determined.
	Since $r_1 \prec r_2 \prec \dotsb \prec r_j$, and $r_{j'+1} \prec r_{j'+2} \prec \dotsb \prec r_{m}$, the reflections $r_1, \dotsc, r_j$ and $r_{j'+1}, \dotsc, r_{m}$ are uniquely determined.
	Finally, the total order $\prec$ coincides with $\prec_\h$ on the horizontal reflections, and $\prec_\h$ makes $\lambda$ an EL-labeling of $[1,h]$ by \Cref{lemma:horizontal-intervals}.
	Then the reflections $r_{j+1}, \dotsc, r_{j'}$ are uniquely determined, because they satisfy $r_{j+1} \prec_\h r_{j+2} \prec_\h \dotsb \prec_\h r_{j'}$.
\end{proof}

\begin{lemma}[Increasing chains in hyperbolic intervals]
	\label{lemma:hyperbolic-intervals}
	Fix an axial ordering $\prec$ of $R_0$, and let $u \in [1,w]$ be an hyperbolic element such that the Coxeter subgroup $W_u \subseteq W$ is irreducible.
	The interval $[1,u]$ has at most one increasing maximal chain.
\end{lemma}

\begin{proof}
	Suppose to have an increasing maximal chain, which corresponds to a minimal length factorization $u = r_1r_2\dotsm r_m$ with $r_1 \prec r_2 \prec \dotsb \prec r_m$.
	In this factorization, the horizontal reflections appear in a contiguous (possibly empty) middle segment.
	Since $u$ is vertical, at least one of $r_1$ and $r_m$ is a vertical reflection.
	Denote by $\A_u$ the hyperplane arrangement associated with the irreducible affine Coxeter subgroup $W_u \subseteq W$.
	
	\begin{itemize}
		\item Case 1: $r_1$ is the first reflection of $R_0 \cap [1,u]$ with respect to the total order $\prec$.
		By \Cref{lemma:positive-negative-reflection}, $r_1$ is a positive vertical reflection.
		Then $v = r_2\dotsm r_m$ is an elliptic element, and by \Cref{lemma:elliptic-intervals} the reflections $r_2, \dotsc, r_m$ are uniquely determined.
		More precisely, by \Cref{lemma:unique-lexicographically-smallest}, the factorization $v = r_2 \dotsm r_m$ has to be the unique colexicographically largest minimal length factorization of $v$.
		
		\item Case 2: $r_1$ is some other vertical reflection.
		Let $C'$ be the axial chamber immediately below $\Fix (r_1) \cap \ell$, and let $C$ be the chamber of $\A_u$ containing $C'$.
		By \Cref{lemma:hyperbolic-coxeter} there is a minimal length factorization of $u$ that starts with $r_1$ and uses all the reflections with respect to the walls of $C$.
		Since $r_1$ is not the first reflection of $R_0 \cap [1,u]$, there is a reflection $r$ with respect to a wall of $C$ such that $r \prec r_1$.
		Then $r \leq r_2\dotsm r_m$.
		By \Cref{lemma:elliptic-intervals}, the reflections $r_2, \dotsc, r_m$ are uniquely determined, and by \Cref{lemma:unique-lexicographically-smallest} the factorization $r_2\dotsm r_m$ has to be the unique lexicographically smallest factorization of $r_1 u$.
		Therefore $r_2 \preceq r \prec r_1$, which is impossible because $r_1 \prec r_2$.
		
		\item Case 3: $r_m$ is the last reflection of $R_0 \cap [1,u]$ with respect to the total order $\prec$.
		As in case 1, the reflections $r_1, \dotsc, r_{m-1}$ are uniquely determined.
		
		\item Case 4: $r_m$ is some other vertical reflection.
		As in case 2, this is impossible.
	\end{itemize}

	We have shown that there are at most two minimal length factorizations $u = r_1 r_2 \dotsm r_m$ with $r_1 \prec r_2 \prec \dotsb \prec r_m$: one where $r_1$ is the first reflection of $R_0 \cap [1,u]$, and $r_2, \dotsc, r_m$ are uniquely determined (case 1), and one where $r_m$ is the last reflection of $R_0 \cap [1,u]$, and $r_1, \dotsc, r_{m-1}$ are uniquely determined (case 3).
	By \Cref{lemma:hyperbolic-coxeter}, applied to the chamber of $\A_u$ containing $C_0$, there is a minimal length factorization of $u$ that starts with the first reflection of $R_0 \cap [1,u]$ and ends with the last reflection of $R_0 \cap [1,u]$.
	This means that, in case 1 (where $r_1$ is the first reflection of $R_0 \cap [1,u]$, and $v = r_2\dotsm r_m$), the interval $[1,v]$ contains the last reflection of $R_0 \cap [1,u]$.
	Since $r_2\dotsm r_m$ is the unique colexicographically largest factorization of $v$, the last factor $r_m$ is the last reflection of $R_0 \cap [1,u]$.
	Then the factorization of case 1 coincides with the factorization of case 3.
\end{proof}

Finally, we prove that axial orderings make the natural edge labeling of $[1,w]$ an EL-labeling.

\begin{theorem}[EL-shellability]
	\label{thm:shellability}
	Let $W$ be an irreducible affine Coxeter group, and $w$ one of its Coxeter elements.
	Let $\lambda \colon \E([1,w]) \to R_0$ be the natural edge labeling of $[1,w]$, where $R_0$ is totally ordered by an axial ordering.
	Every interval $[u,v]$ in $[1,w]$ has a unique increasing maximal chain, and this chain is both the lexicographically smallest and the colexicographically largest maximal chain of $[u,v]$.
	In particular, $\lambda$ is an EL-labeling of $[1,w]$.
\end{theorem}

\begin{proof}
	By \Cref{rmk:isomorphic-interval}, it is enough to consider intervals of the form $[1,u]$, with $u \in [1,w]$.
	In view of \Cref{lemma:unique-lexicographically-smallest}, it only remains to show that $[1,u]$ has at most one increasing maximal chain.
	If $u$ is elliptic, this is done in \Cref{lemma:elliptic-intervals}.
	If $u$ is hyperbolic and the Coxeter subgroup $W_u \subseteq W$ is irreducible, this is done in \Cref{lemma:hyperbolic-intervals}.
	Suppose now that $u$ is any hyperbolic element, and let $u = u'h$ be its hyperbolic-horizontal decomposition (see \Cref{lemma:hyperbolic-decomposition}).
	Then $u'$ is hyperbolic, the Coxeter subgroup $W_{u'} \subseteq W$ is irreducible, $h$ is horizontal elliptic, and $[1,u] = [1,u'] \times [1,h]$.
	We have already proved that $\lambda$ is an EL-labeling of $[1,u']$ and $[1,h]$, so it is an EL-labeling of $[1,u]$ by \Cref{thm:product-el-labelings}.
	In particular, $[1,u]$ has at most one increasing maximal chain.
\end{proof}

\section{Dual CW models for the orbit configuration spaces}
\label{sec:dual-salvetti-complex}

In this section, we introduce new finite CW models for the orbit configuration space $Y_W$ of a Coxeter group $W$.
Each of them is naturally included in the interval complex of one of the noncrossing partition posets associated with $W$.
We define them for any Coxeter group, not necessarily finite or affine.

\begin{lemma}
	\label{lemma:coxeter-reflection-length}
	Let $W$ be a Coxeter group.
	The reflection length of every Coxeter element is equal to the size of a set of simple reflections of $W$.
\end{lemma}

\begin{proof}
	Let $w = s_1s_2\dotsm s_n$ be a Coxeter element, where $S = \{s_1, s_2, \dotsc, s_n\}$ is a set of simple reflections.
	By the deletion condition \cite[Corollary 5.8]{humphreys1992reflection} and the fact that $S$ is a minimal generating set for $W$ \cite[Theorem 5.5]{humphreys1992reflection}, we obtain that $s_1s_2\dotsm s_n$ is a reduced expression for $w$ (meaning that $w$ cannot be written as a product of less than $n$ element of $S$).
	By \cite[Theorem 1.1]{dyer2001minimal}, the reflection length of $w$ is the smallest natural number $p$ such that $s_{i_1}s_{i_2}\dotsm s_{i_{n-p}} = 1$ for some choice of the indices $1 \leq i_1 < i_2 < \dotsb < i_{n-p} \leq n$.
	If $p < n$, then the relation $s_{i_1}s_{i_2}\dotsm s_{i_{n-p}} = 1$ allows to write a simple reflection as a product of other simple reflections, which is impossible because $S$ is a minimal generating set.
	Therefore $p=n$.
\end{proof}

\Cref{lemma:coxeter-reflection-length} was proved in \cite[Lemma 1.3.3]{bessis2003dual} for finite Coxeter groups, and in \cite[Proposition 7.2]{mccammond2015dual} for affine Coxeter groups.

Let $W$ be a Coxeter group, and $R$ its set of reflections.
Fix a set of simple reflections $S = \{s_1, s_2, \dotsc, s_n \} \subseteq R$, and a Coxeter element $w = s_1s_2 \dotsm s_n$.
Denote by $K_W$ the interval complex associated with the noncrossing partition poset $[1,w]^W$.
Let $X_W$ be the Salvetti complex of $W$, and recall from \Cref{sec:coxeter-artin} that its cells are indexed by the simplicial complex
\[ \Delta_W = \{ T \subseteq S \mid \text{the standard parabolic subgroup $W_T$ is finite} \}. \]

For every $T \in \Delta_W$, denote by $w_T$ the product of the elements of $T$ in the same relative order as in the list $s_1,s_2,\dotsc,s_n$.
Then $w_T$ is a Coxeter element of the parabolic subgroup $W_T$, and it belongs to $[1,w]^W$ by \Cref{lemma:coxeter-reflection-length} and \Cref{lemma:hurwitz-action}.

\begin{lemma}
	For every $T \subseteq S$ we have $[1,w_T]^{W_T} = [1,w_T]^W$, and the length functions of $W_T$ and $W$ agree on these intervals.
	\label{lemma:interval-in-parabolic-subgroup}
\end{lemma}

\begin{proof}
	By \cite[Corollary 1.4]{dyer2001minimal}, the length function of the parabolic subgroup $W_T$ agrees with the length function of $W$.
	By \Cref{lemma:coxeter-reflection-length}, $l(w_T) = |T|$.
	
	By \cite[Theorem 1.3]{baumeister2014note} (see also \cite[Theorem 1.4]{igusa2010exceptional}), the Hurwitz action is transitive on the minimal length factorizations of $w_T$ as a product of reflections of $W$.
	There is at least one minimal length factorization of $w_T$ that uses only reflections of $W_T$, and therefore this is true for all minimal length factorizations.
	This means that the interval $[1,w_T]$ is the same in $W_T$ (using the reflections of $W_T$ as the generating set) and in $W$.
\end{proof}

Thanks to the previous lemma, for every $T \subseteq S$ we can safely write $[1,w_T]$ in place of $[1,w_T]^{W_T} = [1,w_T]^{W}$, without the need to specify the ambient group.

\begin{definition}
	Let $X_W'$ be the finite subcomplex of $K_W$ consisting of the simplices $[x_1|x_2|\dotsb|x_d] \in K_W$ such that $x_1x_2\dotsm x_d \in [1,w_T]$ for some $T \in \Delta_W$.
	\label{def:dual-salvetti}
\end{definition}


\begin{remark}
	\label{rmk:finite-dual-salvetti-complex}
	If $W$ is finite, then $S \in \Delta_W$ and therefore $X_W' = K_W$.
	In this case, the interval complex $K_W$ is a classifying space for the dual Artin group $W_w$ (by \Cref{thm:garside-classifying-space,thm:finite-lattice}), which is naturally isomorphic to the Artin group $G_W$ (by \Cref{thm:dual-artin-group}).
\end{remark}

For every $T \in \Delta_W$, the complex $X_W'$ has a subcomplex consisting of the simplices $[x_1|x_2|\dotsb|x_d]$ such that $x_1x_2\dotsm x_d \in [1,w_T] = [1,w_T]^{W_T}$.
This is exactly the interval complex associated with $[1,w_T]^{W_T}$, which coincides with $X_{W_T}'$ and is a classifying space for the Artin group $G_{W_T}$ by \Cref{rmk:finite-dual-salvetti-complex}.

By definition, $X_W'$ is the union of all subcomplexes $X_{W_T}'$ for $T \in \Delta_W$.
Similarly, the Salvetti complex $X_W$ is the union of the Salvetti complexes $X_{W_T}$ for $T \in \Delta_W$.
Each $X_{W_T}$ is a classifying space for $G_{W_T}$, because the $K(\pi, 1)$ conjecture holds for spherical Artin groups \cite{deligne1972immeubles}.

\begin{theorem}
	For every Coxeter group $W$, the complex $X_W'$ is homotopy equivalent to the Salvetti complex $X_W$ and to the orbit configuration space $Y_W$.
	\label{thm:dual-salvetti}
\end{theorem}

\begin{proof}
	%
	%
	%
	Since $X_{W}$ is homotopy equivalent to $Y_W$, it is enough to show that $X_W'$ is homotopy equivalent to $X_W$.
	To keep our notation uncluttered, throughout this proof we indicate $X_W, X_W', X_{W_T}, X_{W_T}'$ by $X, X', X_T, X_T'$, respectively.
	
	For every $T \in \Delta_W$, both complexes $X_{T}$ and $X_{T}'$ are classifying spaces for the Artin group $G_{W_T}$.
	We are going to inductively construct homotopy equivalences $\varphi_T \colon X_{T} \to X'_{T}$ satisfying the following naturality property: for all $Q \subseteq T \in \Delta_W$, there is a commutative diagram
	\begin{center}
		\begin{tikzcd}
			X_{Q} \arrow[r, "\varphi_Q"] \arrow[d, hook]
			& X'_{Q} \arrow[d, hook] \\
			X_{T} \arrow[r, "\varphi_T"]
			& X'_{T}.
		\end{tikzcd}    
	\end{center}
	The construction of $\varphi_T$ is the following, assuming to have already constructed $\varphi_Q$ for all $Q \subsetneq T$.
	\begin{itemize}
		\item If $T = \emptyset$, then $X_{T}$ and $X'_{T}$ are single points, and there is only one map $\varphi_T$ between them.
		
		\item If $|T| = 1$, then both $X_{T}$ and $X'_{T}$ consist of one oriented $1$-cell attached to one $0$-cell.
		Define $\varphi_T$ as any orientation-preserving cellular homeomorphism.
		
		\item Let $|T| = 2$, say $T = \{s, s'\}$.
		We can extend $\varphi_{\{s\}} \cup \varphi_{\{s'\}} \colon X_{\{s\}} \cup X_{\{s'\}} \to X_T'$ to the $2$-cell of $X_T$, and obtain a map $\varphi_T\colon X_T \to X_T'$ such that the induced map $(\varphi_T)_* \colon \pi_1(X_{T}, X_{\emptyset}) \to \pi_1(X_{T}', X_{\emptyset}')$ is an isomorphism (see the proof of \cite[Proposition 1B.9]{hatcher}).
		Since $X_{T}$ and $X_T'$ are classifying spaces, we have that $\varphi_T$ is a homotopy equivalence.
		
		\item Let $|T| \geq 3$. By construction, the map $\bigcup_{Q\subsetneq T} \varphi_Q \colon \bigcup_{Q\subsetneq T} X_Q \to X_T'$ induces an isomorphism on the fundamental groups (which are both isomorphic to the Artin group $G_{W_T}$).
		Extend this map to the $|T|$-cell of $X_T$, to get a map $\varphi_T \colon X_T \to X_T'$ which also induces an isomorphism on the fundamental groups (as in the proof of \cite[Proposition 1B.9]{hatcher}).
		As before, $\varphi_T$ is a homotopy equivalence.
	\end{itemize}
	Gluing together all these maps, we obtain a map $\varphi\colon X_W \to X'_W$.
	This is a homotopy equivalence by a repeated application of the gluing theorem for adjunction spaces \cite[Theorem 7.5.7]{brown2006topology}.
\end{proof}

\begin{corollary}
	\label{cor:dual-fundamental-group}
	For every Coxeter group $W$, the fundamental group of $X_W'$ is isomorphic to the Artin group $G_W$.
\end{corollary}

\begin{remark}
	The complex $X_W'$ depends on the Coxeter element $w$ and on the set of simple reflections $S$.
	However, since all Coxeter elements of a finite parabolic subgroup $W_T$ are geometrically equivalent, the $f$-vector of $X_W'$ depends only on $W$ (it can be computed via inclusion-exclusion in terms of the subcomplexes $K_{W_T}$).
\end{remark}

In view of \Cref{thm:dual-salvetti}, the $K(\pi,1)$ conjecture holds for an Artin group $G_W$ if and only if $X_W'$ is a classifying space.
The following is an alternative characterization of the cells of the complex $X_W'$ in the affine case.

\begin{lemma}
	\label{lemma:affine-dual-salvetti-complex}
	Let $W$ be an irreducible affine Coxeter group, with a set $S$ of simple reflections and a Coxeter element $w$ obtained as a product of the elements of $S$.
	Denote by $C_0$ the chamber of the Coxeter complex associated with $S$.
	A simplex $[x_1|x_2|\dotsb|x_d] \in K_W$ belongs to $X_W'$ if and only if $x_1x_2\dotsm x_d$ is an elliptic element that fixes at least one vertex of $C_0$.
\end{lemma}

\begin{proof}
	For every subset $T \subseteq S$ with $|T| = |S|-1$, the fixed set of the parabolic Coxeter element $w_T$ is given by one of the vertices of $C_0$ by \Cref{lemma:model-poset}.
	Conversely, every vertex of $C_0$ is the fixed set of exactly one such parabolic Coxeter element.
	We conclude using \Cref{lemma:model-poset,lemma:elliptic-subposet}.
\end{proof}

\begin{example}[Dual complexes for $\tilde A_2$ and $\tilde G_2$]
	The $f$-vector of $X_W'$ is $(1,9,9)$ if $W$ is of type $\tilde A_2$, and $(1,11,11)$ if $W$ is of type $\tilde G_2$.
	For $\tilde A_2$, the complex $X_W'$ is explicitly described in \Cref{ex:fiber-components-A2} and \Cref{fig:fiber-components-A2}.
	Notice that the Salvetti complex $X_W$ is much smaller, as in both cases its $f$-vector is $(1,3,3)$.
\end{example}

\section{Classifying spaces for dual affine Artin groups}
\label{sec:classifying-spaces}

Let $W$ be an irreducible affine Coxeter group, and $w$ one of its Coxeter elements.
In this section, we prove that the interval complex $K_W$ associated with $[1,w]^W$ is a classifying space for the dual Artin group $W_w$.
This is somewhat surprising since the interval $[1,w]^W$ is not a lattice in general.

As usual, let $W$ act by Euclidean isometries on $\R^n$, where $n$ is the rank of $W$.
Let $R \subseteq W$ be the set of reflections, and denote by $R_\h$ and $R_\v$ the horizontal and vertical reflections of $[1,w]^W$, respectively.
To proceed, we briefly recall the construction of McCammond and Sulway that leads to braided crystallographic groups \cite{mccammond2017artin}.
For this, new groups of Euclidean isometries are introduced.

\begin{itemize}
	\item The \emph{diagonal group} $D$, generated by $R_\h$ and $T$.
	Here $T$ is the (finite) set of all translations of $[1,w]^W$.
	Translations are assigned a weight of $2$.
	\item The \emph{factorable group} $F$, generated by $R_\h$ and by a set $T_F$ of \emph{factored translations}.
	There are $k$ factored translations $t_1, \dotsc, t_k$ for each translation $t \in T$, and they satisfy $t_1\dotsm t_k = t$, where $k$ is the number of irreducible components of the horizontal root system $\Phi_\h$.
	Factored translations are assigned a weight of $\frac{2}{k}$.
	\item The \emph{crystallographic group} $C$, generated by $R_\h$, $R_\v$, and $T_F$.
\end{itemize}
The diagonal group $D$ is included in both $W$ and $F$, and all of them are included in the crystallographic group $C$.
By \cite[Lemma 7.2]{mccammond2017artin}, the associated intervals are related as follows:
\begin{align*}
[1,w]^C &= [1,w]^W \cup [1,w]^F \\
[1,w]^D &= [1,w]^W \cap [1,w]^F.
\end{align*}

The intervals $[1,w]^D$ and $[1,w]^F$ are finite, whereas $[1,w]^W$ and $[1,w]^C$ are infinite.
The factored translations are introduced so that the intervals $[1,w]^F$ and $[1,w]^C$ are balanced lattices \cite[Propositions 7.4 and 7.6, and Theorem 8.10]{mccammond2017artin}.
On the other hand, the intervals $[1,w]^D$ and $[1,w]^W$ are lattices if and only if the horizontal root system $\Phi_\h$ is irreducible (i.e.\ $k=1$), in which case $D=F$ and $W=C$.

From these new intervals, one can construct the interval groups $D_w$, $F_w$, and $C_w$. The group $C_w$ is called the \emph{braided crystallographic group}. 
The inclusions between the four intervals induce inclusions between the corresponding interval groups: $D_w \hookrightarrow W_w$, $D_w \hookrightarrow F_w$, $W_w \hookrightarrow C_w$, and $F_w \hookrightarrow C_w$ \cite[Theorem 9.6]{mccammond2017artin}.
Since the intervals $[1,w]^F$ and $[1,w]^C$ are lattices, the interval groups $F_w$ and $C_w$ are Garside groups (\Cref{thm:garside-structure}) and the corresponding interval complexes $K_F$ and $K_C$ are classifying spaces (\Cref{thm:garside-classifying-space}).
A consequence of the relations between the four intervals is that $K_C = K_W \cup K_F$ and $K_D = K_W \cap K_F$, where $K_D$ is the interval complex associated with $[1,w]^D$.
As noted in \cite[Proposition 11.1]{mccammond2017artin}, the cover of $K_C$ corresponding to the subgroup $W_w \subseteq C_w$ is a finite-dimensional classifying space for the (dual) Artin group $W_w$.
However this cover is not very explicit, and it is difficult to work with it in practice.

Let $H$ be the subgroup of $D$ generated by $R_\h$.
With the notation of \Cref{sec:horizontal-components} we have $H = W_1 \times \dotsb \times W_k$, where $W_i$ is the subgroup generated by the reflections associated with the $i$-th irreducible component $\Phi_i$ of the horizontal root system $\Phi_\h$.
Recall that $W_i$ is a Coxeter group of type $\tilde A_{n_i}$, where $n_i$ is the rank of $\Phi_i$.
By \cite[Proposition 7.6]{mccammond2017artin}, the horizontal part $[1,w]^W \cap H$ of the interval $[1,w]^W$ decomposes as
\begin{equation}
	\label{eq:horizontal-decomposition}
	[1,w]^W \cap H = \Big( [1,w]^W \cap W_1 \Big) \times \dotsb \times \Big( [1,w]^W \cap W_k \Big),
\end{equation}
and the single factors are described in \Cref{sec:horizontal-components} (a more direct proof of this decomposition can be derived in a similar way to the proof of \Cref{lemma:common-horizontal}).
Let $H_w$ be the group with generating set $R_\h$ and subject only to the relations visible in $[1,w]^W \cap H$.
It is a subgroup of $D_w$ \cite[Lemma 9.3]{mccammond2017artin}, and it decomposes as a direct product of $k$ Artin groups of types $\tilde A_{n_1}, \dotsc, \tilde A_{n_k}$, one for each irreducible component.

%
%

Let $K_H$ be the subcomplex of $K_D$ consisting of the simplices $[x_1|x_2|\dotsb|x_d]$ such that $x_1x_2\dotsm x_d \in H$.
The fundamental group of $K_H$ is naturally isomorphic to $H_w$.
Denote by $K_i$ the subcomplex of $K_H$ consisting of the simplices $[x_1|x_2|\dotsb|x_d]$ such that $x_1x_2\dotsm x_d \in [1,w]^W \cap W_i$.
The fundamental group of $K_i$ is an Artin group of type $\tilde A_{n_i}$.

\begin{lemma}
	$K_H$ is homeomorphic to $K_{1} \times \dotsb \times K_{k}$.
	\label{lemma:KH}
\end{lemma}

\begin{proof}
	We claim that $K_H$ is a triangulation of the natural cellular structure of $K_{1} \times \dotsb \times K_{k}$.
	To show this, we explicitly construct a homeomorphism $\psi \colon K_{1} \times \dotsb \times K_{k} \to K_H$.
	Consider a cell of $K_{1} \times \dotsb \times K_{k}$, which is a product of simplices
	\[ [x_{11}| x_{12}| \dotsb | x_{1d_1}] \times \dotsb \times [x_{k1}| x_{k2}| \dotsb | x_{kd_k}]. \]
	This is realized as $\Delta^{d_1} \times \dotsb \times \Delta^{d_k} \subseteq \R^{d_1+\dotsb+d_k}$.
	Consider a point $p \in \Delta^{d_1} \times \dotsb \times \Delta^{d_k}$, with coordinates given by
	\[ (a_{11},\dotsc, a_{1d_1},\dotsc,a_{k1},\dotsc,a_{kd_k}) \in \R^{d_1+\dotsb+d_k}. \]
	Assume for now that the coordinates of $p$ are pairwise distinct.
	Then there is a unique enumeration $\gamma \colon \{1,\dotsc,d_1+\dotsb+d_k\} \to \{11, \dotsc, 1d_1, \dotsc, k1, \dotsc, kd_k \}$ of the indices such that $a_{\gamma(1)} \geq a_{\gamma(2)} \geq \dotsb \geq a_{\gamma(d_1+\dotsb+d_k)}$.
	Notice that, in this enumeration, the relative order of indices in the same horizontal component is preserved.
	Define $\psi(p)$ as the point of $[x_{\gamma(1)}| x_{\gamma(2)} |\dotsb | x_{\gamma(d_1+\dotsb+d_k)}]$ with coordinates
	\[ (a_{\gamma(1)}, a_{\gamma(2)}, \dotsc, a_{\gamma(d_1+\dotsb+d_k)}) \in \Delta^{d_1+\dotsb+d_k}. \]
	If the coordinates of $p$ are not pairwise distinct, there are multiple choices for the enumeration $\gamma$, and any choice gives the same definition of $\psi(p)$.
	See \Cref{fig:triangulation-KH} for some examples.
	Using the decomposition \eqref{eq:horizontal-decomposition}, we obtain that $\psi$ is a homeomorphism.
\end{proof}

\tikzstyle{reverseclip}=[insert path={(current page.north east) --
	(current page.south east) --
	(current page.south west) --
	(current page.north west) --
	(current page.north east)}
]

\begin{figure}
	\raisebox{0.15cm}{
	\begin{tikzpicture}[scale=3]
	\begin{scope}[every path/.style={-{Latex[length=2mm,width=0.8mm]}}]
	\draw (0,0) -- (1,0) node[midway,below] {$[x]$};
	\draw (0,0) -- (0,1) node[midway,left] {$[y]$};
	\draw (1,0) -- (1,1) node[midway,right] {$[y]$};
	\draw (0,1) -- (1,1) node[midway,above] {$[x]$};
	
	\draw[dashed] (0,0) -- (1,1);
	\end{scope}
	
	\node at (0.7,0.25) {$[x|y]$};
	\node at (0.3,0.7) {$[y|x]$};
	\node[fill=white, inner sep=1pt] at (0.505,0.505) {$[xy]$};
	\end{tikzpicture}}\qquad
	\begin{tikzpicture}[scale=3]
	\newcommand{\xcoord}{1.2}
	\draw[name path=a, white] (1,0,0) -- (1,1,0);
	\draw[name path=b, white] (0,0,0) -- (\xcoord,0,-1);
	\draw[name path=c, white] (0,0,0) -- (\xcoord,1,-1);
	\draw[name path=d, white] (1,0,0) -- (\xcoord,1,-1);
	
	\path [name intersections={of=a and b,by=P}];
	\path [name intersections={of=b and d,by=Q}];
	\path [name intersections={of=a and c,by=R}];
	
	\fill[black!15] (0,1,0) -- (1,1,0) -- (\xcoord,1,-1) -- cycle;
	\fill[black!7] (0,0,0) -- (1,0,0) -- (\xcoord,0,-1) -- cycle;

	\begin{scope}[every path/.style={-{Latex[length=2mm,width=0.8mm]}}]
	
	\begin{scope}
	\begin{pgfinterruptboundingbox}
	\clip (P) circle (0.035) [reverseclip];
	\clip (Q) circle (0.035) [reverseclip];
	\end{pgfinterruptboundingbox}
	\draw[black!35] (0,0,0) -- (\xcoord,0,-1);
	\end{scope}
	
	\draw (0,0,0) -- (0,1,0) node[midway,left] {$[y]$};
	\draw (0,1,0) -- (1,1,0);
	\draw (0,1,0) -- (\xcoord,1,-1);
	\draw (1,1,0) -- (\xcoord,1,-1);
	\draw (\xcoord,0,-1) -- (\xcoord,1,-1) node[midway,right] {$[y]$};
	
	\draw[dashed, black!40] (0,0,0) -- (\xcoord,1,-1);
	
	\node[inner sep=1.6pt, fill=white, circle] at (R) {};

	\draw (1,0,0) -- (1,1,0);
	
	\draw[dashed] (0,0,0) -- (1,1,0);
	\draw[dashed] (1,0,0) -- (\xcoord,1,-1);
	\draw (0,0,0) -- (1,0,0);
	\draw (1,0,0) -- (\xcoord,0,-1);
	\node[fill=white, inner sep=1.8pt] at (1,0.5,0) {$[y]$};
	\end{scope}
	
	\node at (0.8, -0.15, 0) {$[x_1|x_2]$};
	\node[fill=white, inner sep=1pt, opacity=0.9] at (0.8, 1.25, 0) {$[x_1|x_2]$};
	\end{tikzpicture}
	\caption{
		On the left, triangulation of a cell $[x] \times [y]$ of $K_1 \times K_2$.
		It is homeomorphic to $[x|y] \cup [y|x] \subseteq K_H$.
		On the right, triangulation of a cell $[x_1|x_2] \times [y]$, which is homeomorphic to $[x_1|x_2|y] \cup [x_1|y|x_2] \cup [y|x_1|x_2] \allowbreak\subseteq K_H$.
	}
	\label{fig:triangulation-KH}
\end{figure}
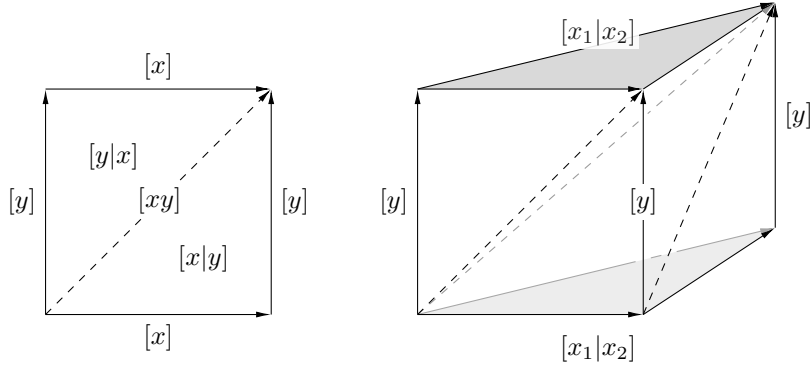

Denote by $\varphi\colon [1,w]^W \to [1,w]^W$ the conjugation by $w$: $\varphi(u) = w^{-1} u w$.

\begin{lemma}
	$K_D$ is homeomorphic to $K_H \times [0,1]\,/ {\sim}$, where the relation $\sim$ identifies $[x_1|x_2|\dotsb|x_d] \times \{1\}$ and $[\varphi(x_1) | \varphi(x_2) | \dotsb | \varphi(x_d)] \times \{0\}$ for every simplex $[x_1|x_2|\dotsb|x_d]$ of $K_H$.
	\label{lemma:KD}
\end{lemma}

\begin{proof}
	Let $Z = K_H \times [0,1] \, / {\sim}$.
	Similarly to \Cref{lemma:KH}, we show that $K_D$ is a triangulation of the natural cell structure of $Z$, by explicitly constructing a homeomorphism $\psi\colon Z \to K_D$.
	Notice that $K_H \times \{0\} \subseteq Z$ is naturally included in $K_D$: we define $\psi|_{K_H \times \{0\}}$ as the natural homeomorphism with $K_H \subseteq K_D$.
	Consider now a cell of $K_H \times [0,1]$ of the form
	\[ [x_1|x_2|\dotsb|x_d] \times [0,1]. \]
	This cell is realized as $\Delta^d \times [0,1] \subseteq \R^{d+1}$.
	Then a point $p$ in this cell has coordinates $(a_1,a_2,\dotsc,a_d,t)$, with $1 \geq a_1 \geq \dotsb \geq a_d \geq 0$ and $t \in [0,1]$.
	Let $y$ be the right complement of $x_1x_2\dotsm x_d$, so that $x_1x_2\dotsm x_d y = w$.
	Notice that $y$ is hyperbolic, so in particular $y \neq 1$.
	Assume for now that none of the coordinates $a_1,\dotsc, a_d$ is equal to $1-t$.
	Then there is a unique index $i \in \{0,\dotsc,d\}$ such that $a_1 \geq \dotsb \geq a_i \geq 1-t \geq a_{i+1} \geq \dotsb \geq a_d$.
	Let $\sigma = [x_{i+1}|\dotsb |x_d|y|\varphi(x_1)|\dotsb|\varphi(x_{i})]$.
	Define $\psi(p)$ as the point of $\sigma$ with coordinates $(t+a_{i+1},\dotsc,t+a_d,t,t+a_1-1,\dotsc,t+a_i-1) \in \Delta^{d+1}$.
	If some of the coordinates $a_1,\dotsc, a_d$ are equal to $1-t$, choose any index $i$ as above. Different choices of $i$ give the same point of $K_D$.
	See \Cref{fig:triangulation} for some examples.
	
	The definition of $\psi$ on different cells is coherent.
	We only explicitly check that the definition on $[x_1|x_2|\dotsb|x_d] \times [0,1]$ agrees with the definition on $[x_1|x_2|\dotsb|x_d] \times \{1\}$, since this is where the non-trivial gluing occurs.
	Consider a point $p \in [x_1|x_2|\dotsb|x_d] \times \{1\}$, with coordinates $(a_1,\dotsc,a_d,1) \in \Delta^d \times [0,1]$.
	\begin{itemize}
		\item Since $[x_1|\dotsb|x_d]\times \{1\}$ is identified with $[\varphi(x_1)|\dotsb|\varphi(x_d)] \times \{0\}$, we have that $\psi(p)$ is the point of $[\varphi(x_1)|\dotsb|\varphi(x_d)]$ with coordinates $(a_1,\dotsc,a_d)$.
		
		\item As an element of $[x_1|\dotsb|x_d] \times [0,1]$, the same point $p$ is sent to the point of $[y|\varphi(x_1)|\dotsb|\varphi(x_d)]$ with coordinates $(1,a_1,\dotsc,a_d)$.
		By definition of the faces in an interval complex (\Cref{def:interval-complex}), this point is the same as the point of $[\varphi(x_1)|\dotsb|\varphi(x_d)]$ with coordinates $(a_1,\dotsc,a_d)$.
	\end{itemize}
	Therefore the two definitions of $\psi$ agree in this case.
	
	Every maximal cell of $K_D$ is of the form $[x_1|\dotsb|x_i|y|x_{i+1}'|\dotsb|x_d']$ where $y$ is hyperbolic and each $x_j$ and $x_j'$ is horizontal elliptic.
	Thus $\psi$ is a homeomorphism.
\end{proof}

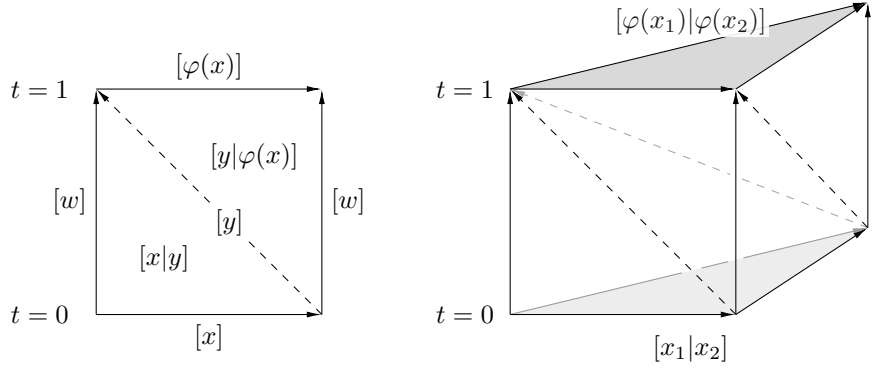
\begin{figure}
	\raisebox{0.15cm}{
	\begin{tikzpicture}[scale=3]
		\begin{scope}[every path/.style={-{Latex[length=2mm,width=0.8mm]}}]
			\draw (0,0) -- (1,0) node[midway,below] {$[x]$};
			\draw (0,0) -- (0,1) node[midway,left] {$[w]$};
			\draw (1,0) -- (1,1) node[midway,right] {$[w]$};
			\draw (0,1) -- (1,1) node[midway,above] {$[\varphi(x)]$};
			
			\draw[dashed] (1,0) -- (0,1);
		\end{scope}
		
		\node at (0.3,0.25) {$[x|y]$};
		\node at (0.7,0.7) {$[y|\varphi(x)]$};
		\node[fill=white, inner sep=1pt] at (0.59,0.41) {$[y]$};
		
		\node at (-0.25,0,0) {$t=0$};
		\node at (-0.25,1,0) {$t=1$};
	\end{tikzpicture}}\qquad
	\begin{tikzpicture}[scale=3]
		\newcommand{\xcoord}{1.2}
		\draw[name path=a, white] (1,0,0) -- (1,1,0);
		\draw[name path=b, white] (0,0,0) -- (\xcoord,0,-1);
		\draw[name path=c, white] (1,0,0) -- (0,1,0);
		\draw[name path=d, white] (\xcoord,0,-1) -- (0,1,0);
		
		\path [name intersections={of=a and b,by=P}];
		\path [name intersections={of=b and c,by=Q}];
		\path [name intersections={of=a and d,by=R}];
		
		\fill[black!15] (0,1,0) -- (1,1,0) -- (\xcoord,1,-1) -- cycle;
		\fill[black!7] (0,0,0) -- (1,0,0) -- (\xcoord,0,-1) -- cycle;

		\begin{scope}[every path/.style={-{Latex[length=2mm,width=0.8mm]}}]
			\begin{scope}
				\begin{pgfinterruptboundingbox}
					\clip (P) circle (0.035) [reverseclip];
					\clip (Q) circle (0.045) [reverseclip];
				\end{pgfinterruptboundingbox}
				\draw[black!40] (0,0,0) -- (\xcoord,0,-1);
			\end{scope}
			
			\draw (0,0,0) -- (0,1,0);
			\draw (0,1,0) -- (1,1,0);
			\draw (0,1,0) -- (\xcoord,1,-1);
			\draw (1,1,0) -- (\xcoord,1,-1);
			\draw (\xcoord,0,-1) -- (\xcoord,1,-1);
						
			\draw[dashed, black!35] (\xcoord,0,-1) -- (0,1,0);
			
			\node[inner sep=1.6pt, fill=white, circle] at (R) {};

			\draw (1,0,0) -- (1,1,0);
			
			\draw[dashed] (1,0,0) -- (0,1,0);
			\draw[dashed] (\xcoord,0,-1) -- (1,1,0);
			\draw (0,0,0) -- (1,0,0);
			\draw (1,0,0) -- (\xcoord,0,-1);
		\end{scope}
		
		\node at (-0.2,0,0) {$t=0$};
		\node at (-0.2,1,0) {$t=1$};
		\node at (0.8, -0.15, 0) {$[x_1|x_2]$};
		\node[fill=white, inner sep=1pt, opacity=0.9] at (0.8, 1.3, 0) {$[\varphi(x_1)|\varphi(x_2)]$};
	\end{tikzpicture}
	\caption{
		On the left, triangulation of a cell $[x] \times [0,1]$ of $K_H \times [0,1]\,/ {\sim}$.
		It is homeomorphic to $[x|y] \cup [y|\varphi(x)] \subseteq K_D$, where $y$ is the right complement of $x$.
		On the right, triangulation of a cell $[x_1|x_2] \times [0,1]$, which is homeomorphic to $[x_1|x_2|y] \cup [x_2|y|\varphi(x_1)] \cup [y|\varphi(x_1)|\varphi(x_2)] \subseteq K_D$, where $y$ is the right complement of $x_1x_2$.
	}
	\label{fig:triangulation}
\end{figure}

An immediate consequence of \Cref{lemma:KD} is that $D_w = \Z \ltimes H_w$, where $\Z$ is the cyclic subgroup of $D_w$ generated by $w$.

\begin{lemma}
	$K_{i}$ is a classifying space for the affine Artin group of type $\tilde A_{n_i}$.
	\label{lemma:KHi}
\end{lemma}

\begin{proof}
	By \Cref{lemma:common-horizontal}, we can assume until the end of this proof that $H = W_i$ is the unique irreducible horizontal component of $W$.
	Therefore $[1,w]^D = [1,w]^F$ is a lattice, $D_w$ is a Garside group, and $K_D$ is a classifying space for $D_w$.
	By \Cref{lemma:KD}, there is a covering map
	\[ \rho\colon K_H \times \R \to K_D \]
	which corresponds to the subgroup $H_w$ of $D_w$.
	Then $K_H \times \R$ is a classifying space.
	Since $K_H \times \R \simeq K_H = K_{i}$, also $K_{i}$ is a classifying space.
\end{proof}

\begin{remark}
	The complex $K_{i}$ is closely related to the complexes $X_{W_i}'$ introduced in \Cref{sec:dual-salvetti-complex}.
	Indeed, $K_{i}$ is obtained by gluing the interval complexes associated with the noncrossing partition lattices of type $A_{n_i}$ corresponding to the maximal proper standard parabolic subgroups of $W_i$.
	However, the Coxeter elements of these parabolic subgroups are not below a common Coxeter element of $\tilde A_{n_i}$ (see also \Cref{rmk:horizontal-loop}).
\end{remark}

\begin{theorem}
	$K_D$ is a classifying space for $D_w$.
	\label{thm:KD-classifying}
\end{theorem}

\begin{proof}
	By \Cref{lemma:KH}, we have that $K_H \cong K_{1} \times \dotsb \times K_{k}$.
	Each factor is a classifying space by \Cref{lemma:KHi}, therefore $K_H$ is a classifying space for $H_w$.
	As discussed in the proof of \Cref{lemma:KHi}, there is a covering map $\rho \colon K_H \times \R \to K_D$ by \Cref{lemma:KD}.
	Therefore $K_D$ is a classifying space.
\end{proof}

We can finally prove that $K_W$ is a classifying space.

\begin{theorem}
	Let $W$ be an irreducible affine Coxeter group, and $w$ one of its Coxeter elements.
	The interval complex $K_W$ is a classifying space for the dual Artin group $W_w$.
	\label{thm:KW-classifying}
\end{theorem}

\begin{proof}
	Consider the universal cover $\rho \colon \tilde K_C \to K_C$ of the interval complex $K_C$.
	Recall that $K_D = K_W \cap K_F$ and $K_C = K_W \cup K_F$, and therefore $\rho^{-1}(K_D) = \rho^{-1}(K_W) \cap \rho^{-1}(K_F)$ and $\tilde K_C = \rho^{-1}(K_C) = \rho^{-1}(K_W) \cup \rho^{-1}(K_F)$.
	Then there is a Mayer-Vietoris long exact sequence
	\[ \dotsb \to H_i(\rho^{-1}(K_D)) \to H_i(\rho^{-1}(K_W)) \oplus H_i(\rho^{-1}(K_F)) \to H_i(\tilde K_C) \to \dotsb \]
	where all homology groups are with integer coefficients.
	Since $D_w = \pi_1(K_D)$ is a subgroup of $C_w = \pi_1(K_C)$, we have that $\rho^{-1}(K_D)$ is a union of (infinitely many) disjoint copies of the universal cover $\tilde K_D$ of $K_D$.
	Similarly, $\rho^{-1}(K_W)$ is a union of disjoint copies of the universal cover $\tilde K_W$ of $K_W$, and $\rho^{-1}(K_F)$ is a union of disjoint copies of the universal cover $\tilde K_F$ of $K_F$.
	Since $K_F$ and $K_C$ are classifying spaces, both $\tilde K_F$ and $\tilde K_C$ are contractible.
	By \Cref{thm:KD-classifying}, $\tilde K_D$ is also contractible.
	Then, for $i \geq 1$, the homology groups $H_i(\rho^{-1}(K_D))$, $H_i(\rho^{-1}(K_F))$, and $H_i(\tilde K_C)$ vanish.
	By the Mayer-Vietoris long exact sequence, $H_i(\rho^{-1}(K_W))$ also vanishes for $i \geq 1$.
	This means that $\tilde K_W$ has a trivial reduced homology, so it is contractible by a standard application of the Whitehead and Hurewicz theorems (see \cite[Corollary 4.33]{hatcher}).
\end{proof}

\section{Finite classifying spaces}
\label{sec:finite-classifying-spaces}

Let $W$ be an irreducible affine Coxeter group, with a fixed Coxeter element $w$.
In this section we show that the interval complexes $K_W$ and $K_C$ deformation retract onto finite subcomplexes $K_W' \subseteq K_W$ and $K_C' \subseteq K_C$.
In the case of $K_W$, this is an intermediate step to prove the $K(\pi,1)$ conjecture, whereas for $K_C$ this proves that the braided crystallographic group $C_w$ has a classifying space $K_C'$ with a finite number of cells.
The notation is the same as in \Cref{sec:classifying-spaces}.

As noted in the proof of \cite[Lemma 7.2]{mccammond2017artin}, there is no minimal length factorization of $w$ in $C$ that includes both a factored translation and a vertical reflection.
Recall that every element of $[1,w]^C \setminus [1,w]^W$ is hyperbolic.

\begin{lemma}
	Let $\sigma = [x_1|x_2|\dotsb|x_d]$ be a $d$-simplex of $K_C$, with $d \geq 1$.
	Then exactly one of the following occurs:
	\begin{enumerate}[(i)]
		\item every $x_i$ is elliptic, and at least one is vertical;
		\item every $x_i$ is horizontal elliptic or hyperbolic.
	\end{enumerate}
	\label{lemma:simplices-classification}
\end{lemma}

\begin{proof}
	We shall divide the proof into four cases.
	\begin{itemize}
		\item If at least one $x_i$ is not in $[1,w]^W$, then no minimal length factorization of any $x_i$ includes vertical reflections. In particular, no $x_i$ is vertical elliptic, so (ii) holds and (i) does not.
		In the remaining cases, assume that $x_i \in [1,w]^W$ for all $i$.
		
		\item If every $x_i$ is horizontal elliptic, then (ii) holds and (i) does not.
		
		\item Suppose that $x_j$ is hyperbolic for some index $j$. Then $x_1\dotsm x_j$ is also hyperbolic, and therefore its right complement $y$ is horizontal elliptic. Every $x_i$ for $i>j$ is below $y$ in $[1,w]^W$, and so is horizontal elliptic.
		By a similar argument, every $x_i$ for $i<j$ is horizontal elliptic.
		Then (ii) holds and (i) does not.
		
		\item If there is at least one vertical elliptic element and there are no hyperbolic elements, (i) holds and (ii) does not. \qedhere
	\end{itemize}
\end{proof}

Recall that $\F(K_C)$ denotes the face poset of $K_C$.
Consider the poset map $\eta \colon \F(K_C) \to \N$ defined by
\[ \eta([x_1|x_2|\dotsb|x_d]) =
\begin{cases}
d & \text{if $x_1 x_2 \dotsm x_d = w$} \\
d+1 & \text{otherwise}.
\end{cases}
\]
We want to describe the connected components of a fiber $\eta^{-1}(d)$ in the Hasse diagram of $\F(K_C)$.
Let $\sigma, \tau$ be two simplices in the same fiber $\eta^{-1}(d)$. We have that $\tau$ is a face of $\sigma$ if and only if $\sigma = [x_1|x_2|\dotsb|x_d]$ with $x_1 x_2 \dotsm x_d=w$ and either $\tau = [x_2|x_3|\dotsb |x_d]$ or $\tau = [x_1|x_2|\dotsb|x_{d-1}]$.
Therefore, a connected component of $\eta^{-1}(d)$ has the following form:
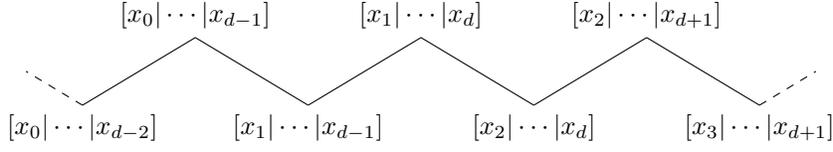
\newcommand{\height}{-1.5}
\begin{center}
	\begin{tikzpicture}
	\node (0) at (0,0) {$[x_1|\dotsb|x_d]$};
	\node (1) at (1.5, \height) {$[x_2|\dotsb|x_d]$};
	\node (2) at (3, 0) {$[x_2|\dotsb|x_{d+1}]$};
	\node (3) at (4.5, \height) {$[x_3|\dotsb|x_{d+1}]$};
	\node (4) at (6,0) {\phantom{$[\,]$}};
	\coordinate (3a) at ($(3)!0.5!(4)$) {};

	\node (-1) at (-1.5, \height) {$[x_1|\dotsb|x_{d-1}]$};
	\node (-2) at (-3, 0) {$[x_0|\dotsb|x_{d-1}]$};
	\node (-3) at (-4.5, \height) {$[x_0|\dotsb|x_{d-2}]$};
	\node (-4) at (-6, 0) {\phantom{$[\,]$}};
	\coordinate (-3a) at ($(-3)!0.5!(-4)$) {};
	
	\draw (0.south) -> (1.north);
	\draw (2.south) -> (1.north);
	\draw (2.south) -> (3.north);
	\draw[dashed] (3.north) -> (3a);
	
	\draw (0.south) -> (-1.north);
	\draw (-2.south) -> (-1.north);
	\draw (-2.south) -> (-3.north);
	\draw[dashed] (-3.north) -> (-3a);
	\end{tikzpicture}
\end{center}
where $x_i x_{i+1}\dotsm x_{i+d-1} = w$ for all $i$.
Define a \emph{$d$-fiber component} (or simply \emph{fiber component}) as a connected component of the fiber $\eta^{-1}(d)$.
As described above, a $d$-fiber component has an associated sequence $(x_i)_{i\in\Z}$ of elements of $[1,w]^C$ such that the product of any $d$ consecutive elements is $w$.
This sequence is well-defined up to a translation of the indices.

Let $\varphi\colon [1,w]^C \to [1,w]^C$ be the conjugation by the Coxeter element $w$: $\varphi (u) = w^{-1}uw$.
Notice that, if $(x_i)_{i\in\Z}$ is the sequence associated with a $d$-fiber component, we have $\varphi(x_i) = x_{i+d}$ for all $i\in \Z$.
Since $\varphi$ restricts to a map $[1,w]^W \to [1,w]^W$, every fiber component is either disjoint from $\F(K_W)$ or contained in $\F(K_W)$, where $\F(K_W)$ is the face poset of $K_W$.
The same is true for the subcomplexes $K_D$ and $K_F$.

\begin{lemma}
	Let $u \in [1,w]^C$. The set $\{\varphi^j(u) \mid j \in \Z\}$ is infinite if and only if $u$ is vertical elliptic.
	\label{lemma:conjugation}
\end{lemma}

\begin{proof}
	Suppose that $\{\varphi^j(u) \mid j \in \Z\}$ is infinite.
	Since $[1,w]^C$ has only a finite number of horizontal elliptic and hyperbolic elements, at least one element of $\{\varphi^j(u) \mid j \in \Z\}$ is vertical elliptic. Then $u$ is vertical elliptic.
	
	Conversely, let $u$ be a vertical elliptic element, and suppose by contradiction that $\varphi^j(u) = u$ for some $j \in \Z \setminus \{0\}$.
	Let $w^p$ be a power of $w$ that acts as a translation in the positive direction of the Coxeter axis $\ell$, where $p$ is a positive integer (see \Cref{lemma:translation}).
	We have that $\varphi^{pj}(u) = u$, so $u$ commutes with the (non-trivial) translation $w^{pj}$.
	By \cite[Lemma 11.3]{mccammond2017artin}, $\Fix(u)$ is invariant under $w^{pj}$.
	Then $\Dir(\ell) \subseteq \Dir(\Fix(u))$.
	We have that $\Mov(u) = \Dir(\Mov(u))$ because $u$ is elliptic, and that $\Dir(\Mov(u))$ is orthogonal to $\Dir(\Fix(u))$ by \cite[Lemma 3.6]{brady2015factoring}.
	Therefore $u$ is horizontal, and this is a contradiction.
\end{proof}

\begin{lemma}
	Let $\sigma \in \F(K_C)$.
	The fiber component containing $\sigma$ is infinite if and only if $\sigma$ is of type (i) in \Cref{lemma:simplices-classification}.
	In particular:
	\begin{itemize}
		\item every infinite fiber component is contained in $\F(K_W)$;
		\item $\F(K_F)$ is the union of all finite fiber components;
		\item $\F(K_D)$ is the union of all finite fiber components of $\F(K_W)$.
	\end{itemize}
	\label{lemma:component-finiteness}
\end{lemma}

\begin{proof}
	Let $\sigma = [x_1|x_2|\dotsb|x_d]$.
	If $\sigma$ is of type (i), at least one $x_i$ is vertical elliptic. Then the set $\{ x_{i+jd} = \varphi^j(x_i) \mid j \in \Z \}$ is infinite by \Cref{lemma:conjugation}, so the fiber component of $\sigma$ is infinite.
	The simplices of type (i) are in $\F(K_W)$, so the component is contained in $\F(K_W)$.
	
	If $\sigma$ is of type (ii), then every $x_i$ is horizontal elliptic or hyperbolic, so $\sigma \in \F(K_F)$.
	Since $\F(K_F)$ is finite, the fiber component of $\sigma$ is finite.
	Conversely, every simplex of $\F(K_F)$ is of type (ii), and so its fiber component is finite.
	
	The very last point follows from the fact that $K_D = K_W \cap K_F$.
\end{proof}

\begin{lemma}
	Let $\CC \subseteq \F(K_W)$ be a finite $d$-fiber component.
	Then there exists a simplex $[x_1|x_2|\dotsb|x_{d-1}] \in \CC$ such that $x_1x_2\dotsm x_{d-1}$ is horizontal elliptic.
	\label{lemma:horizontal-elliptic-product}
\end{lemma}

\begin{proof}
	Consider any $d$-simplex $\sigma = [x_1|x_2|\dotsb|x_d] \in \CC$, with $x_1 x_2 \dotsm x_d = w$.
	Since $\CC$ is finite, at least one $x_i$ is hyperbolic by \Cref{lemma:component-finiteness}.
	Suppose without loss of generality that $x_d$ is hyperbolic.
	Then its left complement $x_1x_2\dotsm x_{d-1}$ is horizontal elliptic (here we are using the fact that $\sigma \in \F(K_W)$).
	This completes the proof, because $[x_1|x_2|\dotsb|x_{d-1}] \in \CC$.
\end{proof}

\begin{lemma}
	Let $\CC \subseteq \F(K_W)$ be an infinite $d$-fiber component.
	Then there exists a simplex $[x_1|x_2|\dotsb|x_{d-1}] \in \CC$ such that $x_1x_2\dotsm x_{d-1}$ is vertical elliptic.
	\label{lemma:vertical-elliptic-product}
\end{lemma}

\begin{proof}
	Consider any $d$-simplex $[x_1|x_2|\dotsb|x_d] \in \CC$, with $x_1 x_2 \dotsm x_d = w$.
	Since $\CC$ is infinite, at least one $x_i$ is vertical elliptic by \Cref{lemma:component-finiteness}.
	Suppose without loss of generality that $x_d$ is vertical elliptic.
	Then its left complement $x_1x_2\dotsm x_{d-1}$ is also vertical elliptic, and $[x_1|x_2|\dotsb|x_{d-1}] \in \CC$.
\end{proof}

From now on, fix an axial chamber $C_0$ of the Coxeter complex.
If $S$ is the set of simple reflections associated with $C_0$, the Coxeter element $w$ can be written as the product of the elements of $S$ by \Cref{thm:coxeter-A} (for the case $\tilde A_n$) and \Cref{thm:bipartite-coxeter} (for the other cases).
Let $X_W' \subseteq K_W$ be the complex introduced in \Cref{def:dual-salvetti}.
By \Cref{lemma:affine-dual-salvetti-complex}, it consists of the simplices $[x_1|x_2|\dotsb|x_d]$ of $K_W$ such that $x_1x_2\dotsm x_d$ fixes a vertex of $C_0$.

\begin{lemma}
	Let $\CC \subseteq \F(K_W)$ be a $d$-fiber component.
	Then there exists a simplex $[x_1|x_2|\dotsb|x_{d-1}] \in \CC$ such that $x_1 x_2 \dotsm x_{d-1}$ is elliptic and fixes a vertex of $C_0$.
	In other words, $\CC \cap \F(X_W') \neq \emptyset$.
	\label{lemma:elliptic-product-vertex}
\end{lemma}%

\begin{proof}
	By \Cref{lemma:vertical-elliptic-product,lemma:horizontal-elliptic-product}, there exists a simplex $\sigma = [x_1|x_2|\dotsb|x_{d-1}] \in \CC$ such that $x_1 x_2\dotsm x_{d-1}$ is elliptic.
	By \Cref{lemma:fixed-axial-vertex}, $x_1 x_2\dotsm x_{d-1}$ fixes an axial vertex.
	By \Cref{prop:axial-orbits-A} (for the case $\tilde A_n$) and \Cref{rmk:axial-orbits-bipartite} (for the other cases), every axial vertex can be written uniquely as $w^j(b)$ for some vertex $b$ of $C_0$.
	Then, up to a conjugation by a power of $w$ (i.e.\ up to a translation of the indices in the sequence $(x_i)_{i\in \Z}$), we can assume that $x_1 x_2 \dotsm x_{d-1}$ fixes a vertex of $C_0$.
\end{proof}

\begin{corollary}
	\label{cor:finite-number-of-components}
	The face poset $\F(K_C)$ of the interval complex $K_C$ contains only a finite number of fiber components.
\end{corollary}

\begin{proof}
	Every fiber component contained in $\F(K_W)$ intersects $\F(X_W')$ by \Cref{lemma:elliptic-product-vertex}.
	Since $\F(X_W')$ is finite, $\F(K_W)$ contains only a finite number of fiber components.
	If $\CC$ is a fiber component not contained in $\F(K_W)$, by \Cref{lemma:component-finiteness} we have that $\CC$ is finite and $\CC\subseteq \F(K_F)$.
	Since $\F(K_F)$ is finite, it contains only a finite number of fiber components.
\end{proof}

We are finally able to show that $K_C$ and $K_W$ deformation retract onto finite subcomplexes $K_C' \subseteq K_C$ and $K_W' \subseteq K_W$, respectively.

\begin{definition}
	Let $K$ be either $K_C$ or $K_W$.
	A \emph{nice subcomplex} of $K$ is a subcomplex $K' \subseteq K$ such that
	\begin{enumerate}
		\item every finite fiber component $\CC \subseteq \F(K)$ is also contained in $\F(K')$;
		\item for every infinite fiber component $\CC \subseteq \F(K)$, the intersection $\CC \cap \F(K')$ is non-empty and its Hasse diagram is connected.
	\end{enumerate}
	\label{def:nice-subcomplex}
\end{definition}

\begin{theorem}
	Let $K$ be either $K_C$ or $K_W$.
	\begin{enumerate}[(a)]
		\item $K$ deformation retracts onto every nice subcomplex $K'$.
		\item Finite nice subcomplexes of $K$ exist.
	\end{enumerate}
	\label{thm:nice-subcomplex}
\end{theorem}

\begin{proof}
	For part (a), on every infinite fiber component $\CC$ consider the only acyclic matching $\M_\CC$ with critical simplices given by $\CC \cap \F(K')$. Existence and uniqueness of $\M_\CC$ follow from the fact that $\CC \cap \F(K')$ is non-empty and its Hasse diagram is connected.
	By \Cref{thm:patchwork}, the union of the matchings $\M_\CC$ is an acyclic matching with the desired set of critical simplices.
	This matching is also proper.
	We conclude using the main theorem of discrete Morse theory (\Cref{thm:discrete-morse-theory}).
	
	For part (b), recall that $\F(K)$ has only a finite number of (finite or infinite) fiber components by \Cref{cor:finite-number-of-components}.
	Then it is enough to inductively choose a finite non-empty interval in the Hasse diagram of every infinite $d$-fiber component $\CC$, starting from $d=\dim(K)$ (the highest possible value of $d$) and going down to $d=1$, so that every simplex in the boundary of a chosen simplex is also chosen.
\end{proof}

Since $K_C$ is a classifying space for $C_w$, we immediately obtain the following.

\begin{theorem}
	Let $W$ be an irreducible affine Coxeter group, and $w$ one of its Coxeter elements.
	The braided crystallographic group $C_w$ admits a classifying space with a finite number of cells.
	\label{thm:finite-crystallographic}
\end{theorem}

We end this section by noticing that there is a canonical choice of a nice subcomplex $K'$ of $K$ (where $K$ is either $K_W$ or $K_C$): for every infinite fiber component, $K'$ contains all the simplices between the first and the last simplex belonging to $X_W'$; in addition, $K'$ contains all the finite fiber components of $K$.

\begin{lemma}
	Let $K$ be either $K_C$ or $K_W$.
	Then $K'$ is a nice subcomplex of $K$.
	\label{lemma:canonical-nice}
\end{lemma}

\begin{proof}
	First, we check that $K'$ is a subcomplex.
	The finite fiber components of $K$ form a subcomplex ($K_D$ or $K_F$) by \Cref{lemma:component-finiteness}.
	Let $\sigma$ be a simplex of $K'$ which belongs to an infinite fiber component $\CC$.
	Then, in $\CC$, the simplex $\sigma$ is between two simplices $\sigma_1,\sigma_2 \in X_W'$ (if $\sigma \in X_W'$, we have $\sigma_1= \sigma_2= \sigma$).
	Then a face $\tau$ of $\sigma$ is either in the same fiber component $\CC$, or is between two faces of $\sigma_1$ and $\sigma_2$, in the fiber component of $\tau$. Therefore $\tau \in \F(K')$.
	
	By \Cref{lemma:component-finiteness,lemma:elliptic-product-vertex}, every infinite fiber component contains at least one simplex of $X_W'$.
	Then $K'$ satisfies the conditions of \Cref{def:nice-subcomplex}.
\end{proof}

We call this subcomplex $K'$ the \emph{canonical nice subcomplex} of $K$.
By construction, it contains $X_W'$ as a subcomplex.
In general, $K'$ is not the smallest nice subcomplex of $K$.

\begin{figure}
	\renewcommand{\height}{-1.5}
	\newcommand{\length}{0.95}
	\begin{center}
		\begin{tikzpicture}
		\node[inner sep=5pt] (0) at (0,0) {$[w]$};
		\node[inner sep=5pt] (1) at (0, \height) {$[\,]$};
		\draw (0.south) -> (1.north);
		\node[fill=white, draw, circle, inner sep=1.4pt] at (0.south) {};
		\node[fill=black, draw, circle, inner sep=1.4pt] at (1.north) {};
		\end{tikzpicture}\qquad\quad
		\begin{tikzpicture}[every node/.style={inner sep=5pt}]
		\node (0) at (0,0) {$[c_2c_0|b']$};
		\node (1) at (\length, \height) {$[b']$};
		\node (2) at (2*\length, 0) {$[b'|a_1a_{-1}]$};
		\node (3) at (3*\length, \height) {$[a_1a_{-1}]$};
		\node (4) at (4*\length, 0) {$[a_1a_{-1} | b]$};
		
		\node (-1) at (-\length, \height) {$[c_2c_0]$};
		\node (-2) at (-2*\length, 0) {$[b | c_2c_0]$};
		\node (-3) at (-3*\length, \height) {$[b]$};

		\draw[name path=a, white] (0.south) -> (1.north);
		\draw[name path=b, white] (2.south) -> (1.north);
		\draw[name path=c, white] (2.south) -> (3.north);
		
		\draw[name path=d, white] (0.south) -> (-1.north);
		\draw[name path=e, white] (-2.south) -> (-1.north);
	
		\draw[name path=z] (4.south) -> (-3.north);
		\begin{scope}[every node/.style={inner sep=2pt, fill=white, circle}]
			\path [name intersections={of=a and z,by=A}];
			\node at (A) {};
			\path [name intersections={of=b and z,by=B}];
			\node at (B) {};
			\path [name intersections={of=c and z,by=C}];
			\node at (C) {};
			\path [name intersections={of=d and z,by=D}];
			\node at (D) {};
			\path [name intersections={of=e and z,by=E}];
			\node at (E) {};
		\end{scope}

		\draw (0.south) -> (1.north);
		\draw (2.south) -> (1.north);
		\draw (2.south) -> (3.north);
		\draw (4.south) -> (3.north);
		
		\draw (0.south) -> (-1.north);
		\draw (-2.south) -> (-1.north);
		\draw (-2.south) -> (-3.north);
		
		\begin{scope}[every node/.style={fill=white, draw, circle, inner sep=1.4pt}]
			\node at (0.south) {};
			\node at (2.south) {};
			\node at (4.south) {};
			\node at (-2.south) {};
			\node at (-1.north) {};
			\node at (3.north) {};
		\end{scope}
		
		\begin{scope}[every node/.style={fill=black, draw, circle, inner sep=1.4pt}]
			\node at (-3.north) {};
			\node at (1.north) {};
		\end{scope}
		\end{tikzpicture}

		\vskip0.5cm

		\begin{tikzpicture}
		\clip (-1.6*\length, 1.2*\height) rectangle (1.6*\length, 0.2);
		
		\begin{scope}[every node/.style={inner sep=5pt}]
			\node (0) at (0,0) {$[a_1|bc_0]$};
			\node (1) at (\length, \height) {$[bc_0]$};
			\node (-1) at (-\length, \height) {$[a_1]$};
		\end{scope}
		
		\node (2) at (2*\length,0) {\phantom{$[\,]$}};
		\coordinate (1a) at ($(1)!0.5!(2)$) {};
		
		\node (-2) at (-2*\length, 0) {\phantom{$[\,]$}};
		\coordinate (-1a) at ($(-1)!0.5!(-2)$) {};
		
		\draw (0.south) -> (1.north);
		\draw[dashed] (1.north) -> (1a);
		
		\draw (0.south) -> (-1.north);
		\draw[dashed] (-1.north) -> (-1a);
		
		\begin{scope}[every node/.style={fill=white, draw, circle, inner sep=1.4pt}]
		\node at (0.south) {};
		\end{scope}
		
		\begin{scope}[every node/.style={fill=black, draw, circle, inner sep=1.4pt}]
		\node at (1.north) {};
		\node at (-1.north) {};
		\end{scope}
		\end{tikzpicture}\qquad
		\begin{tikzpicture}
		\clip (-1.6*\length, 1.2*\height) rectangle (1.6*\length, 0.2);

		\begin{scope}[every node/.style={inner sep=5pt}]
			\node (0) at (0,0) {$[a_1b|c_0]$};
			\node (1) at (\length, \height) {$[c_0]$};
			\node (-1) at (-\length, \height) {$[a_1b]$};
		\end{scope}
		
		\node (2) at (2*\length,0) {\phantom{$[\,]$}};
		\coordinate (1a) at ($(1)!0.5!(2)$) {};
		
		\node (-2) at (-2*\length, 0) {\phantom{$[\,]$}};
		\coordinate (-1a) at ($(-1)!0.5!(-2)$) {};
		
		\draw (0.south) -> (1.north);
		\draw[dashed] (1.north) -> (1a);
		
		\draw (0.south) -> (-1.north);
		\draw[dashed] (-1.north) -> (-1a);
		
		\begin{scope}[every node/.style={fill=white, draw, circle, inner sep=1.4pt}]
		\node at (0.south) {};
		\end{scope}
		
		\begin{scope}[every node/.style={fill=black, draw, circle, inner sep=1.4pt}]
		\node at (1.north) {};
		\node at (-1.north) {};
		\end{scope}
		\end{tikzpicture}\qquad
		\begin{tikzpicture}
		\clip (-3.6*\length, 1.2*\height) rectangle (1.6*\length, 0.2);
		
		\begin{scope}[every node/.style={inner sep=5pt}]
			\node (0) at (0,0) {$[a_1c_0|a_{-1}]$};
			\node (1) at (\length, \height) {$[a_{-1}]$};
			\node (-1) at (-\length, \height) {$[a_1c_0]$};
			\node (-2) at (-2*\length, 0) {$[c_2|a_1c_0]$};
			\node (-3) at (-3*\length, \height) {$[c_2]$};
		\end{scope}
		
		\node (2) at (2*\length,0) {\phantom{$[\,]$}};
		\coordinate (1a) at ($(1)!0.5!(2)$) {};

		\node (-4) at (-4*\length, 0) {\phantom{$[\,]$}};
		\coordinate (-3a) at ($(-3)!0.5!(-4)$) {};
		
		\draw (0.south) -> (1.north);
		\draw[dashed] (1.north) -> (1a);
		
		\draw (0.south) -> (-1.north);
		\draw (-2.south) -> (-1.north);
		\draw (-2.south) -> (-3.north);
		\draw[dashed] (-3.north) -> (-3a);
		
		\begin{scope}[every node/.style={fill=white, draw, circle, inner sep=1.4pt}]
		\node at (0.south) {};
		\node at (-2.south) {};
		\end{scope}
		
		\begin{scope}[every node/.style={fill=black, draw, circle, inner sep=1.4pt}]
		\node at (1.north) {};
		\node at (-1.north) {};
		\node at (-3.north) {};
		\end{scope}
		\end{tikzpicture}

		\vskip0.5cm
		
		\begin{tikzpicture}
		\clip (-1.6*\length, 1.2*\height) rectangle (1.6*\length, 0.2);
		
		\begin{scope}[every node/.style={inner sep=5pt}]
			\node (0) at (0,0) {$[a_1|b|c_0]$};
			\node (1) at (\length, \height) {$[b|c_0]$};
			\node (-1) at (-\length, \height) {$[a_1|b]$};
		\end{scope}
		
		\node (2) at (2*\length,0) {\phantom{$[\,]$}};
		\coordinate (1a) at ($(1)!0.5!(2)$) {};

		\node (-2) at (-2*\length, 0) {\phantom{$[\,]$}};
		\coordinate (-1a) at ($(-1)!0.5!(-2)$) {};
		
		\draw (0.south) -> (1.north);
		\draw[dashed] (1.north) -> (1a);
		
		\draw (0.south) -> (-1.north);
		\draw[dashed] (-1.north) -> (-1a);
		
		\begin{scope}[every node/.style={fill=white, draw, circle, inner sep=1.4pt}]
		\node at (0.south) {};
		\end{scope}
		
		\begin{scope}[every node/.style={fill=black, draw, circle, inner sep=1.4pt}]
		\node at (1.north) {};
		\node at (-1.north) {};
		\end{scope}
		\end{tikzpicture}\qquad
		\begin{tikzpicture}
		\clip (-3.6*\length, 1.2*\height) rectangle (1.6*\length, 0.2);
		
		\begin{scope}[every node/.style={inner sep=5pt}]
			\node (0) at (0,0) {$[a_1|c_0|a_{-1}]$};
			\node (1) at (\length, \height) {$[c_0 | a_{-1}]$};
			\node (-1) at (-\length, \height) {$[a_1|c_0]$};
			\node (-2) at (-2*\length, 0) {$[c_2|a_1|c_0]$};
			\node (-3) at (-3*\length, \height) {$[c_2|a_1]$};
		\end{scope}
		
		\node (2) at (2*\length,0) {\phantom{$[\,]$}};
		\coordinate (1a) at ($(1)!0.5!(2)$) {};
		
		\node (-4) at (-4*\length, 0) {\phantom{$[\,]$}};
		\coordinate (-3a) at ($(-3)!0.5!(-4)$) {};
		
		\draw (0.south) -> (1.north);
		\draw[dashed] (1.north) -> (1a);
		
		\draw (0.south) -> (-1.north);
		\draw (-2.south) -> (-1.north);
		\draw (-2.south) -> (-3.north);
		\draw[dashed] (-3.north) -> (-3a);
		
		\begin{scope}[every node/.style={fill=white, draw, circle, inner sep=1.4pt}]
		\node at (0.south) {};
		\node at (-2.south) {};
		\end{scope}
		
		\begin{scope}[every node/.style={fill=black, draw, circle, inner sep=1.4pt}]
		\node at (1.north) {};
		\node at (-1.north) {};
		\node at (-3.north) {};
		\end{scope}
		\end{tikzpicture}

		\vskip0.5cm
		
		\begin{tikzpicture}
		\clip (-3.6*\length, 1.2*\height) rectangle (1.6*\length, 0.2);
		
		\begin{scope}[every node/.style={inner sep=5pt}]
			\node (0) at (0,0) {$[c_2|c_0|b']$};
			\node (1) at (\length, \height) {$[c_0 | b']$};
			\node (-1) at (-\length, \height) {$[c_2|c_0]$};
			\node (-2) at (-2*\length, 0) {$[b|c_2|c_0]$};
			\node (-3) at (-3*\length, \height) {$[b|c_2]$};
		\end{scope}
		
		\node (2) at (2*\length,0) {\phantom{$[\,]$}};
		\coordinate (1a) at ($(1)!0.5!(2)$) {};
		
		\node (-4) at (-4*\length, 0) {\phantom{$[\,]$}};
		\coordinate (-3a) at ($(-3)!0.5!(-4)$) {};
		
		\draw (0.south) -> (1.north);
		\draw[dashed] (1.north) -> (1a);
		
		\draw (0.south) -> (-1.north);
		\draw (-2.south) -> (-1.north);
		\draw (-2.south) -> (-3.north);
		\draw[dashed] (-3.north) -> (-3a);
		
		\begin{scope}[every node/.style={fill=white, draw, circle, inner sep=1.4pt}]
		\node at (0.south) {};
		\node at (-1.north) {};
		\node at (-2.south) {};
		\end{scope}
		
		\begin{scope}[every node/.style={fill=black, draw, circle, inner sep=1.4pt}]
		\node at (1.north) {};
		\node at (-3.north) {};
		\end{scope}
		\end{tikzpicture}\qquad
		\begin{tikzpicture}
		\clip (-3.6*\length, 1.2*\height) rectangle (1.6*\length, 0.2);
		
		\begin{scope}[every node/.style={inner sep=5pt}]
			\node (0) at (0,0) {$[a_1|a_{-1}|b]$};
			\node (1) at (\length, \height) {$[a_{-1} | b]$};
			\node (-1) at (-\length, \height) {$[a_1|a_{-1}]$};
			\node (-2) at (-2*\length, 0) {$[b'|a_1|a_{-1}]$};
			\node (-3) at (-3*\length, \height) {$[b'|a_1]$};
		\end{scope}
		
		\node (2) at (2*\length,0) {\phantom{$[\,]$}};
		\coordinate (1a) at ($(1)!0.5!(2)$) {};
		
		\node (-4) at (-4*\length, 0) {\phantom{$[\,]$}};
		\coordinate (-3a) at ($(-3)!0.5!(-4)$) {};
		
		\draw (0.south) -> (1.north);
		\draw[dashed] (1.north) -> (1a);
		
		\draw (0.south) -> (-1.north);
		\draw (-2.south) -> (-1.north);
		\draw (-2.south) -> (-3.north);
		\draw[dashed] (-3.north) -> (-3a);
		
		\begin{scope}[every node/.style={fill=white, draw, circle, inner sep=1.4pt}]
		\node at (0.south) {};
		\node at (-1.north) {};
		\node at (-2.south) {};
		\end{scope}
		
		\begin{scope}[every node/.style={fill=black, draw, circle, inner sep=1.4pt}]
		\node at (1.north) {};
		\node at (-3.north) {};
		\end{scope}
		\end{tikzpicture}
	\end{center}
	\caption{Fiber components of $K_W = K_C$ in the case $\tilde A_2$.
		The first two fiber components form the subcomplex $K_D = K_F$.
		The black nodes correspond to the simplices of $X_W'$.}
	\label{fig:fiber-components-A2}
\end{figure}
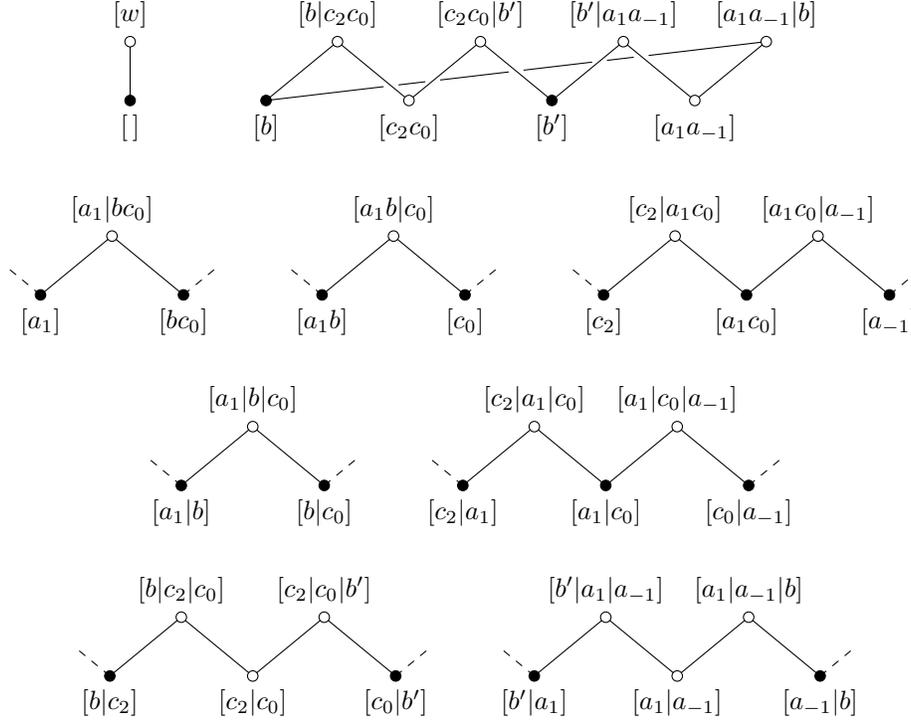

\begin{example}[Fiber components of $\tilde A_2$]
	For $W$ of type $\tilde A_2$, the fiber components of $K_W = K_C$ are shown in \Cref{fig:fiber-components-A2}, using the notation of \Cref{ex:A2}.
	The first 2 fiber components are finite, and they form the subcomplex $K_D = K_F$.
	The other 7 fiber components are infinite.
	Black nodes correspond to simplices in $\F(X_W')$, and white nodes correspond to simplices in $\F(K_W') \setminus \F(X_W')$, where $K_W'$ is the canonical subcomplex of $K_W$.
	The shown simplices are exactly those in $\F(K_W')$.
	\label{ex:fiber-components-A2}
\end{example}

\section{The \texorpdfstring{$K(\pi,1)$}{K(pi,1)} conjecture}
\label{sec:conjecture}

In this section, we prove the $K(\pi,1)$ conjecture for affine Artin groups.
It is enough to consider the irreducible case.

Let $W$ be an irreducible affine Coxeter group, and $R$ its set of reflections.
Fix a Coxeter element $w$ and an axial chamber $C_0$ of the Coxeter complex.
Let $S \subseteq R$ be the set of simple reflections associated with $C_0$.
By \Cref{thm:coxeter-A} (for the case $\tilde A_n$) and \Cref{thm:bipartite-coxeter} (for the other cases), the Coxeter element $w$ can be written as the product of the elements of $S$, say $w = s_1s_2\dotsm s_{n+1}$ (where $n$ is the rank of $W$).
Let $\{p_i\}_{i \in \Z}$ be the sequence of points of the Coxeter axis $\ell$ that are fixed by at least one vertical reflection of $W$ (see \Cref{sec:bipartite-coxeter} and \Cref{lemma:axis-A}).
Enumerate these points so that $p_0$ is below $C_0$ and $p_1$ is above $C_0$.

Let $K_W$ be the interval complex associated with the noncrossing partition poset $[1,w] = [1,w]^W$, $K_W' \subseteq K_W$ its canonical nice subcomplex (introduced at the end of \Cref{sec:finite-classifying-spaces}), and $X_W' \subseteq K_W'$ the complex introduced in \Cref{def:dual-salvetti}.
By \Cref{lemma:affine-dual-salvetti-complex}, $X_W'$ consists of the simplices $[x_1|x_2|\dotsb|x_d]$ of $K_W$ such that $x_1x_2\dotsm x_d$ fixes a vertex of $C_0$.

Recall that $K_W$ is a classifying space for the dual Artin group $W_w$ (\Cref{thm:KW-classifying}), $K_W$ deformation retracts onto $K_W'$ (\Cref{thm:nice-subcomplex}), and $X_W'$ is homotopy equivalent to the orbit configuration space $Y_W$ (\Cref{thm:dual-salvetti}).
Then, in order to prove the $K(\pi,1)$ conjecture for the Artin group $G_W = \pi_1(Y_W)$, it is enough to show that $K_W'$ deformation retracts onto $X_W'$.
This also implies that the dual Artin group $W_w$ is isomorphic to the Artin group $G_W$, thus giving a new proof of \cite[Theorem C]{mccammond2017artin}.

To show that $K_W'$ deformation retracts onto $X_W'$, we are going to use discrete Morse theory.
Specifically, we need to construct an acyclic matching on the face poset $\F(K_W')$ such that the set of critical simplices is exactly $\F(X_W')$.

Denote by $\varphi \colon [1,w] \to [1,w]$ the conjugation by $w$: $\varphi(u) = w^{-1}uw$.
Let $\prec$ be an axial ordering of the set of reflections $R_0 = R \cap [1,w]$ (see \Cref{def:axial-ordering}) that satisfies the following \emph{compatibility property}: if $r,r' \in R_0$ fix the same point of $\ell$, and $r \prec r'$, then $\varphi(r) \prec \varphi(r')$.
Although not strictly necessary, this compatibility property will make the proof of \Cref{lemma:matching2} simpler.

By \Cref{thm:shellability}, every element $u \in [1,w]$ has a unique minimal length factorization $u = r_1 r_2 \dotsm r_m$ as a product of reflections such that $r_1 \prec r_2 \prec \dotsb \prec r_m$.
We call it the \emph{increasing factorization} of $u$.
\Cref{thm:shellability} also implies that $r_1$ is the $\prec$-smallest reflection of $R_0 \cap [1,u]$, and $r_m$ is the $\prec$-largest.
As in \Cref{sec:shellability}, we say that a vertical reflection is \emph{positive} if it fixes a point of $\ell$ above $C_0$, and \emph{negative} otherwise.

Since $w$ acts on the Coxeter axis $\ell$ as a translation in the positive direction, we have that $\Fix(\varphi(r)) \cap \ell$ is below $\Fix(r) \cap \ell$ for every vertical reflection $r \in R_0$.
In addition, if $\Fix(r) \cap \ell$ is above $\Fix(r') \cap \ell$ (for some vertical reflections $r,r'$), then $\Fix(\varphi^j(r)) \cap \ell$ is above $\Fix(\varphi^j(r')) \cap \ell$ for all $j \in \Z$.

\begin{lemma}
	For every vertical reflection $r \in R_0$, there exists a unique $j \in \Z$ such that $\varphi^j(r)$ is one of the $n+1$ $\prec$-smallest reflections in $R_0$.
	In addition, $j \geq 0$ if and only if $r$ is positive.
	\label{lemma:vertical-reflection-conjugate}
\end{lemma}

\begin{proof}
	If $w$ is a bipartite Coxeter element (see \Cref{sec:bipartite-coxeter}), the $n+1$ $\prec$-smallest reflections are the ones that fix $p_1$ or $p_2$.
	Let $p_i$ be the point of $\ell$ which is fixed by $r$.
	By \Cref{rmk:axial-orbits-bipartite}, there is a unique $j$ such that $w^{-j}(p_i)$ is $p_1$ or $p_2$, and specifically $j = \left\lfloor\frac{i-1}{2}\right\rfloor$.
	We have $j \geq 0$ if and only if $r$ is positive.
	Then $\varphi^{j}(r)$ fixes $p_1$ or $p_2$, so it is among the $n+1$ $\prec$-smallest reflections of $R_0$.
	
	Suppose now that $w$ is a $(p,q)$-bigon Coxeter element (with $p+q=n+1$) in a Coxeter group $W$ of type $\tilde A_n$.
	By part (i) of \Cref{prop:axial-point-A}, the $n+1$ $\prec$-smallest reflections are those that fix one of $p_1, p_2, \dotsc, p_{m}$ with $m = \frac{p+q}{\gcd(p,q)}$.
	By \Cref{lemma:axis-A}, for every $i \in \Z$ there is a unique $j \in \Z$ such that $w^{-j}(p_i) \in \{p_1, p_2, \dotsc, p_m\}$.
	We conclude as in the bipartite case.
\end{proof}

\begin{lemma}
	Let $r \in R_0$ be a vertical reflection.
	The right complement of $r$ fixes a vertex of $C_0$ if and only if $r$ is among the $n+1$ $\prec$-smallest reflections of $R_0$.
	\label{lemma:smallest-reflections}
\end{lemma}

\begin{proof}
	By \Cref{lemma:elliptic-coxeter-element}, for each vertex $b$ of $C_0$ there is a unique vertical elliptic isometry $u \in [1,w]$ such that $l(u) = n$ and $u$ fixes $b$.
	By taking the left complement of these vertical elliptic isometries, we obtain that there are exactly $n+1$ vertical reflections $r$ such that the right complement of $r$ fixes a vertex of $C_0$.
	We only need to show that they are the $n+1$ $\prec$-smallest reflections of $R_0$.
	
	If $w$ is a bipartite Coxeter element, then the $n+1$ $\prec$-smallest reflections are those that fix $p_1$ or $p_2$.
	By \Cref{lemma:elliptic-coxeter-bipartite}, these are precisely the ones that have a right complement that fixes a vertex of $C_0$.
	
	Suppose now that $w$ is a $(p,q)$-bigon Coxeter element in a Coxeter group $W$ of type $\tilde A_n$.
	By part (i) of \Cref{prop:axial-point-A}, the $n+1$ $\prec$-smallest reflections of $R_0$ are those that fix one of $p_1, p_2, \dotsc, p_{m}$ with $m = \frac{p+q}{\gcd(p,q)}$.
	Let $r \in R_0$ be a vertical reflection such that its right complement fixes a vertex $b$ of $C_0$.
	Let $C$ be the axial chamber immediately below $\Fix(r) \cap \ell$.
	By \Cref{lemma:elliptic-coxeter-A}, $r$ is positive (otherwise its right complement would not fix a vertex of $C_0$) and $\b$ is a vertex of $C$.
	Suppose by contradiction that $r$ is not among the $n+1$ $\prec$-smallest reflections of $R_0$.
	Then $\Fix(r)$ intersects $\ell$ in a point $p_i$ with $i \geq m+1$, and there are at least $m+1$ axial chambers between $C_0$ and $C$ (including $C_0$ and $C$ themselves).
	By part (ii) of \Cref{prop:axial-orbits-A}, the axial point $\b$ is a vertex of exactly $m$ consecutive axial chambers.
	This is a contradiction because $\b$ is a vertex of both $C_0$ and $C$.
\end{proof}

It is convenient to introduce some additional notation.
Given a simplex $\sigma = [x_1|x_2|\dotsb|x_d] \in K_W$, let $\pi(\sigma) = x_1x_2\dotsm x_d \in [1,w]$.
Also, if $\CC$ is the $d$-fiber component containing $\sigma$, let $\lambda(\sigma)$ (resp.\ $\rho(\sigma)$) be the simplex that appears immediately to the left (resp.\ immediately to the right) of $\sigma$ in $\CC$.
More explicitly:
\begin{IEEEeqnarray*}{rCl}
	\lambda(\sigma) &=&
	\begin{cases}
		[x_1|\dotsb|x_{d-1}] & \text{if $\pi(\sigma) = w$} \\
		[y|x_1|\dotsb | x_{d}] & \text{otherwise (here $y$ is the left complement of $x_1\dotsm x_d$)}
	\end{cases} \\[0.1cm]
	\rho(\sigma) &=&
	\begin{cases}
		[x_2|\dotsb|x_{d}] & \text{if $\pi(\sigma) = w$} \\
		[x_1|\dotsb | x_{d}|y] & \text{otherwise (here $y$ is the right complement of $x_1\dotsm x_d$).}
	\end{cases}
\end{IEEEeqnarray*}
Then $\lambda$ is the inverse of $\rho$.
More generally, we say that a simplex $\tau \in K_W$ is \emph{to the left of $\sigma$} (resp.\ \emph{to the right of $\sigma$}) if $\tau = \lambda^k(\sigma)$ (resp.\ $\tau= \rho^k(\sigma)$) for some $k > 0$.
Notice that, if $\sigma$ and $\tau$ belong to a finite component $\CC$, then $\tau$ is both to the left of $\sigma$ and to the right of $\sigma$.

With this notation, the definitions of $X_W'$ and $K_W'$ can be written as follows:
\begin{itemize}
	\item $\sigma \in \F(X_W')$ if and only if $\pi(\sigma)$ fixes a vertex of $C_0$;
	\item $\sigma \in \F(K_W')$ if and only if $\lambda^k(\sigma) \in \F(X_W')$ for some $k \geq 0$ and $\rho^k(\sigma) \in \F(X_W')$ for some $k \geq 0$.
\end{itemize}

The following definition will be used in the construction of the matching on $\F(K_W')$.

\begin{definition}[Depth]
	Let $\sigma = [x_1|x_2|\dotsb|x_d] \in \F(K_W)$, with $\pi(\sigma) = w$.
	Define the \emph{depth} $\delta(\sigma)$ of $\sigma$ as the minimum $i \in \{1,2,\dotsc, d\}$ such that one of the following occurs:
	\begin{enumerate}[(i)]
		\item $l(x_i) \geq 2$;
		\item $l(x_i) = 1$, $i \leq d-1$, and $x_i \prec r$ for every reflection $r\leq x_{i+1}$ in $[1,w]$.
	\end{enumerate}
	If no such $i$ exists, let $\delta(\sigma) = \infty$.
	\label{def:depth}
\end{definition}

\begin{lemma}
	Let $\sigma \in \F(K_W')$.
	If $\pi(\sigma) = w$, and $\pi(\rho(\sigma))$ fixes a vertex of $C_0$, then $\delta(\sigma) \neq \infty$.
	\label{lemma:well-defined-depth}
\end{lemma}

\begin{proof}
	Let $\sigma = [x_1|x_2|\dotsb|x_d]$, and suppose by contradiction that $\delta(\sigma) = \infty$.
	Then each $x_i$ is a reflection, $d = l(w) = n+1$, and $x_1 \succ x_2 \succ \dotsb \succ x_{n+1}$.
	Since $\pi(\rho(\sigma)) = x_2\dotsm x_{n+1}$ fixes a vertex of $C_0$, by \Cref{lemma:smallest-reflections} we have that $x_1$ is among the $n+1$ $\prec$-smallest reflections of $R_0$.
	Then $x_1, x_2, \dotsc, x_{n+1}$ are the $n+1$ $\prec$-smallest reflections of $R_0$.
	
	Let $[y_1|y_2|\dotsb|y_{n+1}]$ be a $(n+1)$-simplex to the left of $\sigma$.
	Then $y_1 = \varphi^j(x_i)$ for some $i \in \{1, \dotsc, n+1 \}$ and $j < 0$.
	The right complement $y_2\dotsm y_{n+1}$ of $y_1$ is equal to $\varphi^j(u)$, where $u$ is the right complement of $x_i$.
	Since $x_i$ is among the $n+1$ $\prec$-smallest reflections, its right complement $u$ fixes a vertex $b$ of $C_0$ by \Cref{lemma:smallest-reflections}.
	Then $y_2\dotsm y_{n+1} = \varphi^j(u)$ fixes the axial vertex $w^{-j}(b)$.
	By \Cref{prop:axial-orbits-A} (for the case $\tilde A_n$) and \Cref{rmk:axial-orbits-bipartite} (for the other cases), $w^{-j}(b)$ is not a vertex of $C_0$.
	Therefore $[y_2|\dotsb|y_{n+1}] \not\in \F(X_W')$.
	This conclusion applies to every $n$-simplex $[y_2|\dotsb|y_{n+1}]$ to the left of $\sigma$.
	Then $\sigma \not\in \F(K_W')$, which is a contradiction.
\end{proof}

We are now ready to define a function $\mu$ which will ultimately give us the matching we need.

\begin{definition}[Matching function]
	Given a simplex $\sigma \in \F(K_W') \setminus \F(X_W')$, define a simplex $\mu(\sigma) \in \F(K_W)$ as follows.
	\begin{enumerate}[(1)]
		\item If $\pi(\sigma) \neq w$, let $\mu(\sigma) = \lambda(\sigma)$.
		
		\item If $\pi(\sigma) = w$, and $\pi(\rho(\sigma))$ does not fix a vertex of $C_0$, let $\mu(\sigma) = \rho(\sigma)$.
	\end{enumerate}
	Suppose now that $\pi(\sigma) = w$, and $\pi(\rho(\sigma))$ fixes a vertex of $C_0$. Let $\sigma = [x_1|x_2|\dotsb|x_d]$ and $\delta = \delta(\sigma)$. Notice that $\delta \neq \infty$ by \Cref{lemma:well-defined-depth}.
	\begin{enumerate}[(1)]
		\setcounter{enumi}{2}
		\item If $l(x_\delta) \geq 2$, define $\mu(\sigma) = [x_1|\dotsb|x_{\delta-1}|y|z|x_{\delta+1}|\dotsb|x_d]$, where $y$ is the $\prec$-smallest reflection of $R_0 \cap [1,x_\delta]$, and $yz = x_\delta$.
		
		\item If $l(x_\delta) = 1$, define $\mu(\sigma) = [x_1|\dotsb|x_{\delta-1}|x_\delta x_{\delta+1}|x_{\delta+2}|\dotsb | x_d]$.
	\end{enumerate}
	Notice that $\mu(\sigma) \gtrdot \sigma$ if $\sigma$ occurs in case (1) or (3), whereas $\mu(\sigma) \lessdot \sigma$ if $\sigma$ occurs in case (2) or (4).
	In addition, $\mu(\sigma) \not \in \F(X_W')$ by \Cref{lemma:affine-dual-salvetti-complex}.
	\label{def:matching}
\end{definition}

\begin{figure}
	{\def\arraystretch{1.2}
	\begin{tabular}{rcl}
		$[c_2c_0]$ & $\longleftrightarrow$ & $[b|c_2 c_0]$ \\
		$[a_1a_{-1}]$ & $\longleftrightarrow$ & $[b'|a_1a_{-1}]$ \\
		$[c_2 | c_0]$ & $\longleftrightarrow$ & $[b | c_2 | c_0]$ \\
		$[a_1 | a_{-1}]$ & $\longleftrightarrow$ & $[b' | a_1 | a_{-1}]$ \\
		$[w]$ & $\longleftrightarrow$ & $[a_1 | b c_0]$
	\end{tabular}\qquad
	\begin{tabular}{rcl}
		$[c_2c_0 | b']$ & $\longleftrightarrow$ & $[c_2 | c_0 | b']$ \\
		$[a_1 a_{-1} | b]$ & $\longleftrightarrow$ & $[a_1 | a_{-1} | b]$ \\
		$[c_2 | a_1 c_0]$ & $\longleftrightarrow$ & $[c_2 | a_1 | c_0]$ \\
		$[a_1 c_0 | a_{-1}]$ & $\longleftrightarrow$ & $[a_1 | c_0 | a_{-1}]$ \\
		$[a_1 b | c_0]$ & $\longleftrightarrow$ & $[a_1 | b | c_0]$
	\end{tabular}
	}
	\caption{Matching on $\F(K_W') \setminus \F(X_W')$ in the case $\tilde A_2$.}
	\label{fig:matching-A2}
\end{figure}

\begin{example}[Matching for $\tilde A_2$]
	\Cref{fig:matching-A2} shows the matching defined by $\mu$ on $\F(K_W') \setminus \F(X_W')$ in the case $\tilde A_2$, using the axial ordering of \Cref{ex:axial-orderings}.
	See also \Cref{fig:fiber-components-A2}, where the involved simplices are indicated by white nodes.
	The first 4 pairs on the left column occur in cases (1) and (2) of \Cref{def:matching}.
	The other 6 pairs occur in cases (3) and (4).
\end{example}

\begin{lemma}
	Let $\sigma \in \F(K_W') \setminus \F(X_W')$ be a simplex such that $\pi(\sigma) = w$, and $\rho(\sigma) \in \F(X_W')$.
	Then also $\rho(\mu(\sigma)) \in \F(X_W')$.
	\label{lemma:matching1}
\end{lemma}

\begin{proof}
	Let $\sigma = [x_1|x_2|\dotsb|x_d]$.
	It is not possible for $\sigma$ to occur in case (1) or (2) of \Cref{def:matching}, because $\pi(\sigma) = w$ and $\rho(\sigma) \in \F(X_W')$.
	
	If $\sigma$ occurs in case (4), then we have $\mu(\sigma) = [x_1|\dotsb|x_{\delta-1}|x_\delta x_{\delta+1}|x_{\delta+2}|\dotsb | x_d]$.
	Therefore $\pi(\rho(\mu(\sigma))) \leq \pi(\rho(\sigma))$ in $[1,w]$.
	Since $\pi(\rho(\sigma))$ fixes a vertex of $C_0$, also $\pi(\rho(\mu(\sigma)))$ fixes a vertex of $C_0$ by \Cref{lemma:model-poset}.
	Then $\rho(\mu(\sigma)) \in \F(X_W')$.
	
	If $\sigma$ occurs in case (3), then $\mu(\sigma) = [x_1|\dotsb|x_{\delta-1}|y|z|x_{\delta+1}|\dotsb|x_d]$.
	If $\delta > 1$, we have $\pi(\rho(\mu(\sigma))) = \pi(\rho(\sigma))$, and so $\pi(\rho(\mu(\sigma)))$ fixes a vertex of $C_0$.
	Assume from now on that $\delta = 1$, so $\mu(\sigma) = [y|z|x_{2}|\dotsb|x_d]$.
	Let $\b$ be a vertex of $C_0$ fixed by $\pi(\rho(\sigma)) = x_2\dotsm x_d$.
	Let $u \in [1,w]$ be the unique vertical elliptic element that fixes $\b$ with $l(u)=n$ (see \Cref{lemma:elliptic-coxeter-element}).
	By \Cref{lemma:elliptic-subposet}, we have $x_2\dotsm x_d \leq u$ in $[1,w]$.
	Let $r$ be the left complement of $u$ (it is a vertical reflection).
	By \Cref{lemma:smallest-reflections}, $r$ is one of the $n+1$ $\prec$-smallest reflections of $R_0$.
	Passing to the left complements in the inequality $x_2\dotsm x_d \leq u$, we get $x_1 \geq r$ in $[1,w]$.
	By definition of $\mu$, we have that $y$ is the $\prec$-smallest reflection of $R_0 \cap [1,x_1]$, and therefore $y \preceq r$.
	In particular, $y$ is one of the $n+1$ $\prec$-smallest reflections of $R_0$.
	By \Cref{lemma:smallest-reflections}, its right complement $zx_2\dotsm x_d = \pi(\rho(\mu(\sigma)))$ fixes a vertex of $C_0$.
\end{proof}

\begin{lemma}
	For every $\sigma \in \F(K_W') \setminus \F(X_W')$, we have that $\mu(\sigma) \in \F(K_W') \setminus \F(X_W')$.
	\label{lemma:matching2}
\end{lemma}

\begin{proof}
	We have already noted in \Cref{def:matching} that $\mu(\sigma) \not \in \F(X_W')$.
	We have $\mu(\sigma) = \lambda(\sigma)$ in case (1), and $\mu(\sigma) = \rho(\sigma)$ in case (2).
	Since $\sigma \in \F(K_W') \setminus \F(X_W')$, in these two cases $\mu(\sigma) \in \F(K_W')$.
	Suppose from now on that $\sigma$ occurs in case (3) or case (4).
	In particular, $\pi(\sigma) = w$ and $\rho(\sigma) \in \F(X_W')$.
	
	Let $\sigma = [x_1|x_2|\dotsb|x_d]$.
	By \Cref{lemma:matching1}, we have $\rho(\mu(\sigma)) \in \F(X_W')$.
	Therefore we only need to prove that there is a simplex in $\F(X_W')$ to the left of $\mu(\sigma)$.
	Since $\sigma \in \F(K_W')$, there is a $(d-1)$-simplex $\tau \in \F(X_W')$ to the left of $\sigma$.
	It has the following form:
	\[ \tau = [\varphi^{h}(x_{i+1}) | \dotsb | \varphi^{h}(x_d) | \varphi^{h+1}(x_1) | \dotsb | \varphi^{h+1}(x_{i-1})], \]
	for some $h < 0$ and $i \in \{1,\dotsc, d\}$.
	
	If $\sigma$ occurs in case (4), then $\mu(\sigma) = [x_1|\dotsb|x_{\delta-1}|x_\delta x_{\delta+1}|x_{\delta+2}|\dotsb | x_d]$ where $\delta = \delta(\sigma)$.
	To the left of $\mu(\sigma)$ there is a simplex $\tau'$ such that $\pi(\tau') \leq \pi(\tau)$, given by:
	\[
		\begin{rightcases}
			[\varphi^{h}(x_{\delta+2}) | \dotsb | \varphi^{h}(x_d) | \varphi^{h+1}(x_1) | \dotsb | \varphi^{h+1}(x_{\delta-1})] & \hspace{-1.5cm} \text{if $i=\delta$ or $i = \delta+1$\phantom{.}} \\
			[\varphi^h(x_{i+1})|\dotsb | \varphi^h(x_\delta x_{\delta+1}) | \dotsb | \varphi^h(x_d)| \varphi^{h+1}(x_1) | \dotsb | \varphi^{h+1}(x_{i-1})] & \text{if $i < \delta\phantom{.}$}
			\\
			[\varphi^h(x_{i+1})|\dotsb | \varphi^h(x_d)| \varphi^{h+1}(x_1) | \dotsb | \varphi^{h+1}(x_\delta x_{\delta+1}) | \dotsb | \varphi^{h+1}(x_{i-1})] &  \text{if $i > \delta + 1$}.
		\end{rightcases}
	\]
	Since $\tau \in \F(X_W')$, we have that $\pi(\tau)$ fixes a vertex of $C_0$, so $\pi(\tau')$ also fixes a vertex of $C_0$ (by \Cref{lemma:model-poset}), which means that $\tau' \in \F(X_W')$.
	Therefore $\mu(\sigma) \in \F(K_W')$.
	
	If $\sigma$ occurs in case (3), then $\mu(\sigma) = [x_1|\dotsb|x_{\delta-1}|y|z|x_{\delta+1}|\dotsb|x_d]$.
	If $i \neq \delta$, we can find a simplex $\tau'$ to the left of $\mu(\sigma)$ such that $\pi(\tau') = \pi(\tau)$, namely:
	\[
		\begin{cases}
			[\varphi^{h}(x_{i+1}) | \dotsb | \varphi^h(y) | \varphi^h(z) | \dotsb | \varphi^{h}(x_d) | \varphi^{h+1}(x_1) | \dotsb | \varphi^{h+1}(x_{i-1})] & \text{if $i < \delta$} \\
			
			[\varphi^{h}(x_{i+1}) | \dotsb | \varphi^{h}(x_d) | \varphi^{h+1}(x_1) | \dotsb | \varphi^{h+1}(y) | \varphi^{h+1}(z) | \dotsb | \varphi^{h+1}(x_{i-1})] & \text{if $i > \delta$}.
		\end{cases}
	\]
	As before, this implies that $\mu(\sigma) \in \F(K_W')$.
	Suppose from now on that $i = \delta$.
	
	The right complement of $\varphi^h(x_\delta)$ is equal to $\pi(\tau)$, which is elliptic because it fixes a vertex of $C_0$. Therefore $\varphi^h(x_\delta)$ is vertical, and thus also $x_\delta = yz$ is vertical.
	Let $z = r_1r_2\dotsm r_m$ be the increasing factorization of $z$, with $m \geq 1$.
	By definition of $y$, we have that $y \prec r_1 \prec \dotsb \prec r_m$.
	Since $yz$ is vertical, at least one of $y, r_1, \dotsc, r_m$ is a vertical reflection.
	\begin{itemize}
		\item Case 1: $r_m$ is a negative vertical reflection.
		By \Cref{lemma:vertical-reflection-conjugate}, there exists a $j < 0$ such that $\varphi^j(r_m)$ is among the $n+1$ $\prec$-smallest reflections of $R_0$.
		By \Cref{lemma:smallest-reflections}, its right complement $u$ fixes a vertex of $C_0$.
		Consider the following simplex to the left of $\mu(\sigma)$:
		\[ \tau' = [\varphi^j(x_{\delta+1}) | \dotsb | \varphi^j(x_d) | \varphi^{j+1}(x_1) | \dotsb | \varphi^{j+1}(x_{\delta-1}) | \varphi^{j+1}(y)]. \]
		The left complement of $\pi(\tau')$ is $\varphi^j(z)$, and we have $\varphi^j(r_m) \leq \varphi^j(z)$ in $[1,w]$.
		Passing to the right complements in this inequality, we obtain that $u \geq \pi(\tau')$.
		Since $u$ fixes a vertex of $C_0$, by \Cref{lemma:model-poset} also $\pi(\tau')$ fixes a vertex of $C_0$.
		Therefore $\tau' \in \F(X_W')$.
		
		\item Case 2: $r_m$ is a horizontal reflection or a positive vertical reflection.
		The same is true for all the reflections $y, r_1, \dotsc, r_m$, because $y \prec r_1 \prec \dotsb \prec r_m$.
		Recall that at least one of them is vertical.
		Then, for some $k \in \{0, \dotsc, m\}$, we have that $y, r_1, \dotsc, r_k$ are positive vertical reflections, and $r_{k+1}, \dotsc, r_m$ are horizontal reflections.
		Since $h < 0$, we have that $\varphi^h(y), \varphi^h(r_1), \dotsc, \varphi^h(r_k)$ are also positive vertical reflections.
		The compatibility property of $\prec$ then implies $\varphi^h(y) \prec \varphi^h(r_1) \prec \dotsb \prec \varphi^h(r_k)$.
		Let $r_{k+1}'\dotsm r_m'$ be the increasing factorization of $\varphi^h(r_{k+1}\dotsm r_m)$.
		Since $r_{k+1}\dotsm r_m$ is horizontal, $\varphi^h(r_{k+1}\dotsm r_m)$ is also horizontal, so the reflections $r_{k+1}', \dotsc, r_m'$ are horizontal.
		Therefore we have $\varphi^h(y) \prec \varphi^h(r_1) \prec \dotsb \prec \varphi^h(r_k) \prec r_{k+1}' \prec \dotsb \prec r_m'$.
		By construction, the product of these $m+1$ reflections yields the increasing factorization of $\varphi^h(yz) = \varphi^h(x_\delta)$.
		In particular, $\varphi^h(y)$ is the $\prec$-smallest reflection of $[1, \varphi^h(x_\delta)]$ by \Cref{thm:shellability}.
		
		\noindent Recall that $\tau \in \F(X_W')$, so $\pi(\tau)$ fixes a vertex $\b$ of $C_0$.
		Let $u \in [1,w]$ be a vertical elliptic element that fixes $\b$ and such that $l(u) = n$ (see \Cref{lemma:elliptic-coxeter-element}).
		By \Cref{lemma:elliptic-subposet}, we have $\pi(\tau) \leq u$ in $[1,w]$.
		If we pass to the left complements in this inequality, we get $\varphi^h(x_\delta) \geq r$ where $r$ is the left complement of $u$ (it is a vertical reflection).
		By \Cref{lemma:smallest-reflections}, $r$ is among the $n+1$ $\prec$-smallest reflections of $R_0$.
		Since $\varphi^h(y)$ is the $\prec$-smallest reflection of $[1,\varphi^h(x_\delta)]$, we have $\varphi^h(y) \preceq r$, and thus $\varphi^h(y)$ is also among the $n+1$ $\prec$-smallest reflections of $R_0$.
		By \Cref{lemma:smallest-reflections}, the right complement of $\varphi^h(y)$ fixes a vertex of $C_0$.
		Consider the following simplex to the left of $\mu(\sigma)$:
		\[ \tau' = [\varphi^h(z)|\varphi^h(x_{\delta+1})| \dotsb |\varphi^h(x_d)|\varphi^{h+1}(x_1)|\dotsb |\varphi^{h+1}(x_{\delta-1})]. \]
		We have that $\pi(\tau')$ is the right complement of $\varphi^h(y)$, so it fixes a vertex of $C_0$, and therefore $\tau' \in \F(X_W')$. \qedhere
	\end{itemize}
\end{proof}

\begin{proposition}
	The function $\mu\colon \F(K_W') \setminus \F(X_W') \to \F(K_W')\setminus \F(X_W')$ is an involution, i.e.\ it satisfies $\mu(\mu(\sigma)) = \sigma$.
	In addition, if $\sigma$ occurs in case (3) or (4) of \Cref{def:matching}, then $\delta(\mu(\sigma)) = \delta(\sigma)$.
	\label{prop:involution}
\end{proposition}

\begin{proof}
	\Cref{lemma:matching2} shows that the image of $\mu$ is contained in $\F(K_W') \setminus \F(X_W')$, so we can compose $\mu$ with itself.
	Let $\sigma = [x_1|x_2|\dotsb|x_d] \in \F(K_W') \setminus \F(X_W')$.
	
	If $\sigma$ occurs in case (1) of \Cref{def:matching}, then $\mu(\sigma) = \lambda(\sigma)$ occurs in case (2), so $\mu(\mu(\sigma)) = \rho(\lambda(\sigma)) = \sigma$.
	Similarly, if $\sigma$ occurs in case (2), then $\mu(\sigma) = \rho(\sigma)$ occurs in case (1), and $\mu(\mu(\sigma)) = \lambda(\rho(\sigma)) = \sigma$.
	
	If $\sigma$ occurs in case (3) or (4), by definition of $\delta = \delta(\sigma)$ we have that $x_1, \dotsc, x_{\delta-1}$ are reflections such that $x_1 \succ x_2 \succ \dotsb \succ x_{\delta-1}$.
	In addition, $x_{\delta-1} \succ y$ if $y$ is the $\prec$-smallest reflection of $[1,x_\delta]$.
	If $\sigma$ occurs in case (3), then $x_1 \succ \dotsb \succ x_{\delta-1} \succ y$, and $y \prec r$ for every reflection $r \leq z$.
	Therefore $\delta(\mu(\sigma)) = \delta$, and $\mu(\sigma)$ occurs in case (4), so $\mu(\mu(\sigma)) = \sigma$.
	If $\sigma$ occurs in case (4), then $\delta(\mu(\sigma)) = \delta$ because $x_{\delta-1} \succ x_\delta$. In particular, $\mu(\sigma)$ occurs in case (3).
	By definition of $\delta$, we also have that $x_\delta \prec r'$ for every reflection $r' \leq x_{\delta+1}$.
	Then, if we concatenate $x_\delta$ with the increasing factorization of $x_{\delta+1}$, we get the increasing factorization of $x_\delta x_{\delta+1}$. Therefore $x_\delta$ is the $\prec$-smallest reflection of $[1,x_\delta x_{\delta+1}]$, and $\mu(\mu(\sigma)) = \sigma$. 
\end{proof}

Thanks to \Cref{prop:involution}, we can finally define a matching $\M$ on $\F(K_W')$:
\[ \M = \{ (\mu(\sigma), \sigma) \mid \sigma \in \F(K_W') \setminus \F(X_W') \text{ and } \mu(\sigma) \lessdot \sigma \}. \]
A simplex $\sigma \in \F(K_W')$ is critical if and only if $\sigma \in \F(X_W')$.
It only remains to prove that $\M$ is acyclic and proper.
For this, our strategy is to define a poset $(P, \trianglelefteq)$ and
a map $\xi \colon \F(K_W') \setminus \F(X_W') \to P$ that decreases along alternating paths.

Let $P \subseteq R_0^{n+1}$ be the set of all minimal length factorizations of $w$ as a product of reflections.
We endow $P$ with the total ordering $\trianglelefteq$ defined as follows.
Let $\alpha, \alpha' \in P$, and denote by $r$ (resp.\ $r'$) the $\prec$-largest reflection appearing in $\alpha$ (resp.\ $\alpha'$).
\begin{itemize}
	\item If $r \neq r'$, then set $\alpha \vartriangleleft \alpha'$ if and only if $r \succ r'$.
	
	\item If $r = r'$, let $k$ (resp.\ $k'$) be the position where $r$ ($= r'$) appears in $\alpha$ (resp.\ $\alpha'$).
	If $k \neq k'$, then set $\alpha \vartriangleleft \alpha'$ if and only if $k > k'$.
	
	\item If $r = r'$ and $k = k'$, then set $\alpha \vartriangleleft \alpha'$ if and only if $\alpha$ is lexicographically smaller than $\alpha'$ (as usual, reflections are compared using the total ordering $\prec$ of $R_0$).
\end{itemize}
The transitive property of $\trianglelefteq$ is immediate to check.

Define a closure operator $\kappa \colon \F(K_W') \setminus \F(X_W') \to \F(K_W') \setminus \F(X_W')$ in the following way:
\[
\kappa(\sigma) =
\begin{cases}
\sigma & \text{if $\pi(\sigma) = w$} \\
\lambda(\sigma) = \mu(\sigma) & \text{otherwise}.
\end{cases}
\]
Given a simplex $\sigma \in \F(K_W') \setminus \F(X_W')$, define $\xi(\sigma) \in P$ as the concatenation of the increasing factorizations of $x_1, x_2, \dotsc, x_d$, where $[x_1|x_2|\dotsb|x_d] = \kappa(\sigma)$.
Since $\pi(\kappa(\sigma)) = w$, we have that $\xi(\sigma)$ is indeed a minimal length factorization of $w$.

\begin{lemma}
	For every $\sigma \in \F(K_W') \setminus \F(X_W')$, we have $\xi(\mu(\sigma)) = \xi(\sigma)$.
	\label{lemma:matching-compatibility}
\end{lemma}

\begin{proof}
	Suppose without loss of generality that $\mu(\sigma) \lessdot \sigma$, so that $\sigma$ occurs in case (2) or (4) of \Cref{def:matching}.
	If $\sigma$ occurs in case (2), then $\kappa(\mu(\sigma)) = \kappa(\sigma)$, and therefore $\xi(\mu(\sigma)) = \xi(\sigma)$.

	Suppose now that $\sigma$ occurs in case (4).
	Then $\kappa(\sigma) = \sigma$ and $\kappa(\mu(\sigma)) = \mu(\sigma)$.
	Let $\sigma = [x_1|x_2|\dotsb|x_d]$ and $\mu(\sigma) = [x_1|\dotsb|x_{\delta-1}|x_\delta x_{\delta+1}|x_{\delta+2}|\dotsb|x_d]$, where $\delta = \delta(\sigma)$.
	As already noticed in the proof of \Cref{prop:involution}, the increasing factorization of $x_\delta x_{\delta+1}$ is given by the reflection $x_\delta$ followed by the increasing factorization of $x_{\delta+1}$.
	Therefore $\xi(\mu(\sigma)) = \xi(\sigma)$.
\end{proof}

\begin{lemma}
	Let $\sigma = [x_1|x_2|\dotsb|x_d] \in \F(K_W') \setminus \F(X_W')$ be a simplex such that $\pi(\sigma) = w$.
	There exist a negative vertical reflection $r \in R_0$ and an index $i \in \{1,2,\dotsc, d\}$ such that $r \leq x_i$ in $[1,w]$.
	In particular, the $\prec$-largest reflection appearing in $\xi(\sigma)$ is a negative vertical reflection.
	\label{lemma:negative-vertical-reflection}
\end{lemma}

\begin{proof}
	Since $\sigma \in \F(K_W')$, there is a $(d-1)$-simplex $\tau \in \F(X_W')$ to the left of $\sigma$.
	It has the following form:
	\[ \tau = [\varphi^h(x_{i+1})| \dotsb | \varphi^h(x_d) | \varphi^{h+1}(x_1) | \dotsb | \varphi^{h+1}(x_{i-1})], \]
	for some $h < 0$ and $i \in \{1,\dotsc, d\}$.
	We have that $\pi(\tau)$ fixes a vertex $\b$ of $C_0$, and $\varphi^h(x_i)$ is the left complement of $\pi(\tau)$.
	By \Cref{lemma:smallest-reflections}, there is a reflection $r' \leq \varphi^h(x_i)$ among the $n+1$ $\prec$-smallest reflections of $R_0$ (see the last part of the proof of \Cref{lemma:matching2}).
	By \Cref{lemma:vertical-reflection-conjugate}, $r = \varphi^{-h}(r')$ is a negative vertical reflection, because $h < 0$.
	This proves the first part of the statement because $r \leq x_i$.
	
	The $\prec$-increasing factorization of $x_i$ ends with the $\prec$-largest reflection $r''$ of $[1,x_i]$ by \Cref{thm:shellability}.
	If $\bar r$ is the $\prec$-largest reflection appearing in $\xi(\sigma)$, we have $\bar r \succeq r'' \succeq r$, so $\bar r$ is a negative vertical reflection.
\end{proof}

\begin{lemma}
	Let $\sigma, \tau \in \F(K_W') \setminus \F(X_W')$ be two simplices such that $\pi(\sigma) = w$ and $\tau$ is a face of $\sigma$.
	Then $\xi(\tau) \trianglelefteq \xi(\sigma)$.
	If, in addition, $\tau = \lambda(\sigma)$, then we have the strict inequality $\xi(\tau) \vartriangleleft \xi(\sigma)$.
	\label{lemma:matching-invariant}
\end{lemma}

\begin{proof}
	Let $\sigma = [x_1|x_2|\dotsb|x_d]$.
	Notice that $\kappa(\sigma) = \sigma$, because $\pi(\sigma) = w$.
	Let $r$ (resp.\ $r'$) be the $\prec$-largest reflection of $\xi(\tau)$ (resp.\ $\xi(\sigma)$), appearing in position $k$ (resp.\ $k'$).
	
	\begin{itemize}
		\item Case 1: $\tau = [x_2|\dotsb|x_d]$.
		Then $\tau = \mu(\sigma)$ and therefore $\xi(\tau) = \xi(\sigma)$ by \Cref{lemma:matching-compatibility}.
		
		\item Case 2: $\tau= [x_1|\dotsb | x_{i-1}|x_ix_{i+1}|x_{i+2}|\dotsb |x_d]$ for some $i \in \{1, \dotsc, d-1\}$.
		In particular, we have $\pi(\tau) = w$ and therefore $\kappa(\tau) = \tau$.
		Since $[1, x_i]$ and $[1, x_{i+1}]$ are both included in $[1, x_ix_{i+1}]$, the $\prec$-largest reflection of $[1, x_ix_{i+1}]$ is at least as $\prec$-large as the $\prec$-largest reflections of $[1,x_i]$ and $[1,x_{i+1}]$.
		Therefore $r \succeq r'$.
		If $r \succ r'$, then $\xi(\tau) \vartriangleleft \xi(\sigma)$, as desired.
		
		\noindent Suppose that $r = r'$.
		The $\prec$-largest reflection of $[1,x_ix_{i+1}]$ appears as the last reflection of the increasing factorization of $x_ix_{i+1}$.
		Therefore $k \geq k'$.
		If $k > k'$, then $\xi(\tau) \vartriangleleft \xi(\sigma)$, as desired.
		
		\noindent By \Cref{thm:shellability}, the increasing factorization of $x_i x_{i+1}$ is lexicographically smaller than (or equal to) the concatenation of the increasing factorizations of $x_i$ and $x_{i+1}$.
		Then $\xi(\tau)$ is lexicographically smaller than (or equal to) $\xi(\sigma)$.
		Therefore $\xi(\tau) \trianglelefteq \xi(\sigma)$.
		
		\item Case 3: $\tau = [x_1|\dotsb|x_{d-1}] = \lambda(\sigma)$.
		Then $\kappa(\tau) = [\varphi^{-1}(x_d)|x_1|\dotsb|x_{d-1}] \in \F(K_W') \setminus \F(X_W')$.
		
		\noindent Suppose by contradiction that $r \prec r'$.
		Since $x_1, \dotsc, x_{d-1}$ are common to both $\tau$ and $\sigma$, we have that $r' \leq x_d$.
		Then $\varphi^{-1}(r') \leq \varphi^{-1}(x_d)$, and therefore $\varphi^{-1}(r') \preceq r \prec r'$.
		By \Cref{lemma:negative-vertical-reflection}, $r'$ is a negative vertical reflection.
		Then $\Fix(\varphi^{-1}(r')) \cap \ell$ is above $\Fix(r') \cap \ell$.
		Since $\varphi^{-1}(r') \prec r'$, we have that $\varphi^{-1}(r')$ is a positive vertical reflection.
		By \Cref{lemma:vertical-reflection-conjugate}, $\varphi^{-1}(r')$ is among the $n+1$ $\prec$-smallest reflections of $R_0$.
		By \Cref{lemma:smallest-reflections}, its right complement $u$ fixes a vertex of $C_0$.
		Since $\varphi^{-1}(r') \leq \varphi^{-1}(x_d)$, passing to the right complements we get $u \ge x_1\dotsm x_{d-1} = \pi(\tau)$.
		By \Cref{lemma:model-poset}, $\Fix(u) \subseteq \Fix(\pi(\tau))$, and so $\pi(\tau)$ also fixes a vertex of $C_0$.
		This is a contradiction, because $\tau\not\in \F(X_W')$.
		Therefore $r \succeq r'$.
		If $r \succ r'$, then $\xi(\tau) \vartriangleleft \xi(\sigma)$, as desired.
		
		\noindent 
		Suppose now that $r = r'$.
		If $r = r' \leq x_d$, the previous argument yields again a contradiction.
		Therefore $r \leq x_i$ for some $i \in \{1, \dotsc, d-1\}$.
		Then the position $k$ (where $r$ appears in $\xi(\tau)$), is strictly greater than the position $k'$ (where $r$ appears in $\xi(\sigma)$).
		Thus $\xi(\tau) \vartriangleleft \xi(\sigma)$. \qedhere
	\end{itemize}
\end{proof}

\begin{lemma}
	The matching $\M$ on $\F(K_W')$ is acyclic.
	\label{lemma:matching-acyclic}
\end{lemma}

\begin{proof}
	Suppose by contradiction that there is an alternating cycle $\sigma_1 \gtrdot \tau_1 \lessdot \sigma_2 \gtrdot \tau_2 \lessdot \dotsb \gtrdot \tau_m \lessdot \sigma_{m+1} = \sigma_1$ in $\F(K_W')$, with $m\geq 1$.
	We have that $(\tau_j, \sigma_j) \not\in \M$ and $(\tau_j, \sigma_{j+1}) \in \M$ for all $j \in \{1, \dotsc, m\}$.
	In particular, all the simplices involved are matched, so they are in $\F(K_W') \setminus\F(X_W')$.
	Also, by \Cref{def:matching} we have that $\pi(\sigma_j) = w$ for all $j$.
	
	By \Cref{lemma:matching-compatibility} and \Cref{lemma:matching-invariant}, we have \[ \xi(\sigma_1) \trianglerighteq \xi(\tau_1) = \xi(\sigma_2) \trianglerighteq \xi(\tau_2) = \dotsb \trianglerighteq \xi(\tau_m) = \xi(\sigma_{m+1}) = \xi(\sigma_1). \]
	Then, all these inequalities are actually equalities:
	\[ \xi(\sigma_1) = \xi(\tau_1) = \xi(\sigma_2) = \xi(\tau_2) = \dotsb = \xi(\tau_m) = \xi(\sigma_{m+1}) = \xi(\sigma_1). \]
	By the second part of \Cref{lemma:matching-invariant}, we have $\tau_j \neq \lambda(\sigma_j)$ for all $j$.
	Also, $\tau_j \neq \rho(\sigma_j)$ for all $j$, because otherwise we would have $\tau_j= \mu(\sigma_j)$ i.e.\ $(\tau_j, \sigma_j) \in \M$.
	Since $\tau_j$ is a face of $\sigma_j$ different from $\lambda(\sigma_j)$ and $\rho(\sigma_j)$, we have $\pi(\tau_j) = \pi(\sigma_j) = w$ for all $j$.
	As a consequence, each $\tau_j$ occurs in case (3) of \Cref{def:matching}, and each $\sigma_j$ occurs in case (4).
	
	By the second part of \Cref{prop:involution}, $\delta(\tau_j) = \delta(\sigma_{j+1})$ for all $j$.
	In addition, since $\xi(\sigma_j) = \xi(\tau_j)$, we have $\delta(\sigma_j) \geq \delta(\tau_j)$ for all $j$.
	As before, since $\sigma_{m+1}= \sigma_1$, all inequalities are actually equalities:
	\[ \delta(\sigma_1) = \delta(\tau_1) = \delta(\sigma_2) = \delta(\tau_2) = \dotsb = \delta(\tau_m) = \delta(\sigma_{m+1}) = \delta(\sigma_1). \]
	
	Let $\sigma_1 = [x_1|x_2|\dotsb|x_d]$, $\tau_1 = [x_1|\dotsb|x_{i-1}|x_ix_{i+1}|x_{i+2}|\dotsb|x_d]$ for some $i \in \{1, \dotsc, d-1\}$, and $\delta = \delta(\sigma_1)$.
	Since $\sigma_1$ occurs in case (4), we have $l(x_\delta) = 1$.
	Also, $l(x_ix_{i+1}) \geq 2$ implies $\delta(\tau_1) \leq i$.
	Since $\delta(\tau_1) = \delta(\sigma_1) = \delta$, we deduce that $i \geq \delta$.
	If $i = \delta$, then $\tau_1 = \mu(\sigma_1)$, which is impossible because $(\tau_1, \sigma_1) \not\in \M$.
	Therefore $i > \delta$, so $\tau_1$ also occurs in case (4), because $l(x_\delta) = 1$.
	However $\tau_1$ occurs in case (3), and this is a contradiction.
\end{proof}

\begin{theorem}
	Let $W$ be an irreducible affine Coxeter group, with a set of simple reflections $S = \{s_1, s_2, \dotsc, s_{n+1}\}$ and a Coxeter element $w = s_1s_2\dotsm s_{n+1}$.
	The interval complex $K_W$ deformation retracts onto its subcomplex $X_W'$.
	\label{thm:deformation-retraction}
\end{theorem}

\begin{proof}
	By \Cref{thm:nice-subcomplex} and \Cref{lemma:canonical-nice}, $K_W$ deformation retracts onto its ca\-no\-ni\-cal nice subcomplex $K_W'$.
	We have constructed a matching $\M$ on the face poset $\F(K_W')$.
	This matching has $\F(X_W')$ as the set of critical cells and is acyclic by \Cref{lemma:matching-acyclic}.
	It is also proper because $\F(K_W')$ is finite.
	By the main theorem of discrete Morse theory (\Cref{thm:discrete-morse-theory}), $K_W'$ deformation retracts onto $X_W'$.
\end{proof}

We can finally prove the $K(\pi,1)$ conjecture for affine Artin groups.

\begin{theorem}[$K(\pi,1)$ conjecture]
	Let $W$ be an irreducible affine Coxeter group.
	The $K(\pi,1)$ conjecture holds for the corresponding Artin group $G_W$.
	\label{thm:conjecture}
\end{theorem}

\begin{proof}
	Fix a set of simple reflections $S = \{s_1, s_2, \dotsc, s_{n+1}\}$ and a Coxeter element $w = s_1s_2\dotsm s_{n+1}$.
	By \Cref{thm:KW-classifying}, the interval complex $K_W$ is a classifying space.
	By \Cref{thm:deformation-retraction,thm:dual-salvetti}, we have homotopy equivalences $Y_W \simeq X_W' \simeq K_W$, where $Y_W$ is the orbit configuration space associated with $W$.
	Therefore $Y_W$ is a classifying space for its fundamental group $G_W$.
\end{proof}

We also obtain a new proof of the following theorem of McCammond and Sulway.

\begin{theorem}[{\cite[Theorem C]{mccammond2017artin}}]
	Let $W$ be an irreducible affine Coxeter group, with a set of simple reflections $S = \{s_1, s_2, \dotsc, s_{n+1}\}$ and a Coxeter element $w = s_1s_2\dotsm s_{n+1}$.
	The natural homomorphism from the Artin group $G_W$ to the dual Artin group $W_w$ is an isomorphism.
\end{theorem}

\begin{proof}
	Consider the homotopy equivalences $X_W \simeq X_W' \simeq K_W$ of \Cref{thm:dual-salvetti,thm:deformation-retraction}.
	The composition $\psi \colon X_W \to K_W$ sends the $1$-cell $c_{\{s\}}$, associated with a simple reflection $s \in S$, to the corresponding $1$-cell $[s]$ of $K_W$, preserving the orientation.
	Then the induced map $\psi_*$ on the fundamental groups is exactly the natural homomorphism $G_W \to W_w$, which is, therefore, an isomorphism.
\end{proof}

\appendix
\section{The four infinite families}
\label{appendix}

The main purpose of this appendix is to prove \Cref{lemma:axis-A}, \Cref{thm:coxeter-A}, \Cref{prop:axial-point-A,prop:axial-orbits-A} (these are statements about $\tilde A_n$), and \Cref{lemma:hyperbolic-coxeter} (for the cases $\tilde A_n$, $\tilde B_n$, $\tilde C_n$, and $\tilde D_n$).
We do this by explicitly examining the four infinite families of irreducible affine Coxeter groups.
This appendix can be also used as a source of examples, and it complements the computations of \cite[Section 11]{mccammond2015dual}.
We refer to \cite[Section 2.10]{humphreys1992reflection} for the standard construction of root systems (see also \cite{bourbaki1968elements}).

\subsection{Case \texorpdfstring{$\tilde A_n$}{An tilde}}
\label{sec:appendix-A}

Let $W$ be a Coxeter group of type $\tilde A_n$. It is realized as the reflection group in $E = \R^{n+1} / \langle 1,\dotsc, 1 \rangle$ associated with the hyperplane arrangement
\[ \A= \{ \{ x_i - x_j = k \} \mid 1 \leq i < j \leq n+1, \; k \in \Z \}. \]
If $\a \in \R^{n+1}$, denote by $[\a]$ its class in $E = \R^{n+1}/\langle 1, \dotsc, 1\rangle$.

Let $(p,q)$ be a pair of positive integers such that $p+q = n+1$.
Label the coordinates of $\R^{n+1} = \R^p \times \R^q$ as follows: $x_1,x_2,\dotsc, x_p, y_1, y_2, \dotsc, y_q$.
Given a point $\b \in E$, denote its coordinates by $x_1^\b, \dotsc, x_p^\b, y_1^\b, \dotsc, y_q^\b$ (they are well defined up to a multiple of $(1, \dotsc, 1)$).
Construct a $(p,q)$-bigon Coxeter element $w$ as in \cite[Example 11.6]{mccammond2015dual}:
\begin{equation}
\label{eq:coxeter-A}
w (\b) = [x_p^\b+1, x_1^\b, \dotsc, x_{p-1}^\b \mid y_q^\b-1, y_1^\b, \dotsc, y_{q-1}^\b].
\end{equation}
Then the shortest vector in $\Mov(w)$ is $\mu = \left[\frac{1}{p}, \dotsc, \frac{1}{p} \bigm| -\frac{1}{q}, \dotsc, -\frac{1}{q}\right]$, and the points of the Coxeter axis $\ell$ (i.e.\ the points $\a$ such that $w(\a) = \a + \mu$) are of the form
\begin{equation}
\label{eq:axial-point}
\left[ \textstyle \frac{p-1}{p}, \frac{p-2}{p}, \dotsc, \frac{1}{p}, 0 \bigm| 0, \frac{1}{q}, \dotsc, \frac{q-2}{q}, \frac{q-1}{q} \right] + \theta\mu
\end{equation}
for $\theta \in \R$.
The hyperplanes of the form $\{ x_i - x_j = k\}$ or $\{ y_i - y_j = k\}$ are horizontal, whereas those of the form $\{ x_i - y_j = k\}$ are vertical.

\begin{proof}[Proof of \Cref{lemma:axis-A}]
	From \eqref{eq:axial-point} it is immediate to see that the Coxeter axis $\ell$ is not contained in any hyperplane of $\A$.
	The value of $\theta$ that yields the intersection point of a vertical hyperplane $\{ x_i -  y_j = k \}$ with $\ell$ satisfies
	\[ \textstyle\left(\frac1p + \frac 1q \right)\theta = k - \textstyle\frac{p-i}{p} + \frac{j-1}{q} = k-1 + \frac{i}{p} + \frac{j-1}{q}. \]
	If we let $k$, $i$, and $j$ vary, then $\theta$ can assume any value which is an integer multiple of $\frac{\gcd(p,q)}{p+q}$.
\end{proof}

Consider now a point $\a \in E$ which is not contained in any hyperplane of $\A$, and let $C_\a$ be the chamber containing $\a$.
In particular, for every $j$ we have that $y_j^\a - x_p^\a \not\in\Z$, because otherwise $\a$ would lie on some vertical hyperplane $\{ x_p - y_j = k \}$.
Consider the line $\ell_\a$ passing through $\a$ and with the same direction as the Coxeter axis:
\[ \ell_\a = \big\{ [x_1^\a, x_2^\a, \dotsc, x_p^\a \mid y_1^\a, y_2^\a, \dotsc, y_q^\a] + \theta\mu \mid \theta \in \R \big\}. \]
Let $S_\a \subseteq R$ be the set of the reflections with respect to the walls of $C_\a$.
Write $S_\a = S_\a^+ \sqcup S_\a^- \sqcup S_\a^\h$, where $S_\a^+$ (resp.\ $S_\a^-$) consists of the reflections that intersect $\ell_\a$ above (resp.\ below) $\a$, and $S_\a^\h$ consists of the horizontal reflections.

\begin{lemma}
	\label{lemma:product-A}
	Let $W$ be a Coxeter group of type $\tilde A_n$, and $w$ a $(p,q)$-bigon Coxeter element as in \eqref{eq:coxeter-A}.
	Let $\a \in E$ be a point which is not contained in any hyperplane of $\A$, and such that $x_p^\a < x_{p-1}^\a < \dotsb < x_1^\a < x_p^\a +1$ and $y_1^\a < y_2^\a < \dotsb < y_q^\a < y_1^\a + 1$.
	Then the reflections in $S_\a^+$ (resp.\ $S_\a^-$) pairwise commute.
	In addition, $w$ can be written as a product of the reflections in $S_\a$, where the reflections in $S_\a^+$ come first, and the reflections in $S_\a^-$ come last.
\end{lemma}

\begin{proof}
Since the coordinates of $\a$ are defined up to a multiple of $(1, \dotsc, 1)$, we can assume that $x_p^\a= 0$.
Therefore we have $0 = x_p^\a < x_{p-1}^\a < \dotsb < x_1^\a < 1$.
Let $s$ be the (unique) index such that the fractional part of $y_s^\a$ is minimal, and let $h = \lfloor y_s^\a \rfloor$ be the largest integer which is less than $y_s^\a$.
Then we have
\[ h < y_s^\a < y_{s+1}^\a < \dotsb < y_q^\a < y_1^\a + 1 < \dotsb < y_{s-1}^\a + 1 < h+1. \]

Let $X = \{ x_1^\a, \dotsc, x_p^\a \}$ and $Y = \{ y_s^\a-h, \dotsc, y_{q}^\a - h, y_1^\a + 1 - h, \dotsc, y_{s-1}^\a + 1 - h \}$.
The set $Z = X \cup Y$ consists of $n+1$ distinct real numbers between $0$ (included) and $1$ (excluded).
Write $Z = \{ 0 = z_1^\a < z_2^\a < \dotsb < z_{n+1}^\a \}$, where each $z_l$ stands either for some $x_i$ or for some translate of some $y_j$ (notice that $z_1 = x_p$).
Then the inequalities $z_1 < z_2 < \dotsb < z_{n+1} < z_1+1$ define the chamber $C_\a$, and the walls of $C_\a$ are:
\begin{equation}
	\{ z_1 = z_2\}, \{ z_2 = z_3 \}, \dotsc, \{z_{n+1} = z_1 + 1 \}.
	\label{eq:walls}
\end{equation}

For every $i \in \{1,\dotsc, p\}$, the coordinate $x_i$ appears in exactly two walls.
If they are both vertical, say $\{ x_i - y_{j} = k\}$ and $\{x_{i} - y_{j+1} = k' \}$, then the first of these walls intersects $\ell_\a$ below $\a$, because $x^\a_{i} - y_{j}^\a > k$, whereas the second one intersects $\ell_\a$ above $\a$, because $x^\a_i - y_{j+1}^\a < k'$ (here the indices of $y$ are taken modulo $q$).
A similar argument applies to the coordinates $y_j$ for $j \in \{1, \dotsc, q\}$.
This proves that the reflections in $S_\a^+$ (resp.\ $S_\a^-$) pairwise commute.

The horizontal walls are among the following: $\{ x_i = x_{i+1} \}$ for $1 \leq i \leq p-1$; $\{x_1 = x_p + 1\}$; $\{y_j = y_{j+1}\}$ for $1 \leq j \leq q-1$; $\{y_1 = y_q-1\}$.
Let $t$ be the smallest index such that $x_t^\a < y_s^\a - h$.
Order the possible horizontal walls as follows:
\begin{align}
\label{eq:horizontal-hyperplanes-A}
	\begin{split}
	& \{ x_{t+1} = x_{t+2} \}, \, \{ x_{t+2} = x_{t+3} \}, \, \dotsc, \, \{x_{p-1} = x_p \}, \\
	& \{ x_1 = x_p + 1 \}, \, \{ x_1 = x_2 \}, \, \dotsc, \, \{ x_{t-1} = x_t \}, \\
	& \{ y_s = y_{s+1} \}, \, \{ y_{s+1} = y_{s+2} \}, \, \dotsc, \, \{ y_{q-1} = y_q \}, \\
	& \{ y_1 = y_q - 1 \}, \, \{ y_1 = y_2 \}, \, \dotsc, \, \{ y_{s-2} = y_{s-1} \}.
	\end{split}
\end{align}
Let $w_+$ (resp.\ $w_-$) be the product of the reflections in $S_\a^+$ (resp.\ $S_\a^-$), and let $w_\h$ be the product of the reflections in $S_\a^\h$, in the same relative order as in \eqref{eq:horizontal-hyperplanes-A}.
Let $\hat w = w_+ w_\h w_-$.
We want to prove that $\hat w = w$.
For this, it is enough to show that the linear part of $\hat w$ coincides with the linear part of $w$, and that $\hat w (\b) = w(\b)$ for at least one point $\b \in E$.

Denote by $e_{x_i}$ (resp.\ $e_{y_j}$) the unit vector in the direction of $x_i$ (resp.\ $y_j$).
Given two elements $\zeta, \zeta' \in Z$, we write $\zeta \lessdot \zeta'$ if $\zeta < \zeta'$ and there exists no $\zeta'' \in Z$ between $\zeta$ and $\zeta'$.
For $j \in \{1,\dotsc, q\}$ define
\[ k_j = \begin{cases}
	-h & \text{if $j \geq s$} \\
	1-h & \text{othewise},
\end{cases} \]
so that $Y$ consists of the real numbers $y_j^\a + k_j$.

Consider the unit vector $e_{x_i}$ with $i \leq p-1$.
\begin{itemize}
	\item If $x_{i+1}^\a \lessdot x_i^\a$, then: the linear part of $w_-$ fixes $e_{x_i}$; the linear part of $w_\h$ sends $e_{x_i}$ to $e_{x_{i+1}}$; the linear part of $w_+$ fixes $e_{x_{i+1}}$.

	\item If there is at least one element of $Z$ between $x_{i+1}^\a$ and $x_i^\a$, say $x_{i+1}^\a \lessdot y_j^\a + k_j \lessdot y_{j+1}^\a + k_{j+1} \lessdot \dotsb \lessdot y_{j'}^\a + k_{j'} \lessdot x_i^\a$, then: the linear part of $w_-$ sends $e_{x_i}$ to $e_{y_{j'}}$; the linear part of $w_\h$ sends $e_{y_{j'}}$ to $e_{y_j}$; the linear part of $w_+$ sends $e_{y_j}$ to $e_{x_{i+1}}$.
\end{itemize}
Consider now the unit vector $e_{x_p}$.
\begin{itemize}
	\item If $x_1^\a$ is the maximal element of $Z$, then: the linear part of $w_-$ fixes $e_{x_p}$; the linear part of $w_\h$ sends $e_{x_p}$ to $e_{x_{1}}$; the linear part of $w_+$ fixes $e_{x_{1}}$.
	
	\item Otherwise, if $x_1^\a \lessdot y_j^\a + k_j \lessdot y_{j+1}^\a + k_{j+1} \lessdot \dotsb \lessdot y_{s-1}^\a + k_{s-1}$, then: the linear part of $w_-$ sends $e_{x_p}$ to $e_{y_{s-1}}$; the linear part of $w_\h$ sends $e_{y_{s-1}}$ to $e_{y_j}$; the linear part of $w_+$ sends $e_{y_j}$ to $e_{x_{1}}$.
\end{itemize}
A similar argument shows that the linear part of $\hat w$ sends $e_{y_j}$ to $e_{y_{j+1}}$ if $j \leq q-1$, and sends $e_{y_q}$ to $e_{y_1}$.
Therefore the linear part of $\hat w$ coincides with the linear part of $w$ (see \eqref{eq:coxeter-A}).

It remains to show that $\hat{w}(\b) = w(\b)$ for some point $\b$.
Let $\b$ be the vertex of $C_\a$ opposite to the wall $H = \{ x_{t} - y_s = -h \}$.
The point $\b$ is the intersection of all the other walls of $C_\a$, so its coordinates are explicitly determined by the following equations:
\begin{IEEEeqnarray*}{rCl}
	\IEEEeqnarraymulticol{3}{l}{x_p^\b = x_{p-1}^\b = \dotsb = x_{t}^\b} \\
	y_s^\b - h &=& y_{s+1}^\b - h = \dotsb = y_q^\b - h = y_{1}^\b + 1 -h = \dotsb= y_{s-1}^\b + 1 -h \\
	&=& x_{t-1}^\b = x_{t-2}^\b = \dotsb = x_1^\b = x_p^\b + 1.
\end{IEEEeqnarray*}
Therefore,
\[ \b = [\, \underbrace{1, \dotsc, 1}_{t-1}, \underbrace{0, \dotsc, 0}_{p-t+1} \mid \underbrace{h, \dotsc, h}_{s-1}, \underbrace{h+1, \dotsc, h+1}_{q-s+1} ]. \]
Recall that the reflection $r$ with respect to $H$ belongs to $S_\a^+$.
Since $\b$ is fixed by all the other reflections in $S_\a$, we have that
\[ \hat w(\b) = r(\b) = [\, \underbrace{1, \dotsc, 1}_{t}, \underbrace{0, \dotsc, 0}_{p-t} \mid \underbrace{h, \dotsc, h}_{s}, \underbrace{h+1, \dotsc, h+1}_{q-s} ] = w(\b). \qedhere \]
\end{proof}

\begin{proof}[Proof of \Cref{thm:coxeter-A}]
	Let $\a$ be a point in $C \cap \ell$, as in \eqref{eq:axial-point}.
	\Cref{lemma:product-A} immediately implies points (ii) and (iii).
	
	Since $p \geq q$, none of the walls of $C$ is of the form $\{y_j - y_{j'} = k\}$.
	In addition, for every $j \in \{1,\dotsc, q\}$, the coordinate $y_j$ appears in the defining equation of exactly two walls of $C$.
	The corresponding reflections are one in $S^+$ and one in $S^-$.
	Therefore $|S^+| = |S^-| = q$ and $|S^\h| = n+1 - 2q = p-q$, proving point (i).
\end{proof}

\begin{proof}[Proof of \Cref{prop:axial-point-A}]
	Let $d = \gcd(p,q)$.
	For part (i), write $\a$ as in \eqref{eq:axial-point} for some $\theta \in \R$.
	Let $\{x_i - y_j = k\}$ be a vertical hyperplane containing $\a$.
	In particular, we have $x_i^\a - y_j^\a \in \Z$.
	For another hyperplane $\{ x_{i'} - y_{j'} = k' \}$ to contain $\a$, we need $x_{i'}^\a - y_{j'}^\a \in \Z$.
	By \eqref{eq:axial-point}, we have that
	\[ (x_i^\a - y_j^\a) - (x_{i'}^\a - y_{j'}^\a) = \textstyle \left( \frac{p-i+\theta}{p} - \frac{j-1-\theta}{q}\right) - \left( \frac{p-i'+\theta}{p} - \frac{j'-1-\theta}{q} \right) = \frac{i'-i}{p} + \frac{j'-j}{q}. \]
	This is an integer if and only if $i' = i + u \cdot \frac{p}{d}$ and $j' = j + v \cdot \frac{q}{d}$ for some $u, v \in \Z$ such that $d \mid u+v$.
	There are $d$ such pairs $(i',j') \in \{1,\dotsc,p\} \times \{1, \dotsc, q\}$, and each of them yields exactly one hyperplane containing $\a$.
	All these hyperplanes pairwise commute, because no two such pairs share the same $i'$ or the same $j'$.
	
	For part (ii), fix a point $\a \in C \cap \ell$.
	Let $H = \{ x_i -y_j = k \}$ be a vertical wall of $C$, and denote by $\b$ the intersection point of $H$ with $\ell$.
	By the description of the walls of $C$ given in the proof of \Cref{lemma:product-A}, we have either $x_i^\a < y_j^\a + k < x_i^\a + \frac 1p$ or $x_i^\a - \frac1p < y_j^\a + k < x_i^\a$.
	In any case, $|x_i^\a - y_j^\a - k| < \frac 1p$.
	Since $b \in H$, we also have $x_i^\b - y_j^\b - k = 0$, and therefore $|(x_i^\a - y_j^\a) - (x_i^\b - y_j^\b)| < \frac 1p$.
	Then, if we write $\a$ and $\b$ as in \eqref{eq:axial-point} for some $\theta_a, \theta_b \in \R$, we obtain that $|\theta_a - \theta_b| < \frac{q}{p+q}$.
	By \Cref{lemma:axis-A}, consecutive points of the sequence $\{p_i\}_{i \in \Z}$ differ by $\frac{\gcd(p,q)}{p+q} \mu$.
	Therefore there are $\frac{q}{\gcd(p,q)} = m$ possible positions for $\b$ above $\a$, and $m$ possible positions below $\a$.
	More precisely, if $\a$ is between $\p_i$ and $\p_{i+1}$, then $\b$ must be one of the following $2m$ points: $\p_{i-m+1}, \p_{i-m+2}, \dotsc, \p_{i+m}$.
	By part (i), each of these points is contained in exactly $\gcd(p,q)$ vertical hyperplanes of $\A$.
	By \Cref{thm:coxeter-A}, the chamber $C$ has exactly $2q = 2m \cdot \gcd(p,q)$ vertical walls.
	Therefore every hyperplane of $\A$ that intersects $\ell$ in one of the previous $2m$ points must be a wall of $C$.
\end{proof}

\begin{proof}[Proof of \Cref{prop:axial-orbits-A}]
	We begin with part (i).
	Every point $\a \in \ell$ satisfies $x_i^\a - 1 \leq x_{i'}^\a \leq x_i^\a$ for every $i < i'$, and $y_j^\a \leq y_{j'}^\a \leq y_j^\a+1$ for every $j < j'$.
	Then the same non-strict inequalities need to be satisfied by every axial vertex $\b$.
	The walls of an axial chamber $C$ have the form \eqref{eq:walls}, so every vertex of $C$ admits an expression where all the coordinates are integers.
	Therefore every axial vertex has the following form:
	\begin{equation}
		\label{eq:axial-vertex}
		b = [\underbrace{1, \dotsc, 1}_{p_1}, \underbrace{0, \dotsc, 0}_{p_2} \mid \underbrace{h, \dotsc, h}_{q_1}, \underbrace{h+1, \dotsc, h+1}_{q_2} ]
	\end{equation}
	for some $h \in \Z$, with $q_2 \geq 1$ (otherwise we can change $h$ with $h+1$).
	Let $A \subseteq E$ be the set of points of the form \eqref{eq:axial-vertex}.
	The representation \eqref{eq:axial-vertex} of a point $b \in A$ becomes unique, and we call it the \emph{standard form} of $b$, if we also impose $p_2 \geq 1$ (otherwise we can remove $1$ from all coordinates).
	
	By \eqref{eq:coxeter-A}, the Coxeter element $w$ acts on $b \in A$ as follows:
	\[
		w(b) = [\underbrace{1, \dotsc, 1}_{p_1+1}, \underbrace{0, \dotsc, 0}_{p_2-1} \mid \underbrace{h, \dotsc, h}_{q_1+1}, \underbrace{h+1, \dotsc, h+1}_{q_2-1} ].
	\]
	In particular $w(b) \in A$, so there is an action of $\Z$ on $A$ by powers of $w$.
	If $p_2 = 1$, then we can remove $1$ from all coordinates in the previous equation to obtain the standard form
	\[
		w(b) = [0, \dotsc, 0 \mid \underbrace{h-1, \dotsc, h-1}_{q_1+1}, \underbrace{h, \dotsc, h}_{q_2-1} ].
	\]
	
	Denote by $\sigma(b)$ the sum of the coordinates of the standard form of a point $b \in A$.
	We have that $\sigma(w(b)) \equiv \sigma(b)$ modulo $p+q$.
	By looking at the walls of $C$ given by \eqref{eq:walls}, we see that the $p+q$ vertices of an axial chamber $C$ have distinct values of $\sigma$, modulo $p+q$.
	Therefore there are at least $p+q$ different orbits for the action of $\Z$ on $A$, and the vertices of an axial chamber belong to different orbits.

	It remains to show that there are exactly $p+q$ orbits.
	For every point $b \in A$, there is a point $w^j(b)$ which is in one of the following forms:
	\begin{enumerate}[(a)]
		\item $[1, \dotsc, 1, 0 \mid \underbrace{0, \dotsc, 0}_{q_1}, \underbrace{1, \dotsc, 1}_{q_2} ]$ with $q_2 \geq 1$;
		
		\medskip
		\item $[\underbrace{1, \dotsc, 1}_{p_1}, \underbrace{0, \dotsc, 0}_{p_2} \mid 0, \dotsc, 0 ]$ with $p_2 \geq 1$.
	\end{enumerate}
	There are $q$ points of the form (a), and $p$ points of the form (b).
	Therefore there are exactly $p+q$ orbits.
	Notice that this also proves that every point of $A$ is an axial vertex because it is in the orbit of some axial vertex.
	
	To prove part (ii), fix an axial vertex $\b \in A$.
	Without loss of generality, we may assume that $\b$ is of the form (a) or (b).
	We want to describe the set of points $\a \in \ell$ that are contained in some axial chamber $C$ having $\b$ as one of its vertices.
	Let $\a$ be as in \eqref{eq:axial-point}, for some $\theta \in \R$.
	
	If $\b$ is of the form (a), then $x_p^\b < y_j^\b$ for $j \geq q_1+1$, and $y_j^\b < x_i^\b$ for $i \leq p-1$ and $j \leq q_1$.
	The same strict inequalities need to be satisfied by $\a$, an so we get $\frac{\theta}{p} < \frac{j-1-\theta}{q}$ for $j \geq q_1 + 1$, and $\frac{j-1-\theta}{q} < \frac{p-i+\theta}{p}$ for $i \leq p-1$ and $j \leq q_1$.
	These conditions are equivalent to $\frac{q_1-1}{q} - \frac1p < \left(\frac1p + \frac1q \right) \theta < \frac{q_1}{q}$, so $\theta$ belongs to an interval of length $1$.
	Conversely, every value of $\theta$ in this interval yields a point $\a$ contained in a chamber $C$ which has $\b$ as one of its vertices, provided that we exclude the finite set of values corresponding to points that belong to some vertical hyperplane of $\A$.
	By \Cref{lemma:axis-A}, this interval of values of $\theta$ spans exactly $\frac{p+q}{\gcd(p,q)}$ axial chambers.
	
	If $\b$ is of the form (b), the procedure is similar.
	We have $y_j^\b < x_i^\b$ for all $j$ and for $i \leq p_1$, and $y_j^\b > x_i^\b - 1$ for all $j$ and for $i \geq p_1+1$.
	The corresponding inequalities for the point $\a$ give the condition $\frac{p_1}{p} -\frac1q \leq \left( \frac1p + \frac1q \right) \theta \leq \frac{p_1}{p} + \frac1p$, which again yields an interval of length $1$.
\end{proof}

In order to prove \Cref{lemma:hyperbolic-coxeter}, we start by explicitly describing the hyperbolic elements $u \in [1,w]$ with $l(u) = n$.
These are obtained as $u = wr$ where $r \in [1,w]$ is a horizontal reflection.
By \Cref{thm:reflections-A}, the hyperplanes corresponding to horizontal reflections in $[1,w]$ have the following forms:
\begin{enumerate}[(i)]
	\item $\{ x_i = x_{i'}\}$ for $i < i'$;
	\item $\{ x_i = x_{i'} + 1\}$ for $i < i'$;
	\item $\{ y_j = y_{j'}\}$ for $j < j'$;
	\item $\{ y_j = y_{j'} - 1\}$ for $j < j'$.
\end{enumerate}

If $r$ is the reflection with respect to $\{ x_i = x_{i'}\}$ with $i < i' \leq p-1$, then $u = wr$ sends a point $\b \in E$ to
\[ [x_p^\b+1, x_1^\b, \dotsc, x_{i-1}^\b, x_{i'}^\b, x_{i+1}^\b, \dotsc, x_{i'-1}^\b, x_i^\b, x_{i'+1}^\b, \dotsc, x_{p-1}^\b \mid y_q^\b-1, y_1^\b, \dotsc, y_{q-1}^\b]. \]
On the coordinates $x_{i+1}, \dotsc, x_{i'}$, the hyperbolic isometry $u$ acts as a horizontal Coxeter element of type $A_{i'-i-1}$.
On the remaining coordinates, it acts as a $(p-i'+i,q)$-bigon Coxeter element of type $\tilde A_{n-i'+i}$.
This is exactly the hyperbolic-horizontal decomposition of $u$ (\Cref{lemma:hyperbolic-decomposition}).

If $r$ is the reflection with respect to $\{x_i = x_p \}$ for some $i \leq p-1$, then $u=wr$ sends $\b \in E$ to
\[ [x_i^\b+1, x_1^\b, \dotsc, x_{i-1}^\b, x_{p}^\b, x_{i+1}^\b, \dotsc, x_{p-1}^\b \mid y_q^\b-1, y_1^\b, \dotsc, y_{q-1}^\b]. \]
As before, on the coordinates $x_{i+1}, \dotsc, x_p$ we have that $u$ acts as a horizontal Coxeter element of type $A_{p-i-1}$, and on the remaining coordinates it acts as a $(i, q)$-bigon Coxeter element of type $\tilde A_{n-p+i}$.

If $r$ is the reflection with respect to $\{x_i = x_{i'}+1\}$ with $i < i' \leq p-1$, then $u=wr$ sends $\b \in E$ to
\[ [x_p^\b+1, x_1^\b, \dotsc, x_{i-1}^\b, x_{i'}^\b + 1, x_{i+1}^\b, \dotsc, x_{i'-1}^\b, x_i^\b - 1, x_{i'+1}^\b, \dotsc, x_{p-1}^\b \mid y_q^\b-1, \dotsc, y_{q-1}^\b]. \]
Then $u$ acts as a $(i'-i,q)$-bigon Coxeter element of type $\tilde A_{q+i'-i-1}$ on the coordinates $x_{i+1}, \dotsc, x_{i'}, y_1, \dotsc, y_q$, and as a horizontal Coxeter element of type $A_{p-i'+i-1}$ on the remaining coordinates.

If $r$ is the reflection with respect to $\{x_i = x_p + 1\}$ for some $i \leq p-1$, then $u=wr$ sends $\b \in E$ to
\[ [x_i^\b, x_1^\b, \dotsc, x_{i-1}^\b, x_{p}^\b + 1, x_{i+1}^\b, \dotsc, x_{p-1}^\b \mid y_q^\b-1, y_1^\b, \dotsc, y_{q-1}^\b]. \]
As before, on the coordinates $x_{i+1}, \dotsc, x_p, y_1, \dotsc, y_q$ we have that $u$ acts as a $(p-i, q)$-bigon Coxeter element of type $\tilde A_{n-i}$, and on the remaining coordinates it acts as a horizontal Coxeter element of type $A_{i-1}$.

A similar phenomenon happens if $r$ is a reflection of type (iii) or (iv).
We are now ready to prove \Cref{lemma:hyperbolic-coxeter}.

\begin{proof}[Proof of \Cref{lemma:hyperbolic-coxeter} for the case $\tilde A_n$]
	Let $u \in [1,w]$ be a hyperbolic isometry such that $W_u$ is irreducible.
	By iterating the previous argument, we get that $u$ acts as a $(p',q')$-bigon Coxeter element of type $\tilde A_{p'+q'-1}$ on a subset $x_{i_1}, \dotsc, x_{i_{p'}}, y_{j_1}, \dotsc, y_{j_{q'}}$ of the coordinates (for some $i_1 < \dotsb < i_{p'}$ and $j_1 < \dotsb < j_{q'}$).
	It acts as the identity on the other coordinates, because otherwise $u$ would have a non-trivial hyperbolic-horizontal decomposition, and $W_u$ would be reducible.
	If we restrict to the relevant $p'+q'$ coordinates, we get
	\[ u(\b) = [x_{i_{p'}}^\b + 1, x_{i_1}^\b, \dotsc, x_{i_{p'-1}}^\b \mid y_{j_{q'}}^\b - 1, y_{j_1}^\b, \dotsc, y_{j_{q'-1}}^\b ]. \]
	Then $W_u$ is a Coxeter group of type $\tilde A_{p'+q'-1}$.
	
	The point $\a$ of the statement can be written in the form \eqref{eq:axial-point}, and its relevant $p'+q'$ coordinates are given by
	\[ \left[ \textstyle \frac{p-i_1}{p}, \frac{p-i_2}{p}, \dotsc, \frac{p-i_{p'}}{p} \bigm| \frac{j_1-1}{q}, \frac{j_2-1}{q}, \dotsc, \frac{j_{q'}-1}{q} \right] + \theta \, \left[\textstyle\frac1p,\dotsc, \frac1p \mid -\frac1q, \dotsc, -\frac1q\right]. \]
	In particular, notice that $x_{i_{p'}}^\a < x_{i_{p'-1}}^\a < \dotsb < x_{i_1}^\a < x_{i_{p'}}^\a + 1$ and $y_{j_1}^\a < y_{j_2}^\a < \dotsb < y_{j_{q'}}^\a < y_{j_1}^\a + 1$.
	We conclude by applying \Cref{lemma:product-A} to the Coxeter group $W_u$ (in place of $W$), its Coxeter element $u$ (in place of $w$), and the point $\a$.
\end{proof}

\subsection{Case \texorpdfstring{$\tilde C_n$}{Cn tilde}}
\label{sec:appendix-C}

Let $W$ be a Coxeter group of type $\tilde C_n$.
It is realized as the reflection group in $E = \R^n$ associated with the hyperplane arrangement
\[ \A = \Big\{ \{x_i \pm x_j = k \} \Bigm|  1 \leq i < j \leq n, \; k \in \Z \Big\}  \cup \Big\{ \left\{x_i = \textstyle \frac{k}{2} \right\} \Bigm|  1 \leq i \leq n, \; k \in \Z \Big\}. \]

Consider the chamber $C_0 = \left\{ 0 < x_1 < x_2 < \dotsb < x_n < \textstyle \frac12 \right\}$, with walls given by $\{x_1 = 0\}, \{x_1 = x_2\}, \dotsc, \{x_{n-1} = x_n\}, \left\{ x_n = \frac 12 \right\}$.
Let $w$ be the Coxeter element obtained by multiplying the reflections with respect to the walls of $C_0$ (in the order given before).
Then $w$ acts as follows on a point $\b \in E$:
\begin{equation}
	\label{eq:coxeter-C}
	w(\b) = (x_n^\b - 1, x_1^\b, \dotsc, x_{n-1}^\b).
\end{equation}
The shortest vector in $\Mov(w)$ is $\mu = -\left( \frac 1n, \frac 1n, \dotsc, \frac 1n \right)$, and the points of the Coxeter axis $\ell$ are of the form
\[ \left( \textstyle 0, \frac 1n, \frac 2n, \dotsc, \frac{n-1}{n}\right) + \theta \mu \]
for $\theta \in \R$.
The hyperplanes of the form $\{ x_i - x_j = k\}$ are horizontal, and the other hyperplanes of $\A$ are vertical.

Given a point $\a \in E$ which is not contained in any hyperplane of $\A$, let $C_\a$ be the chamber containing $\a$, and consider the line passing through $\a$ and with the same direction as the Coxeter axis:
\[ \ell_\a = \big\{ (x_1^\a, x_2^\a, \dotsc, x_n^\a) + \theta\mu \mid \theta \in \R \big\}. \]
Define $S_\a, S_\a^+, S_\a^-, S_\a^\h$ as in the case $\tilde A_n$ (Section \ref{sec:appendix-A}).

\begin{lemma}
	\label{lemma:product-C}
	Let $W$ be a Coxeter group of type $\tilde C_n$, and $w$ a Coxeter element as in \eqref{eq:coxeter-C}.
	Let $\a \in E$ be a point which is not contained in any hyperplane of $\A$, and such that $x_1^\a < x_2^\a < \dotsb < x_n^\a < x_1^\a +1$.
	Then the reflections in $S_\a^+$ (resp.\ $S_\a^-$) pairwise commute.
	In addition, $w$ can be written as a product of the reflections in $S_\a$, where the reflections in $S_\a^+$ come first, and the reflections in $S_\a^-$ come last.
\end{lemma}

\begin{proof}
	The statement holds for $\a$ if and only if it holds for $w^m(\a)$, for any $m \in \Z$.
	Notice that $w$ permutes cyclically the fractional parts of the real numbers $x_1^\a, \dotsc, x_n^\a$.
	Therefore, without loss of generality, we can assume that $x_1^\a$ has the smallest fractional part among $x_1^\a, \dotsc, x_n^\a$.
	Since $w^n$ is a pure translation of $n\mu = -(1, 1,\dotsc, 1)$, we can also assume that $0 < x_1^\a < 1$.
	Together with the hypothesis $x_1^\a < x_2^\a < \dotsb < x_n^\a < x_1^\a +1$, we obtain that
	\[ 0 < x_1^\a < x_2^\a < \dotsb < x_n^\a < 1. \]
	
	Let $p \in \{0,\dotsc, n\}$ be the largest index such that $x_p^\a < \frac 12$.
	Let $q = n-p$, and define $y_j = 1 - x_{p+j}$ for $j \in \{1,\dotsc, q\}$.
	Notice that the isometry
	\[ (x_1, \dotsc, x_n) \mapsto (x_1, \dotsc, x_p, 1-x_{p+1}, \dotsc, 1-x_n) \]
	is an element of $W$, so it sends chambers to chambers.
	Since $\a$ is not contained in any hyperplane of $\A$, the new coordinates $x_1^\a, \dotsc, x_p^\a$, $y_1^\a, \dotsc, y_q^\a$ are pairwise distinct numbers between $0$ and $\frac 12$.
	They satisfy $x_1^\a < \dotsb < x_p^\a$ and $y_1^\a > \dotsb > y_q^\a$.
	Let $Z = \{x_1^\a, \dotsc, x_p^\a, y_1^\a, \dotsc, y_q^\a \}$, and write $Z = \{ z_1^\a < z_2^\a < \dotsb < z_n^\a \}$, where each $z_l$ is either equal to some $x_i$ or to some $y_j$.
	The inequalities $0 < z_1 < z_2 < \dotsb < z_n < \frac 12$ define the chamber $C_\a$, and the walls of $C_\a$ are
	\begin{equation}
		\{ z_{1} = 0\}, \{ z_{1} = z_{2} \}, \dotsc, \{ z_{n-1} = z_{n} \}, \big\{ z_{n} = \textstyle\frac12 \big\}.
		\label{eq:walls-C}
	\end{equation}
	The vertical walls that intersect $\ell_\a$ above $\a$ are those of the form $\{ y_j = x_i \}$ with $y_j^\a \lessdot x_i^\a$, $\{ x_1 = 0 \}$, and $\big\{ y_1 = \frac12 \big\}$.
	Notice that not all of these hyperplanes necessarily occur as walls of $C_\a$.
	For example, $\{x_1 = 0\}$ is a wall of $C_\a$ if and only if $x_1^\a < y_q^\a$.
	Similarly, the vertical walls that intersect $\ell_\a$ below $\a$ are those of the form $\{ x_i = y_j \}$ with $x_i^\a \lessdot y_j^\a$, $\{ y_q = 0 \}$, and $\big\{ x_p = \frac 12 \big\}$.
	Every coordinate $x_i$ or $y_j$ appears in exactly two walls. If these two walls are both vertical, then one intersects $\ell_\a$ above $\a$ and the other intersects $\ell_\a$ below $\a$.
	Therefore, the reflections in $S_\a^+$ (resp.\ $S_\a^-$) pairwise commute.
	
	The horizontal walls are among the following: $\{ x_i = x_{i+1} \}$ for $1 \leq i \leq p-1$; $\{y_j = y_{j+1}\}$ for $1 \leq j \leq q-1$.
	Notice that we left out the hyperplanes $\{x_1 = x_p - 1\}$ and $\{ y_1 = y_q + 1 \}$: although the corresponding reflections are in $[1,w]$, they cannot occur as walls of $C_\a$, by \eqref{eq:walls-C}.
	Order these hyperplanes as follows:
	\begin{align}
	\label{eq:horizontal-hyperplanes-C}
	\begin{split}
	& \{ x_1 = x_2 \}, \, \{ x_2 = x_3 \}, \, \dotsc, \, \{x_{p-1} = x_p \}, \\
	& \{ y_1 = y_2 \}, \, \{ y_2 = y_3 \}, \, \dotsc, \, \{y_{q-1} = y_q \}.
\end{split}
	\end{align}
	Let $w_+$ (resp.\ $w_-$) be the product of the reflections in $S_\a^+$ (resp.\ $S_\a^-$), and let $w_\h$ be the product of the reflections in $S_\h$, in the same relative order as in \eqref{eq:horizontal-hyperplanes-C}.
	Let $\hat w = w_+ w_\h w_-$.
	We want to prove that $\hat w = w$.
	
	If $p=0$ or $q=0$, the set of walls \eqref{eq:walls-C} can be written explicitly, and it is immediate to check that $\hat w = w$ (in the case $q=0$, the chamber $C_\a$ is precisely the one used to define $w$ in the first place).
	Suppose from now on that $p > 0$ and $q > 0$.
	We are going to show that the linear parts of $\hat w$ and $w$ coincide, and that $\hat w(\b) = w(\b)$ for some point $\b \in E$.
	
	Consider the unit vector $e_{x_i}$ in the direction of $x_i$, with $i \leq p-1$.
	\begin{itemize}
		\item If $x_i^\a \lessdot y_j^\a \lessdot y_{j-1}^\a \lessdot \dotsb \lessdot y_{j'}^\a \lessdot x_{i+1}^\a$ with $j' \leq j$, then: the linear part of $w_-$ sends $e_{x_i}$ to $e_{y_j}$; the linear part of $w_\h$ sends $e_{y_j}$ to $e_{y_{j'}}$; the linear part of $w_+$ sends $e_{y_{j'}}$ to $e_{x_{i+1}}$.
		\item If $x_i^\a \lessdot x_{i+1}^\a$, then: the linear part of $w_-$ fixes $e_{x_i}$; the linear part of $w_\h$ sends $e_{x_i}$ to $e_{x_{i+1}}$; the linear part of $w_+$ fixes $e_{x_{i+1}}$.
	\end{itemize}
	Consider now the unit vector $e_{x_p}$.
	\begin{itemize}
		\item If $x_p^\a \lessdot y_j^\a \lessdot y_{j-1}^\a \lessdot \dotsb \lessdot y_{1}^\a < \frac 12$ with $j \geq 1$, then: the linear part of $w_-$ sends $e_{x_p}$ to $e_{y_j}$; the linear part of $w_\h$ sends $e_{y_j}$ to $e_{y_{1}}$; the linear part of $w_+$ sends $e_{y_1}$ to $-e_{y_1}$.
		
		\item If $y_1^\a \lessdot x_i^\a \lessdot x_{i+1}^\a \lessdot \dotsb \lessdot x_p^\a < \frac 12$ with $i \leq p$, then: the linear part of $w_-$ sends $e_{x_p}$ to $-e_{x_p}$; the linear part of $w_\h$ sends $-e_{x_p}$ to $-e_{x_i}$; the linear part of $w_+$ sends $-e_{x_i}$ to $-e_{y_1}$.
	\end{itemize}
	A similar argument shows that the linear part of $\hat w$ sends $e_{y_j}$ to $e_{y_{j+1}}$ if $j \leq q-1$, and sends $e_{y_q}$ to $-e_{x_1}$.
	Recall now that $y_j = 1-x_{p+j}$, and so $e_{y_j} = -e_{x_{p+j}}$.
	Therefore the linear part of $\hat w$ sends $e_{x_i}$ to $e_{x_{i+1}}$ for all $i \in \{1,\dotsc, n\}$ (with indices taken modulo $n$), so it coincides with the linear part of $w$.
	
	It remains to show that $\hat w(\b) = w(\b)$ for some point $\b \in E$.
	\begin{itemize}
		\item If $0 < x_1^\a < y_q^\a$, let $\b$ be the vertex of $C_\a$ opposite to the wall $H = \{ x_1 = 0\}$, i.e.\ $\b = \big(\frac 12, \frac12, \dotsc, \frac 12 \big)$.
		The reflection $r$ with respect to $H$ belongs to $S_\a^+$.
		Then $\hat w(\b) = r(\b) = \big(-\frac12, \frac12, \dotsc, \frac12) = w(\b)$.
		
		\item If $0 < y_q^\a \lessdot y_{q-1}^\a \lessdot \dotsb \lessdot y_j^\a \lessdot x_1^\a$ with $j \leq q$, let $\b$ be the vertex of $C_\a$ opposite to the wall $H = \{ y_j = x_1 \} = \{ 1-x_{p+j} = x_1 \}$, i.e.
		\[ \b = \big( \underbrace{\textstyle\frac12, \dotsc, \frac12}_{p+j-1}, \underbrace{1, \dotsc, 1}_{q-j+1} \big). \]
		Again, the reflection $r$ with respect to $H$ belongs to $S_\a^+$, and therefore
		\[ \hat w(\b) = r(\b) = \big( 0, \underbrace{\textstyle \frac12, \dotsc, \frac 12}_{p+j-1}, \underbrace{1, \dotsc, 1}_{q-j} \big) = w(\b). \qedhere \]
	\end{itemize}
\end{proof}

As in the case $\tilde A_n$, in order to prove \Cref{lemma:hyperbolic-coxeter} we start by explicitly describing the hyperbolic elements $u \in [1,w]$ with $l(u) = n$.
By \Cref{thm:reflections-bipartite}, the horizontal reflections $r \in [1,w]$ are the ones with the following fixed hyperplanes:
\begin{enumerate}[(i)]
	\item $\{ x_i = x_{j}\}$ for $i < j$;
	\item $\{ x_i = x_{j} - 1\}$ for $i < j$.
\end{enumerate}

If $r$ is the reflection with respect to $\{ x_i = x_j \}$ with $i < j \leq n-1$, then $u = wr$ sends a point $\b \in E$ to
\[ (x_n^\b - 1, x_1^\b, \dotsc, x_{i-1}^\b, x_j^\b, x_{i+1}^\b, \dotsc, x_{j-1}^\b, x_i^\b, x_{j+1}^\b, \dotsc, x_{n-1}^\b). \]
On the coordinates $x_{i+1}, \dotsc, x_j$, the hyperbolic isometry $u$ acts as a horizontal Coxeter element of type $A_{j-i-1}$. On the remaining coordinates, it acts as a Coxeter element of type $\tilde C_{n-j+i}$.

If $r$ is the reflection with respect to $\{ x_i = x_n \}$, then $u=wr$ sends $\b \in E$ to
\[ (x_i^\b - 1, x_1^\b, \dotsc, x_{i-1}^\b, x_n^\b, x_{i+1}^\b, \dotsc, x_{n-1}^\b). \]
Therefore $u$ acts as a horizontal Coxeter element of type $A_{n-i-1}$ on the coordinates $x_{i+1}, \dotsc, x_{n}$, and as a Coxeter element of type $\tilde C_{i}$ on the coordinates $x_1, \dotsc, x_i$.

The situation is similar if $r$ is a reflection with respect to $\{ x_i = x_{j} - 1\}$ for some $i < j$.
In this case, $u$ acts as a Coxeter element of type $\tilde C_{j-i}$ on the coordinates $x_{i+1}, \dotsc, x_j$, and as a horizontal Coxeter element of type $A_{n-j+i-1}$ on the remaining coordinates.

Notice that, in some of the previous cases, a Coxeter element of type $\tilde C_1$ can occur (for instance, this happens if $r$ is the reflection with respect to $\{x_1 = x_n\}$).
The limit case $\tilde C_1$ still makes sense and coincides with $\tilde A_1$.

\begin{proof}[Proof of \Cref{lemma:hyperbolic-coxeter} for the case $\tilde C_n$]
	Let $u \in [1,w]$ be a hyperbolic isometry such that $W_u$ is irreducible.
	By the same argument as in the case $\tilde A_n$, we have that $u$ acts as a Coxeter element of type $\tilde C_{m}$ on a subset $x_{i_1}, \dotsc, x_{i_m}$ of the coordinates, and as the identity on the remaining coordinates.
	If we restrict to the relevant coordinates $x_{i_1}, \dotsc, x_{i_m}$, we have
	\[ u(\b) = (x_{i_m}^\b - 1, x_{i_1}^\b, \dotsc, x_{i_{m-1}}^\b). \]
	The relevant coordinates of $\a$ are given by
	\[ \a = \left( \textstyle\frac{i_1-1}{n}, \frac{i_2-1}{n}, \dotsc, \frac{i_m-1}{n} \right) - \theta\, \left(\textstyle \frac1n,\frac1n,\dotsc, \frac1n \right) \]
	for some $\theta \in \R$.
	These coordinates satisfy $x_{i_1}^\a < x_{i_2}^\a < \dotsb < x_{i_m}^\a < x_{i_1}^\a + 1$.
	We conclude by applying \Cref{lemma:product-C} to the Coxeter group $W_u$, its Coxeter element $u$, and the point $\a$.
\end{proof}

\subsection{Case \texorpdfstring{$\tilde B_n$}{Bn tilde}}

Let $W$ be a Coxeter group of type $\tilde B_n$.
It is realized as the reflection group in $E = \R^n$ associated with the hyperplane arrangement
\[ \A = \Big\{ \{x_i \pm x_j = k \} \Bigm|  1 \leq i < j \leq n, \; k \in \Z \Big\}  \cup \Big\{ \left\{x_i = \textstyle k \right\} \Bigm|  1 \leq i \leq n, \; k \in \Z \Big\}. \]

Consider the chamber $C_0 = \left\{ 0 < x_1 < x_2 < \dotsb < x_n, \; x_{n-1} + x_n < 1 \right\}$, with walls given by $\{x_1 = 0\}, \{x_1 = x_2\}, \dotsc, \{x_{n-1} = x_n\}, \{ x_{n-1} + x_n = 1 \}$.
Let $w$ be the Coxeter element obtained by multiplying the reflections with respect to these walls.
Then $w$ acts as follows on a point $\b \in E$:
\begin{equation}
\label{eq:coxeter-B}
w(\b) = (x_{n-1}^\b - 1, x_1^\b, \dotsc, x_{n-2}^\b \mid 1 - x_n^\b).
\end{equation}
The shortest vector in $\Mov(w)$ is $\mu = -\left( \frac 1{n-1}, \frac 1{n-1}, \dotsc, \frac 1{n-1} \bigm| 0 \right)$, and the points of the Coxeter axis $\ell$ are of the form
\[ \left( \textstyle 0, \frac 1{n-1}, \frac 2{n-1}, \dotsc, \frac{n-2}{n-1} \bigm| \frac12 \right) + \theta \mu \]
for $\theta \in \R$.
The horizontal hyperplanes are: $\{ x_i - x_j = k\}$ for $1 \leq i < j \leq n-1$ and $k \in \Z$; $\{ x_n = k\}$ for $k \in \Z$.

Given a point $\a \in E$ which is not contained in any hyperplane of $\A$, define $C_\a$, $S_\a$, $S_\a^+$, $S_\a^-$, and $S_\a^\h$ as in the previous cases.

\begin{lemma}
	\label{lemma:product-B}
	Let $W$ be a Coxeter group of type $\tilde B_n$, and $w$ a Coxeter element as in \eqref{eq:coxeter-B}.
	Let $\a \in E$ be a point which is not contained in any hyperplane of $\A$, and such that $x_1^\a < x_2^\a < \dotsb < x_{n-1}^\a < x_1^\a +1$ and $x_n^\a= \frac12$.
	Then the reflections in $S_\a^+$ (resp.\ $S_\a^-$) pairwise commute.
	In addition, $w$ can be written as a product of the reflections in $S_\a$, where the reflections in $S_\a^+$ come first, and the reflections in $S_\a^-$ come last.
\end{lemma}

\begin{proof}
	By replacing $\a$ with a $w^m(\a)$ for a suitable $m \in \Z$, we can assume that $0 < x_1^\a < x_2^\a < \dotsb < x_{n-1}^\a < 1$ and $x_n^\a = \frac12$.
	Let $p \in \{0, \dotsc, n-1\}$ be the largest index such that $x_p^\a < \frac12$, and let $q = n-1-p$.
	As in the case $\tilde C_n$, define $y_j = 1 - x_{p+j}$ for $j \in \{1, \dotsc, q\}$.
	Define also
	\[ t = \begin{cases}
	x_n & \text{if $q$ is even} \\
	1-x_n & \text{if $q$ is odd}.
	\end{cases} \]
	Notice that, if we multiply the reflections with respect to $\{ x_i = x_n \}$ and $\{ x_i + x_n = 1 \}$ (for some $i \leq n-1$), we obtain the isometry $(x_i, x_n) \mapsto (1-x_i, 1-x_n)$.
	If we multiply these isometries for all $i \in \{p+1, \dotsc, n-1\}$, we obtain the change of coordinates
	\[ (x_1, \dotsc, x_{n-1} \mid x_n) \mapsto (x_1, \dotsc, x_p, y_1, \dotsc, y_q \mid t), \]
	which is therefore an element of $W$.
	As in the case $\tilde C_n$, we now have $0 < x_1^\a < x_2^\a < \dotsb < x_p^\a < \frac 12$ and $0 < y_q^\a < y_{q-1}^\a < \dotsb < y_1^\a < \frac 12$.
	In addition, there is the last coordinate $t^\a = \frac12$.
	
	Let $Z = \{ x_1^\a, \dotsc, x_p^\a, y_1^\a, \dotsc, y_q^\a \}$, and write $Z = \{ z_1^\a < z_2^\a < \dotsb < z_{n-1}^\a \}$.
	Using the coordinates $z_1, \dotsc, z_{n-1}, t$, the chamber $C_\a$ is given by $\{ 0 < z_1 < \dotsb < z_{n-1} < t, \; z_{n-1} + t < 1 \}$.
	Therefore, its walls are
	\[ \{ z_1 = 0 \}, \{ z_1 = z_2 \}, \dotsc, \{z_{n-1} = t\}, \{ z_{n-1} + t = 1\}. \]
	
	Denote by $r$ and $r'$ the reflections with respect to $\{ z_{n-1} = t \}$ and $\{z_{n-1} + t = 1\}$, respectively.
	They commute, and they are both vertical.
	In addition, they are either both in $S_\a^+$ (if $z_{n-1} = y_1$) or both in $S_\a^-$ (if $z_{n-1} = x_p$).
	The product $rr'$ is given by $(z_{n-1}, t) \mapsto (1-z_{n-1}, 1-t)$, and it is the identity on the other coordinates.
	Since $t$ is either $x_n$ or $1-x_n$, we have that $rr'$ is given by $(z_{n-1}, x_n) \mapsto (1-z_{n-1}, 1-x_n)$.
	On the last coordinate $x_n$, this is exactly how $w$ acts.
	On the coordinate $z_{n-1}$, we have that $rr'$ acts as a reflection with respect to $z_{n-1} = \frac12$.
	The rest of the proof carries out exactly as in the case $\tilde C_{n-1}$ (see \Cref{lemma:product-C}).
\end{proof}

Let us examine the hyperbolic elements $u \in [1,w]$ with $l(u) = n$.
The horizontal reflections $r \in [1,w]$ are:
\begin{enumerate}[(i)]
	\item $\{ x_i = x_{j}\}$ for $i < j \leq n-1$;
	\item $\{ x_i = x_{j} - 1\}$ for $i < j \leq n-1$;
	\item $\{ x_n = 0 \}$;
	\item $\{ x_n = 1 \}$.
\end{enumerate}

Similarly to the case $\tilde C_n$, if $r$ is a reflection of type (i), then $u = wr$ acts as a horizontal Coxeter element of type $A_{j-i-1}$ on some of the coordinates, and as a Coxeter element of type $\tilde B_{n-j+i}$ on the remaining coordinates.
If $r$ is a reflection of type (ii), then $u=wr$ acts as a Coxeter element of type $\tilde B_{j-i+1}$ on some of the coordinates, and as a horizontal Coxeter element of type $A_{n-j+i-2}$ on the remaining coordinates.
Notice that the limit case $\tilde B_2 = \tilde C_2$ can arise (for instance, if $r$ is the reflection with respect to $\{x_1 = x_{n-1}\}$).

If $r$ is the reflection with respect to $\{x_n = 0\}$, then $u = wr$ sends a point $\b \in E$ to
\[ u(\b) = (x_{n-1}^\b - 1, x_1^\b, \dotsc, x_{n-2}^\b \mid x_n^\b + 1). \]
This is a $(n-1,1)$-bigon Coxeter element of type $\tilde A_{n-1}$, and $W_u$ is a Coxeter group of type $\tilde A_{n-1}$.
If $r$ is the reflection with respect to $\{x_n = 1\}$ we obtain the same result up to a conjugation by $w$, so $u = wr$ is again a $(n-1,1)$-bigon Coxeter element of type $\tilde A_{n-1}$.

\begin{proof}[Proof of \Cref{lemma:hyperbolic-coxeter} for the case $\tilde B_n$]
	Let $u \in [1,w]$ be a hyperbolic element such that $W_u$ is irreducible.
	In particular, $u = w v^{-1}$ for some horizontal element $v \in [1,w]$.
	
	If $r \leq v$, where $r$ is a reflection with respect to $\{ x_n = 0\}$ or $\{x_n = 1\}$, then $u \leq wr$. Since $wr$ is a Coxeter element of type $\tilde A_{n-1}$, this case was already covered in Section \ref{sec:appendix-A}.
	
	Otherwise we proceed as for the case $\tilde C_n$, and notice that $u$ must act as a Coxeter element of type $\tilde B_m$ on a subset of the coordinates $x_{i_1}, \dotsc, x_{i_{m-1}}, x_n$.
	If we restrict to these relevant coordinates, we have
	\[ u(\b) = (x_{i_{m-1}}^\b - 1, x_{i_1}^\b, \dotsc, x_{i_{m-2}}^\b \mid 1 - x_n^\b). \]
	We conclude by applying \Cref{lemma:product-B} to the Coxeter group $W_u$, its Coxeter element $u$, and the point $\a$.
\end{proof}

\subsection{Case \texorpdfstring{$\tilde D_n$}{Dn tilde}}

Let $W$ be a Coxeter group of type $\tilde D_n$.
It is realized as the reflection group in $E = \R^n$ associated with the hyperplane arrangement
\[ \A = \Big\{ \{x_i \pm x_j = k \} \Bigm|  1 \leq i < j \leq n, \; k \in \Z \Big\}. \]

Consider the chamber $C_0 = \left\{ 0 < x_1 + x_2, \; x_1 < x_2 < \dotsb < x_n, \; x_{n-1} + x_n < 1 \right\}$, with walls given by $\{x_1 + x_2 = 0\}, \{x_1 = x_2\}, \dotsc, \{x_{n-1} = x_n\}, \{ x_{n-1} + x_n = 1 \}$.
Let $w$ be the Coxeter element obtained by multiplying the reflections with respect to these walls.
Then $w$ acts as follows on a point $\b \in E$:
\begin{equation}
\label{eq:coxeter-D}
w(\b) = (-x_1^\b \mid x_{n-1}^\b - 1, x_2^\b, x_3^\b, \dotsc, x_{n-2}^\b \mid 1 - x_n^\b).
\end{equation}
The shortest vector in $\Mov(w)$ is $\mu = -\left( 0 \bigm| \frac 1{n-2}, \frac 1{n-2}, \dotsc, \frac 1{n-2} \bigm| 0 \right)$, and the points of the Coxeter axis $\ell$ are of the form
\[ \left( 0 \bigm| \textstyle 0, \frac 1{n-2}, \frac 2{n-2}, \dotsc, \frac{n-3}{n-2} \bigm| \frac12 \right) + \theta\mu \]
for $\theta \in \R$.
The horizontal hyperplanes are: $\{ x_i - x_j = k\}$ for $2 \leq i < j \leq n-1$ and $k \in \Z$; $\{ x_1 \pm x_n = k\}$ for $k \in \Z$.

Given a point $\a \in E$ which is not contained in any hyperplane of $\A$, define $C_\a$, $S_\a$, $S_\a^+$, $S_\a^-$, and $S_\a^\h$ as in the previous cases.

\begin{lemma}
	\label{lemma:product-D}
	Let $W$ be a Coxeter group of type $\tilde D_n$, and $w$ a Coxeter element as in \eqref{eq:coxeter-D}.
	Let $\a \in E$ be a point which is not contained in any hyperplane of $\A$, and such that $x_1^\a= 0$, $x_2^\a < x_3^\a < \dotsb < x_{n-1}^\a < x_2^\a +1$, and $x_n^\a= \frac12$.
	Then the reflections in $S_\a^+$ (resp.\ $S_\a^-$) pairwise commute.
	In addition, $w$ can be written as a product of the reflections in $S_\a$, where the reflections in $S_\a^+$ come first, and the reflections in $S_\a^-$ come last.
\end{lemma}

\begin{proof}
	By replacing $\a$ with a $w^m(\a)$ for a suitable $m \in \Z$, we can assume that $x_1^\a= 0$, $0 < x_2^\a < x_3^\a < \dotsb < x_{n-1}^\a < 1$, and $x_n^\a = \frac12$.
	Let $p \in \{1, \dotsc, n-1\}$ be the largest index such that $x_p^\a < \frac12$, and let $q = n-1-p$.
	Define $y_j = 1 - x_{p+j}$ for $j \in \{1, \dotsc, q\}$, and
	\[ t = \begin{cases}
	x_n & \text{if $q$ is even} \\
	1-x_n & \text{if $q$ is odd}.
	\end{cases} \]
	As in the case $\tilde B_n$, the change of coordinates
	\[ (x_1 \mid x_2, \dotsc, x_{n-1} \mid x_n) \mapsto (x_1 \mid x_2, \dotsc, x_p, y_1, \dotsc, y_q \mid t) \]
	is an element of $W$.

	We now have $0 = x_1^\a < x_2^\a < \dotsb < x_p^\a < t^\a = \frac 12$ and $0 = x_1^\a < y_q^\a < y_{q-1}^\a < \dotsb < y_1^\a < t^\a= \frac 12$.
	Let $Z = \{ x_2^\a, \dotsc, x_p^\a, y_1^\a, \dotsc, y_q^\a \}$, and write $Z = \{ z_2^\a < z_3^\a < \dotsb < z_{n-1}^\a \}$.
	Using the coordinates $x_1, z_2, \dotsc, z_{n-1}, t$, the chamber $C_\a$ is given by $\{ 0 < x_1 + z_2, \; x_1 < z_2 < \dotsb < z_{n-1} < t, \; z_{n-1} + t < 1 \}$.
	Therefore its walls are
	\[ \{ x_1 + z_2 = 0 \}, \{ x_1 = z_2 \}, \{z_2 = z_3\}, \dotsc, \{z_{n-1} = t\}, \{ z_{n-1} + t = 1\}. \]
	
	Exactly as in the case $\tilde B_n$, the reflections with respect to $\{ z_{n-1} = t \}$ and $\{z_{n-1} + t = 1\}$ commute, and they are either both in $S_\a^+$ or both in $S_\a^-$.
	Their product acts as $(z_{n-1}, x_n) \mapsto (1-z_{n-1}, 1-x_n)$.
	Similarly, the reflections with respect to $\{ x_1 + z_2 = 0 \}$ and $\{ x_1 = z_2 \}$ commute, and they are either both in $S_\a^+$ or both in $S_\a^-$.
	Their product acts as $(x_1, z_2) \mapsto (-x_1, -z_2)$.
	On the coordinate $x_1$, this is exactly how $w$ acts.
	On the coordinate $z_2$, this is the same as a reflection with respect to $z_2 = 0$.
	We conclude as in the case $\tilde B_n$.
\end{proof}

We now examine the hyperbolic elements $u \in [1,w]$ with $l(u) = n$.
The horizontal reflections $r \in [1,w]$ are:
\begin{enumerate}[(i)]
	\item $\{ x_i = x_{j}\}$ for $2 \leq i < j \leq n-1$;
	\item $\{ x_i = x_{j} - 1\}$ for $2 \leq i < j \leq n-1$;
	\item $\{ x_1 \pm x_n = 0 \}$;
	\item $\{ x_n \pm x_1 = 1 \}$.
\end{enumerate}

Similarly to the previous cases, if $r$ is a reflection of type (i), then $u = wr$ acts as a horizontal Coxeter element of type $A_{j-i-1}$ on some of the coordinates, and as a Coxeter element of type $\tilde D_{n-j+i}$ on the remaining coordinates.
If $r$ is a reflection of type (ii), then $u=wr$ acts as a Coxeter element of type $\tilde D_{j-i+2}$ on some of the coordinates, and as a horizontal Coxeter element of type $A_{n-j+i-3}$ on the remaining coordinates.
Notice that the limit case $\tilde D_3 = \tilde A_3$ can arise (for instance, if $r$ is the reflection with respect to $\{x_2 = x_{n-1}\}$). When this happens, a $(2,2)$-bigon Coxeter element is obtained.

If $r$ is the reflection with respect to $\{x_1 + x_n = 0\}$, then $u = wr$ sends a point $\b \in E$ to
\[ u(\b) = (x_n^\b \mid x_{n-1}^\b - 1, x_2^\b, x_3^\b, \dotsc, x_{n-2}^\b \mid x_1^\b + 1). \]
This is a $(n-2,2)$-bigon Coxeter element of type $\tilde A_{n-1}$, and $W_u$ is a Coxeter group of type $\tilde A_{n-1}$.
If $r$ is the reflection with respect to $\{x_1 + x_n = 1\}$ we obtain the same result up to a conjugation by $w$, so $u = wr$ is again a $(n-2,2)$-bigon Coxeter element of type $\tilde A_{n-1}$.

If $r$ is the reflection with respect to $\{x_n - x_1 = 0\}$, then $u = wr$ sends a point $\b \in E$ to
\[ u(\b) = (-x_n^\b \mid x_{n-1}^\b - 1, x_2^\b, x_3^\b, \dotsc, x_{n-2}^\b \mid 1- x_1^\b). \]
If we set $x_1' = -x_1$, using the coordinates $(x_1', x_2, \dotsc, x_n)$ we get
\[ u(\b) = (x_n^\b \mid x_{n-1}^\b - 1, x_2^\b, x_3^\b, \dotsc, x_{n-2}^\b \mid x_1'^{\,\b} + 1), \]
which is now clearly recognizable as a $(n-2,2)$-bigon Coxeter element of type $\tilde A_{n-1}$.
If $r$ is the reflection with respect to $\{x_n - x_1 = 1\}$ we obtain the same result up to a conjugation by $w$, so $u = wr$ is again a $(n-2,2)$-bigon Coxeter element of type $\tilde A_{n-1}$.

\begin{proof}[Proof of \Cref{lemma:hyperbolic-coxeter} for the case $\tilde D_n$]
	Let $u \in [1,w]$ be a hyperbolic element such that $W_u$ is irreducible.
	In particular, $u = w v^{-1}$ for some horizontal element $v \in [1,w]$.
	
	If $r \leq v$, where $r$ is a reflection with respect to $\{ x_1 \pm x_n = 0\}$ or $\{x_n \pm x_1 = 1\}$, then $u \leq wr$. Since $wr$ is a Coxeter element of type $\tilde A_{n-1}$, this case was already covered in Section \ref{sec:appendix-A}.
	
	Otherwise we proceed as for the cases $\tilde C_n$ and $\tilde B_n$, by applying \Cref{lemma:product-D} to the Coxeter group $W_u$, its Coxeter element $u$, and the point $\a$.
\end{proof}

\bibliographystyle{amsalpha-abbr}
\bibliography{bibliography}

\end{document}